\documentclass[12pt]{amsart}
\pdfoutput=1
\usepackage{amsmath}
\usepackage{amstext}
\usepackage{amsfonts}
\usepackage{amssymb}
\usepackage{amsthm}
\usepackage{amsrefs}
\usepackage{color}
\usepackage{enumitem}
\usepackage{hyperref}
\usepackage{mathtools}
\usepackage{microtype}
\usepackage{thmtools}
\usepackage{tikz-cd}
\allowdisplaybreaks[3]

\numberwithin{equation}{subsection}
\mathtoolsset{centercolon}

\definecolor{darkblue}{rgb}{0.0,0.0,0.3}
\hypersetup{colorlinks,breaklinks,linkcolor=red,urlcolor=red,anchorcolor=red,citecolor=red}

\theoremstyle{plain}
\newtheorem{thm}{Theorem}[subsection]

\newtheorem{cor}[thm]{Corollary}
\newtheorem{prop}[thm]{Proposition}
\newtheorem{lem}[thm]{Lemma}

\theoremstyle{definition}
\newtheorem{defn}[thm]{Definition}
\newtheorem{example}[thm]{Example}
\newtheorem{rem}[thm]{Remark}

\newcommand{\bB}{{\mathbb{B}}}
\newcommand{\bC}{{\mathbb{C}}}
\newcommand{\bD}{{\mathbb{D}}}
\newcommand{\bF}{{\mathbb{F}}}

\newcommand{\bN}{{\mathbb{N}}}

\newcommand{\bR}{{\mathbb{R}}}
\newcommand{\bT}{{\mathbb{T}}}

\newcommand{\B}{{\mathcal{B}}}
\newcommand{\C}{{\mathcal{C}}}

\renewcommand{\H}{{\mathcal{H}}}

\newcommand{\I}{{\mathcal{I}}}

\newcommand{\K}{{\mathcal{K}}}  
\renewcommand{\L}{{\mathcal{L}}}
\newcommand{\M}{{\mathcal{M}}}
\newcommand{\cM}{{\mathcal{M}}}

\renewcommand{\O}{{\mathcal{O}}}

\newcommand{\U}{{\mathcal{U}}}

\newcommand{\fA}{{\mathfrak{A}}}
\newcommand{\fB}{{\mathfrak{B}}}

\newcommand{\rA}{{\mathrm{A}}}
\newcommand{\rB}{{\mathrm{B}}}
\newcommand{\rBI}{{\mathrm{B^\infty}}}
\newcommand{\rBM}{{\mathrm{BM}}}
\newcommand{\rC}{{\mathrm{C}}}
\newcommand{\rP}{{\mathrm{P}}}
\newcommand{\rU}{{\mathrm{U}}}

\newcommand{\ep}{\varepsilon}
\renewcommand{\phi}{\varphi}

\newcommand{\ca}{\mathrm{C}^*}
\newcommand{\conv}{\operatorname{conv}}
\newcommand{\ncconv}{\operatorname{ncconv}}

\newcommand{\id}{\operatorname{id}}

\newcommand{\ip}[1]{\langle #1 \rangle}

\newcommand{\mt}{\emptyset}
\newcommand{\ol}{\overline}
\newcommand{\one}{\mathbf{1}}
\newcommand{\Ran}{\operatorname{Ran}}

\newcommand{\Rep}{\operatorname{Rep}}
\newcommand{\spn}{\operatorname{span}}

\newenvironment{sbmatrix}{\left[\begin{smallmatrix}}{\end{smallmatrix}\right]}

\newcommand{\qand}{\quad\text{and}\quad}

\newcommand{\qfor}{\quad\text{for}\quad}
\newcommand{\qforal}{\quad\text{for all}\quad}
\newcommand{\qif}{\quad\text{if}\quad}

\newcommand{\FOR}{\ \text{for}\ }

\DeclareMathOperator{\graph}{Graph}
\DeclareMathOperator{\epi}{Epi}

\DeclareMathOperator{\re}{Re}

\DeclareMathOperator{\cbmaps}{\text{CB}}

\DeclareMathOperator{\ncpmaps}{CP_{nor}}
\DeclareMathOperator{\ucpmaps}{\text{UCP}}
\DeclareMathOperator{\bor}{Bor}
\DeclareMathOperator*{\wotlim}{\textsc{wot}--lim}

\newcommand{\env}[1]{\overline{#1}}
\newcommand{\cmin}{\mathrm{C}_{\textup{min}}^*}
\newcommand{\cmax}{\mathrm{C}_{\textup{max}}^*}

\makeatletter
\def\widebreve#1{\mathop{\vbox{\m@th\ialign{##\crcr\noalign{\kern\p@}%
  \brevefill\crcr\noalign{\kern0.1\p@\nointerlineskip}%
  $\hfil\displaystyle{#1}\hfil$\crcr}}}\limits}

\def\brevefill{$\m@th \setbox\z@\hbox{}%
 \hfill\scalebox{1.1}{\rotatebox[origin=c]{90}{(}} \kern4pt $}
\makeatother

\begin{document}

\title[Noncommutative Choquet theory]{Noncommutative Choquet theory}

\author[K.R. Davidson]{Kenneth R. Davidson}
\address{Department of Pure Mathematics\\ University of Waterloo\\Waterloo, ON, N2L 3G1, Canada}
\email{krdavids@uwaterloo.ca}

\author[M. Kennedy]{Matthew Kennedy}
\address{Department of Pure Mathematics\\ University of Waterloo\\Waterloo, ON, N2L 3G1, Canada}
\email{matt.kennedy@uwaterloo.ca}

\begin{abstract}
We introduce a new and extensive theory of noncommutative convexity along with a corresponding theory of noncommutative functions. We establish noncommutative analogues of the fundamental results from classical convexity theory, and apply these ideas to develop a noncommutative Choquet theory that generalizes much of classical Choquet theory.

The central objects of interest in noncommutative convexity are noncommutative convex sets. The category of compact noncommutative sets is dual to the category of operator systems, and there is a robust notion of extreme point for a noncommutative convex set that is dual to Arveson's notion of boundary representation for an operator system.

We identify the C*-algebra of continuous noncommutative functions on a compact noncommutative convex set as the maximal C*-algebra of the operator system of continuous noncommutative affine functions on the set. In the noncommutative setting, unital completely positive maps on this C*-algebra play the role of representing measures in the classical setting.

The role of noncommutative convex functions is crucial to our theory, and this is a new notion in the theory of noncommutative functions.
The convex noncommutative functions determine an order on the set of unital completely positive maps that is analogous to the classical Choquet order on probability measures. We characterize this order in terms of the extensions and dilations of the maps, providing a powerful new perspective on the structure of completely positive maps on operator systems.

Finally, we establish a noncommutative generalization of the Choquet-Bishop-de Leeuw theorem asserting that every point in a compact noncommutative convex set has a representing map that is supported on the extreme boundary. In the separable case, we obtain a corresponding integral representation theorem. 
\end{abstract}

\subjclass[2010]{Primary 46A55, 46L07, 47A20; Secondary 46L52, 47L25}
\keywords{noncommutative convexity, noncommutative Choquet theory, noncommutative functions, operator systems, completely positive maps}
\thanks{First author supported by NSERC Grant Number 2018-03973.}
\thanks{Second author supported by NSERC Grant Number 50503-10787.}
\maketitle
\clearpage

\setcounter{tocdepth}{2}
\tableofcontents

\section{Introduction} \label{S:introduction}

Classical Choquet theory is now a fundamental part of infinite-dimensional analysis. The integral representation theorem of Choquet-Bishop-de Leeuw, which has found numerous applications throughout mathematics, is undoubtedly the most well known result in the theory. It asserts that every point in a compact convex set can be represented by a probability measure supported on the extreme points of the set. However, this result is just one piece of classical Choquet theory, which is now a very powerful framework for the analysis of convex sets. 

Many objects in mathematics, especially in the theory of operator algebras, exhibit ``higher order'' convex structure. Various attempts have been made to capture this structure within an abstract framework, most notably in Wittstock's \cite{Wit1984} theory of matrix convexity. However, each of these frameworks suffers from the same serious issue: the non-existence of a suitable notion of extreme point.

In this paper we introduce a new theory of noncommutative convexity that we believe finally resolves this issue. The central objects of interest in the theory are noncommutative convex sets, for which there is a robust notion of extreme point. Working within this framework, we establish analogues of the fundamental results from classical convexity theory, along with a corresponding theory of noncommutative functions. We then apply these ideas to develop a corresponding noncommutative Choquet theory that generalizes much of classical Choquet theory. For example, we obtain a noncommutative generalization of the Choquet-Bishop-de Leeuw integral representation theorem for points in compact noncommutative convex sets.

An nc (noncommutative) convex set over an operator space $E$ is a graded set $K = \coprod_n K_n$, with each $K_n$ consisting of $n \times n$ matrices over $E$. The graded components $K_n$ are related by requiring that $K$ be closed under direct sums and compressions by isometries. The union is taken over all cardinal numbers $n \leq \kappa$, where $\kappa$ is a fixed infinite cardinal number depending on $E$. The fact that $n$ is permitted to be infinite here is an essential part of the theory, since even if $K$ is completely determined by its finite dimensional part, the finite part $\coprod_{n \in \bN} K_n$ of $K$ may not contain any extreme points at all.

For example, if $A$ is a separable unital C*-algebra, then the nc state space $K$ of $A$ is a (compact) nc convex set over $A^*$ defined by $K = \coprod_{n \leq \aleph_0} K_n$ with $K_n = \{ \phi : A \to \B(H_n) \text{ unital and completely positive} \}$, where $H_n$ is a fixed Hilbert space of dimension $n$ and $\B(H_n)$ denotes the C*-algebra of bounded operators on $H_n$. The extreme points $\partial K$ of $K$ are precisely the irreducible representations of $A$, and they completely determine $K$ in the sense that every point in $K$ is a limit of nc convex combinations of points in $\partial K$. Yet if $A$ is simple and infinite dimensional, e.g., if $A$ is the Cuntz algebra $\O_2$, then it has no finite dimensional representations. So in this case $\partial K$ has empty intersection with the finite part of $K$. 

This marks the key point of divergence from the theory of matrix convexity which, on the surface, resembles the theory of noncommutative convexity, but does not allow points corresponding to infinite matrices. As the previous example demonstrates, it is for precisely this reason that there is no suitable notion of extreme point in the matrix convex setting. We will see that this results in major differences between the theory of noncommutative convexity and the theory of matrix convexity.

The fundamental idea underlying classical Choquet theory is the dual equivalence between the category of compact convex sets and the category of function systems. The functor implementing this duality maps a compact convex set $C$ to the corresponding function system $\rA(C)$ of continuous affine functions on $C$, while the inverse functor maps a function system to its state space. This result is Kadison's \cite{Kad1951} representation theorem.

An analogous result holds in the noncommutative setting. The category of compact nc convex sets is dually equivalent to the category of operator systems, which are closed unital self-adjoint subspaces of C*-algebras. The functor implementing this duality maps a compact nc convex set $K$ to the corresponding operator system $\rA(K)$ of continuous affine nc functions on $K$. The inverse functor maps an operator system $S$ to its noncommutative state space $K = \coprod K_n$, where as above, $K_n = \{\phi : S \to \B(H_n) \text{ unital and completely positive} \}$. In particular, $S$ is completely order isomorphic to the operator system $\rA(K)$ of continuous affine nc functions on $K$, providing a noncommuative analogue of Kadison's representation theorem. A similar result was obtained in the matrix convex setting by Webster and Winkler \cite{WebWin1999}*{Proposition~3.5}.

For a compact convex set $C$, the C*-algebra $\rC(C)$ of continuous functions on $C$ is generated by the function system $\rA(C)$. A probability measure $\mu$ on $C$ is said to represent a point $x \in C$ and $x$ is said to be the barycenter of $\mu$ if the restriction $\mu|_{\rA(C)}$ satisfies $\mu|_{\rA(C)} = x$. Since the point mass $\delta_x$ represents $x$, every point in $C$ has at least one representing measure. The points $x \in C$ for which $\delta_x$ is the unique representing measure are precisely the extreme points of $C$. This interplay between the function system $\rA(C)$ and the C*-algebra $\rC(C)$ plays an essential role in classical Choquet theory.

Something similar is true in the noncommutative setting, and it is here that major differences begin to appear between the theory of noncommutative convexity and the theory of matrix convexity.

For a compact nc convex set $K$, we introduce a notion of nc function on $K$. The space $\rB(K)$ of all bounded nc functions is a von Neumann algebra. The space $\rC(K)$ of continuous nc functions on $K$ is a C*-subalgebra  of $\rB(K)$ which is generated by the space $\rA(K)$ of continuous affine nc functions on $K$. By analogy with the Takesaki--Bichteler's noncommutative Gelfand theorem \cites{Tak1967,Bic1969}, we show that $\rB(K) = \rC(K)^{**}$ is the universal von Neumann algebra of $\rC(K)$, and clarify the sense in which elements of $\rC(K)$ are continuous. Finally, we identify $\rC(K)$ with the maximal C*-algebra, $\cmax(\rA(K))$, of $\rA(K)$ introduced by Kirchberg and Wasserman \cite{KirWas1998}.

Motivated by the classical setting, we say that a unital completely positive map $\mu : \rC(K) \to \B(H_n)$ represents a point $x \in K_n$, and that $x$ is the barycenter of $\mu$, if the restriction $\mu|_{\rA(K)}$ satisfies $\mu|_{\rA(K)} = x$. The corresponding point evaluation $\delta_x : \rC(K) \to \B(H_n)$ represents $x$, so every point in $K$ has at least one representing map. As in the classical setting, the points in $x \in K$ for which $\delta_x$ is both irreducible and the unique representing map for $x$ are precisely the extreme points of $K$. 

In fact, this characterization of the extreme points of a compact nc convex set $K$ implies that they are dual to the boundary representations of the operator system $\rA(K)$ in the sense of Arveson \cite{Arv1969}. Hence viewed from the perspective of noncommutative convexity, Arveson's conjecture about the existence of boundary representations is equivalent to the existence of extreme points in compact nc convex sets. This conjecture was resolved only recently, by Arveson himself \cite{Arv2008} in the separable case, and by the authors \cite{DK2015} in complete generality. As an application of the ideas in this paper, we obtain a new proof of this result that is conceptually much different. 

We establish a noncommutative Krein-Milman theorem asserting that a compact nc convex set is the closed nc convex hull of its extreme points, as well as an analogue of Milman's partial converse to the Krein-Milman theorem. In the matrix convex setting, Webster and Winkler \cite{WebWin1999} obtained variants of these results for ``matrix extreme points.'' However, we will see that even in the special case that a matrix convex set is generated by points that are extreme in the sense of noncommutative convexity, there are generally many more matrix extreme points, meaning that our results are much stronger.

A key technical tool in classical Choquet theory is the notion of convex envelope of a  continuous real-valuedfunction. For a compact convex set $C$ and a real-valued continuous function $f \in \rC(C)$, the convex envelope $\bar{f}$ of $f$ is defined by $\bar{f} = \sup \{a  \in \rA(C)_{sa} : a \leq f \}$. It is the best approximation of $f$ from below by a real-valued lower semicontinuous convex function. In particular, $\bar{f} = f$ if and only if $f$ is convex and lower semicontinuous. 

In the noncommutative setting, we introduce a notion of convex nc function along with a corresponding notion of convex envelope of a self-adjoint continuous nc function. As in the classical setting, the convex envelope is a key technical tool. For a compact nc convex set $K$, the convex envelope $\bar{f}$ of an nc function $f \in \rC(K)$ is the best approximation from below by a lower semicontinuous convex nc function. However, since $\rC(K)$ is generally not a lattice, $\bar{f}$ is necessarily a multivalued function, and this introduces some technical difficulties. It is a non-trivial theorem that continuous convex nc functions can be approximated from below by the continuous affine nc functions that they dominate.

For example, if $I \subseteq \bR$ is a compact interval and $K = \coprod_{n \leq \aleph_0} K_n$ is the compact nc convex set defined by letting $K_n$ denote the set of self-adjoint operators in $\rB(H_n)$ with spectrum in $I$, then the convex nc functions on $K$ can be identified with the operator convex functions on $I$. We obtain a noncommutative analogue of Jensen's inequality that specializes in this case to the Hansen-Pedersen-Jensen inequality \cite{HanPed2003}.

For a compact convex set $C$, the classical Choquet order on the space of probability measures on $C$ is a generalization of the even more classical majorization order considered by Hardy, Littlewood and P\`{o}lya and others. For probability measures $\mu$ and $\nu$ on $C$, $\nu$ is said to dominate $\mu$ in the Choquet order if $\int_C f\, d\mu \leq \int_C f\, d\nu$ for every convex function $f \in \rC(C)$. A probability measure is maximal in the Choquet order precisely when it is supported on the extreme boundary $\partial C$ in an appropriate sense.

For a compact nc convex set $K$, we introduce two orders on the unital completely positive maps on $\rC(K)$. The nc Choquet order is analogous to the classical Choquet order. It is determined by comparing the values of the maps on the set of convex nc functions in $\rC(K)$. As in the classical case, a map is maximal in the nc Choquet order precisely when it is supported on the extreme boundary $\partial K$ in an appropriate sense. 

The nc dilation order, determined by comparing the set of dilations of the maps, has no classical counterpart. However, using  the theory of convex envelopes of convex nc functions, we show that it coincides with the nc Choquet order. This result has a number of interesting consequences. For example, we obtain an intrinsic characterization of unital completely positive maps on operator systems that have a unique completely positive extension to the C*-algebra generated by the operator system. A version of this order in the commutative setting was used in \cite{DK2021}.

The culmination of this paper is a noncommutative analogue of the integral representation theorem of Choquet-Bishop-de Leeuw \cites{Ch1956,BdL1959} We show that if $K$ is a compact nc convex set, then every point $x \in K$ has a representing map $\mu$ that is supported on the extreme boundary of $K$ in an appropriate sense. As in the classical setting, if $\rA(K)$ is non-separable, then the extreme boundary $\partial K$ of $K$ may not be a Borel set. In this case we show that if $f$ is nc function contained in the Baire-Pedersen enveloping C*-algebra of $\rC(K)$ that vanishes on the extreme points of $K$, then $\mu(f) = 0$. In the separable case, we obtain an integral representation theorem expressing a unital completely positive map on $\rC(K)$ as an integral against a unital completely positive map-valued probability measure supported on the extreme boundary of $K$.

\section*{Acknowledgements}

The authors thank Ben Passer, Adam Humeniuk, Gregory Patchell and Eli Shamovich for their feedback during the writing of this paper. They also thank Bojan Magajna and the working group at IMPAN run by Tatiana Shulman and Adam Skalski for a number of helpful comments and corrections. They also thank the referee for a careful reading and a number of useful comments.

\section{Noncommutative convex sets}

\subsection{Operator spaces, cardinality, dimension and topology} \label{sec:cardinality-dimension-topology}

We will work with operator spaces and operator systems throughout this paper. In this section, we briefly review some of the relevant technical details and introduce some notation and conventions. For detailed references on operator spaces we direct the reader to the books of Effros and Ruan \cite{ER2000} and Pisier \cite{P2003}. In particular, the details on infinite matrices over operator systems are contained in \cite{ER2000}*{Section 10.1}. For a detailed reference on operator systems, we direct the reader to the book of Paulsen \cite{Paulsen}. 

Let $E$ be an operator space. For nonzero cardinal numbers $m$ and $n$, we let $\cM_{m,n}(E)$ denote the operator space consisting of $m \times n$ matrices over $E$ with uniformly bounded finite submatrices. If $m=n$, then we let $\cM_n(E) = \cM_{n,n}(E)$. If $E = \bC$, then we let $\cM_{m,n} = \cM_{m,n}(\bC)$ and $\cM_n = \cM_{n,n}$.

For each $n$, we fix a Hilbert space $H_n$ of dimension $n$ and identify $\cM_{m,n}$ with the space $\B(H_n,H_m)$ of bounded operators from $H_n$ to $H_m$.

Let $m,n,p$ be nonzero cardinal numbers. For $x = [x_{ij}] \in \cM_n(E)$, $\alpha = [\alpha_{ij}] \in \cM_{m,n}$ and $\beta = [\beta_{ij}] \in \cM_{n,p}$, the products $\alpha x \in \cM_{m,n}(E)$, $x \beta \in \cM_{n,p}(E)$ and $\alpha x \beta \in \cM_{k,n}(E)$ can be defined as compositions under appropriate operator space embeddings. They can also be defined intrinsically by the formulas
\[
[\alpha x]_{ij} = \sum_{k} \alpha_{ik} x_{kj}, \quad [x \beta]_{ij} = \sum_{k} x_{ik} \beta_{kj},
\]
since the above series converge unconditionally, and
\[
[\alpha x \beta]_{ij} = \lim_F \sum_{k \in F} \alpha_{ik} (x\beta)_{kj} = \lim_F \sum_{k \in F} (\alpha x)_{ik} \beta_{kj},
\]
where the limits are taken over finite subsets $F$ of $n$.

We let $\cM(E) = \coprod_n \cM_n(E)$, where the disjoint union is taken over all nonzero cardinal numbers $n \leq \kappa$, where $\kappa$ is a sufficiently large cardinal number. If $E = \bC$, then we let $\cM = \cM(\bC)$. For a subset $X \subseteq \cM(E)$ and a nonzero cardinal $n \leq \kappa$, we will write $X_n = X \cap \cM_n(E)$.

The existence of an upper bound $\kappa$ is necessary to ensure that $\cM(E)$ is a set. However, it will be convenient to allow $\kappa$ to vary depending on the context. If we are considering finitely many operator spaces $E_1,\ldots,E_n$, then it will suffice to take $\kappa = \dim H$, where $H$ is a Hilbert space of minimal infinite dimension such that $E_1,\ldots,E_n$ embed completely isometrically into $\B(H)$. In particular, if $E_1,\ldots,E_n$ are separable, then it will suffice to take $\kappa = \aleph_0$. In practice, we will work with the understanding that $\kappa$ exists and simply write e.g. ``for all $n$'' instead of ``for all $n \leq \kappa$.'' 

If $E$ is a dual operator space with a distinguished predual $E_*$, then there is a natural operator space isomorphism
\[
\cM_n(E) \cong \cbmaps(E_*,\cM_n),
\]
where $\cbmaps(E_*,\cM_n)$ denotes the space of completely bounded maps from $E_*$ to $\cM_n$. The space $\cbmaps(E_*,\cM_n)$ is a dual operator space and the corresponding weak* topology is the point-weak* topology. We identify $\cM_n(E)$ and $\cbmaps(E_*,\cM_n)$ and equip $\cM_n(E)$ with the point-weak* topology. Note that this is the usual weak* topology on $\cM_n$.

Unless otherwise specified, the convergence of a net or a series in $\cM_n(E)$ will always be with respect to the point-weak* topology. For example, we will frequently use the fact that for any bounded family $\{x_i \in \cM_{n_i}(E)\}$ and any family $\{\alpha_i \in \cM_{n_i,n}\}$ satisfying $\sum \alpha_i^* \alpha_i = 1_n$, the sum $\sum \alpha_i^* x_i \alpha_i$ converges in $\cM_n(E)$.
 
If $E$ and $F$ are operator spaces, then the product $E \times F$ is an operator space. If $E$ and $F$ are dual operator spaces with distinguished preduals $E_*$ and $F_*$ respectively, then $E \times F = (E_* \times_1 F_*)^*$, so that $E \times F$ is a dual operator space with the distinguished predual $E_* \times_1 F_*$.

At various points throughout this paper we will review results from classical convexity theory and classical Choquet theory. In particular, we will discuss function systems, also known as Archimedean order unit spaces, which are classical precursors to operator systems. For a detailed reference on classical Choquet theory, we refer the reader to the books of Alfsen \cite{Alfsen}, Phelps \cite{Phelps} and Luke\v{s}-Mal\'{y}-Netuka-Spurn\'{y} \cite{LMNS}. For a modern perspective on function systems, we refer the reader to the paper of Paulsen and Tomforde \cite{PT2009}.

\subsection{Noncommutative convex sets}

\begin{defn} \label{defn:nc-convex-set}
An {\em nc convex set} over an operator space $E$ is a graded subset $K = \coprod K_n \subseteq \cM(E)$ that is closed under direct sums and compressions, meaning that
\begin{enumerate}

\item $\sum \alpha_i x_i \alpha_i^* \in K_n$ for every bounded family $\{x_i \in K_{n_i}\}$ and every family of isometries $\{\alpha_i \in \cM_{n,n_i}\}$ satisfying $\sum \alpha_i \alpha_i^* = 1_n$,

\item $\beta^* x \beta \in K_m$ for every $x \in K_n$ and every isometry $\beta \in \cM_{n,m}$.

\end{enumerate}
We will say that $K$ is {\em closed} if $E$ is a dual operator space and each $K_n$ is closed in the topology on $\cM_n(E)$. Similarly, we will say that $K$ is {\em compact} if each $K_n$ is compact in the topology on $\cM_n(E)$.
\end{defn}

\begin{rem}
Condition (1) is equivalent to the assertion that any unitary that conjugates $\oplus \cM_{n_i}$ into $\cM_n$ necessarily conjugates $\oplus K_{n_i}$ into $K_n$. So we say that (1) means that $K$ is closed under taking infinite direct sums. Condition (2) implies that $K$ is closed under compression to subspaces, and in particular that each $K_n$ is closed under unitary conjugation. Note that each $K_n$ is an (ordinary) convex set.

Infinite direct sums are necessary, as otherwise one could have $K_n = \mt$ for $n \ge \aleph_0$ but $K \ne \mt$.
However, if $K$ is closed, closed under compressions, finite direct sums and infinite ampliations of elements in $K_1$, then it is nc convex.
Indeed, suppose that $\{x_i \in \cM_{n_i}(E) : i \in \I \}$ and $n = \sum_\I n_i$.
Fix $x_0 \in K_1$. For $F \Subset \I$ a finite set, let $n_F = \sum_{i \not\in F} n_i$.
Thus $y_F := \sum_F^\oplus x_{n_i} \oplus (x_0 \otimes 1_{n_F}) \in K_n$, 
where we require the ampliation of $x_0$ to complete the term.
The net $(y_F)_{F \Subset \I}$ is a net in $K_n$ which converges weak-$*$ to $y := \sum_\I^\oplus x_{n_i}$.
Since $K_n$ is closed in the weak-$*$ topology, $y \in K_n$; whence $K$ is nc convex.
\end{rem}

\begin{rem}
As discussed in Section \ref{sec:cardinality-dimension-topology}, there is an infinite cardinal number $\kappa$ such that $K = \coprod_{n \leq \kappa} K_n$. However, it will be convenient to work with the understanding that $\kappa$ exists without necessarily mentioning it explicitly.
\end{rem}

\begin{example}
A simple example of a compact nc convex set is a compact operator interval. Fix $c,d \in \bR$ with $c < d$. For $n \in \bN$, let $K_n = [c1_n, d1_n]$, where
\[
[c1_n, d1_n] = \{ \alpha \in (\cM_n)_{sa} : c 1_n \leq \alpha \leq d 1_n \}. 
\]
Then $K = \coprod_{n \in \bN} K_n$ is a compact matrix convex set over $\bC$. It is not difficult to show that if $L = \coprod_{n \in \bN} L_n$ is a compact nc convex set with $L_1 = [c,d]$, then $L = K$.
\end{example}

\begin{example}
Let $E$ be a dual operator space. The space $\cM(E)$ is a closed nc convex set. For each $n$, let $\bB_n(E)$ denote the unit ball of $\cM_n(E)$ and let $\bB(E) = \coprod_n \bB_n(E)$, where the union is taken over cardinal numbers $n \leq \kappa$ for a sufficiently large infinite cardinal number $\kappa$ as discussed in Section \ref{sec:cardinality-dimension-topology}. Each $\bB_n(E)$ is compact in $\cM_n(E)$, so $\bB(E)$ is a compact nc convex set.
\end{example}

\begin{example} \label{ex:nc-state-space}
Let $S$ be an operator system, i.e. a unital self-adjoint subspace of a unital C*-algebra. The {\em nc state space} of $S$ is the nc convex set $K = \coprod_n K_n$, where $K_n = \ucpmaps(S,\cM_n)$, the set of all unital completely positive maps of $S$ into $\cM_n$. The union is taken over cardinal numbers $n \leq \kappa$ for a sufficiently large infinite cardinal number $\kappa$ as discussed in Section \ref{sec:cardinality-dimension-topology}. Recall that for each $n$, we have identified $\cM_n(S^*)$ with the space $\cbmaps(S,\cM_n)$. Hence the inclusion $K_n = \ucpmaps(S,\cM_n) \subseteq \cbmaps(S,\cM_n)$ implies the inclusion $K_n \subseteq \cM_n(S^*)$. Moreover, $K_n$ is compact in the point-weak* topology. So $K$ is a compact nc convex set over $S^*$. 
\end{example}

\begin{defn}
Let $E$ be a dual operator space. For a bounded family $\{x_i \in \cM_{n_i}(E) \}$ 
and a family $\{ \alpha_i \in \M_{n_i,n} \}$ satisfying $\sum \alpha_i^* \alpha_i = 1_n$, 
we will refer to the element $\sum \alpha_i^* x_i \alpha_i \in \cM_n$ as an {\em nc convex combination} of elements in $K$. 
We will say that a subset $K \subseteq \cM(E)$ is closed under nc convex combinations if every nc convex combination of elements in $K$ belongs to $K$.
\end{defn}

\begin{prop}\label{P:nc combinationss}
Let $E$ be a dual operator space. A subset $K \subseteq \cM(E)$ is an nc convex set if and only if it is closed under nc convex combinations.
\end{prop}

\begin{proof}
Since the expressions in conditions (1) and (2) in Definition \ref{defn:nc-convex-set}  are special cases of nc convex combinations, if $K$ is closed under nc convex combinations, then it is clearly an nc convex set.

Conversely, suppose that $K$ is an nc convex set and consider a bounded family $\{x_i \in K_{n_i}\}$ 
and a family $\{ \alpha_i \in \cM_{n_i,n} \}$ satisfying $\sum \alpha_i^* \alpha_i = 1_n$. 
Let $m = \sum n_i$ and let $\beta_i \in \cM_{m,n_i}$ be isometries such that $\sum \beta_i \beta_i^* = 1_m$. 
Let $\gamma = \sum \beta_i \alpha_i \in \cM_{m,n}$. Then $\gamma^*\gamma = \sum \alpha_i^* \alpha_i = 1_n$, so $\gamma$ is an isometry. 
By (1), $y = \sum \beta_i x_i \beta_i^* \in K_m$. Hence by (2),
\[ 
 x = \gamma^* y \gamma = \sum_i \alpha_i^* \beta_i^* \Big( \sum_j \beta_j x_j \beta_j^*\Big) \sum_k \beta_k \alpha_k  = \sum_i \alpha_i^* x_i \alpha_i  \in K_n.
\]
Here we used the fact that $\beta_i^* \beta_j = 0$ if $i \ne j$.
\end{proof}

%

The next result shows that while closed nc convex sets are closed on each level, they are also closed in a stronger sense.

\begin{prop} \label{prop:convergent-nets}
Let $K$ be a closed nc convex set over a dual operator space $E$. Suppose there is a net $\{ x_i \in K_{n_i} \}$ and a net of isometries $\{ \alpha_i \in \cM_{n,n_i} \}$ satisfying $\lim \alpha_i \alpha_i^* = 1_n$ such that $\lim \alpha_i x_i \alpha_i^* = x \in \cM_n(E)$. Then $x \in K_n$.
\end{prop}

\begin{proof}
Let $\{\beta_i  \in \cM_{n,m_i}\}$ be a net of isometries satisfying $\beta_i \beta_i^* = 1_n - \alpha_i \alpha_i^*$. Then $\lim \beta_i \beta_i^* = 0$. Fix $y \in K_1$ and let $z_i = \alpha_i x_i \alpha_i^* + \beta_i (y \otimes 1_{m_i}) \beta_i^*$. Then by (1) of Definition \ref{defn:nc-convex-set}, $z_i \in K_n$, and by construction, $\lim z_i = x$. Hence $x \in K_n$ since $K$ is closed.
\end{proof}

The next result shows that closed nc convex sets are completely determined by their finite levels.

\begin{prop} \label{prop:equal-finite-levels}
Let $K$ and $L$ be closed nc convex sets over an operator system $E$. If $K_n = L_n$ for $n < \infty$, then $K = L$. 
\end{prop}

\begin{proof}
For arbitrary $n$ and $x \in K_n$, choose a net of finite rank isometries $\{\alpha_i \in \cM_{n,n_i} \}$ such that $\lim \alpha_i \alpha_i^* = 1_n$ and let $x_i = \alpha_i^* x \alpha_i \in K_{n_i}$. Then $\lim \alpha_i x_i \alpha_i^* = x$. Since $K_{n_i} = L_{n_i}$ for each $i$, Proposition \ref{prop:convergent-nets} implies $x \in L$. Hence $K \subseteq L$. By symmetry, $K = L$. 
\end{proof}

\subsection{Matrix convexity} \label{sec:matrix-convexity}

In this section we briefly pause to discuss the relationship between the theory of noncommutative convexity and the theory of matrix convexity introduced by Wittstock \cite{Wit1981}. We will also briefly mention the theory of C*-convexity introduced by Hoppenwasser, Moore and Paulsen \cite{HopMooPau1981}.

At least on the surface, the definition of a matrix convex set is similar to the definition of an nc convex set. The key distinction is that matrix convex sets do not contain points corresponding to infinite matrices. Specifically, a matrix convex $M$ set over an operator space $E$ is a graded subset $M = \coprod_{n \in \bN} M_n \subseteq \cM(E)$. If $E$ is a dual operator space, then $M$ is said to be closed (resp. compact) if each $M_n$ is closed (resp. compact) in the topology on $\cM_n(E)$.

If $K$ is an nc convex set, then the finite part $K_f := \coprod_{n \in \bN} K_n$ is a matrix convex set. On the other hand, if $M$ is a closed matrix convex set, then Proposition \ref{prop:equal-finite-levels} implies that $M$ determines a unique closed nc convex set. In fact, we will obtain results in Section \ref{sec:categorical-duality} that imply the category of compact matrix convex sets is equivalent to the category of compact nc convex sets.

Nevertheless, we will see that there are major differences between the theory of noncommutative convexity and the theory of matrix convexity. This will become particularly apparent when we begin to develop noncommutative Choquet theory, where it will be essential to consider points corresponding to infinite matrices as first class objects.

There are two key reasons for this. First, beginning in Section \ref{sec:nc-functions}, a major part of the noncommutative theory will involve the study of functions on nc convex sets. We will see that, even for reasonably nice functions, the restriction to the finite part of the set will not necessarily completely determine the function.

Second, when we introduce the notion of an extreme point for an nc convex set in Section \ref{sec:extreme-points}, we will establish a noncommutative Krein-Milman theorem, along with an analogue of Milman's partial converse to the Krein-Milman theorem, showing that the set of extreme points in a compact nc convex set is a minimal generating set in a very strong sense. However, we will also see that the finite part of a compact nc convex set, even for simple examples, may not contain any extreme points at all.

To be more specific, in classical convexity theory, the set of extreme points in a compact convex set is a minimal generating set in a sense that can be made precise using the Krein-Milman theorem and Milman's partial converse to the Krein-Milman theorem. If $C$ is a compact convex set and $\partial C$ denotes the extreme points of $C$, then the Krein-Milman theorem asserts that $C$ is the closed convex hull of $\partial C$. If $D \subseteq C$ is a closed subset with the property that the closed convex hull of $D$ is $C$, then Milman's partial converse to the Krein-Milman theorem asserts that $\partial C \subseteq D$. This property of minimality underlies much of classical convexity theory, and is absolutely essential for the development of classical Choquet theory.

There is a notion of ``matrix extreme point'' in the theory of matrix convexity for which a Krein-Milman theorem holds. This was proved by Webster and Winkler \cite{WebWin1999} (see \cite{Far2004} for another proof), along with an analogue of Milman's partial converse to the Krein-Milman theorem. Their result extended a Krein-Milman theorem for C*-extreme points in the matrix state space of a C*-algebra proved earlier by Morenz \cite{Mor1994} and Farenick-Morenz \cite{FarMor1997}.

However, the set of matrix extreme points in a compact matrix convex set is generally not a minimal generating set in any meaningful sense. Simple examples show that if $M$ is a compact matrix convex set and $\partial M$ denotes the set of matrix extreme points in $M$, then it is possible for the closed convex hull of a much smaller subset of $\partial M$ to be equal to $M$. The main problem is that it is possible for the matrix convex hull of a single matrix extreme point in $M_n$ to contain matrix extreme points in $M_m$ for $m < n$. 

There have been attempts to work with a more restricted notion of extreme point in the matrix convex setting. For example, Kleski \cite{Kle2014} defined a notion of ``absolute extreme point'' for matrix convex sets, and proved a corresponding Krein-Milman theorem for state spaces of operator systems that can be represented on finite dimensional Hilbert space. More recently, Evert, Helton, Klep and McCullough \cite{EveHel2019} proved a similar result for a special class of compact matrix convex sets called real spectrahedra. 

In fact, we will see that these results are a special case of the noncommutative Krein-Milman theorem for extreme points in compact nc convex sets. In particular, the fact that the finite part of a compact nc convex set does not necessarily contain any extreme points implies that no general version of these results can hold within the framework of matrix convexity. Instead, it is necessary to work with the framework of noncommutative convexity.

\subsection{Noncommutative separation theorem}

In this section we prove a separation theorem for nc convex sets that extends the separation theorem for matrix convex sets of Effros and Winkler \cite{EffWin1997}.

Let $E$ and $F$ be operator spaces and let $\phi : E \to F$ be a linear map. We write $\phi_n$ for the induced map $\phi_n : \cM_n(E) \to \cM_n(F)$ defined by
\[
\phi_n([e_{ij}]) = [\phi(e_{ij})], \qfor [e_{ij}] \in \cM_n(E).
\]
If $E$ and $F$ are operator systems, then the adjoint $\phi^* : E \to F$ is defined by $\phi^*(e) = \phi(e^*)^*$. We say that $\phi$ is self-adjoint if $\phi = \phi^*$. If $\phi$ is self-adjoint then it maps self-adjoint elements to self-adjoint elements.

\begin{thm}[Noncommutative separation theorem] \label{thm:separation} \strut \\
Let $K$ be a closed nc convex set over a dual operator space $E$ with $0_E \in K$. Suppose there is $n$ and $y \in \cM_n(E)$ such that $y \notin K_n$. Then there is a normal completely bounded linear map $\phi : E \to \cM_n$ such that
\[
\re \phi_n(y) \not \leq 1_n \otimes 1_n \quad \text{but} \quad \re \phi_p(x) \leq 1_p \otimes 1_n
\]
for all $p$ and $x \in K_p$. Moreover, if $E$ is an operator system and $K \cup \{y\}$ consists of self-adjoint elements, then $\phi$ can be chosen self-adjoint.
\end{thm}

\begin{proof}
First suppose $n \in \bN$. Then by the Effros-Winkler separation theorem \cite{EffWin1997}*{Theorem 5.4}, there is a continuous linear map $\phi : E \to M_n$ such that $\re \phi_m(x) \leq 1_m \otimes 1_n$ for all $m \in \bN$ and $x \in K_p$ but $\re \phi_n(y) \not \leq 1_n \otimes 1_n$. For arbitrary $p$ and $x \in K_p$, it follows from above that for $m \in \bN$ and an isometry $\alpha \in \cM_{p,m}$, 
\[
\alpha^* \phi_p(x) \alpha = \phi_m(\alpha^* x \alpha) \leq 1_m \otimes 1_n.
\]
Hence $\phi_p(x) \leq 1_p \otimes 1_n$.

Since $\phi$ is continuous with respect to the weak* topology on $E$, it is bounded. Furthermore, since $n$ is finite, it follows from a result of Smith \cite{Smi1983}*{Theorem 2.10} that $\phi$ is completely bounded.

For infinite $n$, we can consider $y$ as the point-weak-$*$ limit of the net of finite dimensional compressions. 
If each of these compressions were in $K$, then arguing as in Proposition \ref{prop:convergent-nets}, this would imply that $y \in K$. 
Since this is not the case, there is $m \in \bN$ and a compression $z \in \cM_m$ such that $z \notin K_m$.
Applying the above construction to $z$, we obtain a map $\psi : E \to M_m$ such that $\re \psi_p(x) \leq 1_p \otimes 1_m$ 
for all $p$ and $x \in K_p$, but $\re \psi_m(z) \not \leq 1_m \otimes 1_m$. 
Then $\re \psi_n(y) \not \leq 1_n \otimes 1_m$ since $\psi_m(z)$ is a compression of $\psi_n(y)$. 
Hence we can take $\phi$ to be an infinite amplification of $\psi$. 

If $E$ is self-adjoint and $K \cup \{y\}$ consist of self-adjoint elements, then we can replace $\phi$ with $\frac{1}{2}(\phi + \phi^*)$.
\end{proof}

The next result follows immediately from Theorem \ref{thm:separation} by applying a translation.

\begin{cor} \label{cor:separation}
Let $K$ be a closed nc convex set over a dual operator space $E$. Suppose there is $n$ and $y \in \cM_n(E)$ such that $y \notin K_n$. Then there is a normal completely bounded linear map $\phi : E \to \cM_n$ and self-adjoint $\gamma \in \cM_n$ such that
\[
\re \phi_n(y) \not \leq \gamma \otimes 1_n \quad \text{but} \quad 
\re \phi_p(x) \leq \gamma \otimes 1_p
\]
for all $p$ and $x \in K_p$. Furthermore, if $E$ is an operator system and $K \cup \{y\}$ consists of self-adjoint elements, then $\phi$ can be chosen self-adjoint.
\end{cor}

\subsection{Noncommutative affine maps}

\begin{defn} \label{defn:nc-affine-map}
Let $K$ and $L$ be nc convex sets over operator spaces $E$ and $F$ respectively. We say that a map $\theta : K \to L$ is an {\em affine nc map} if it is graded, respects direct sums and is equivariant with respect to isometries, meaning that
\begin{enumerate}
\item $\theta(K_n) \subseteq L_n$ for all $n$,
\item $\theta(\sum \alpha_i x_i \alpha_i^*) = \sum \alpha_i \theta(x_i) \alpha_i^*$ for every bounded family $\{x_i \in K_{n_i}\}$ and every family of isometries $\{\alpha_i \in \cM_{n_i,n}\}$ satisfying $\sum \alpha_i \alpha_i^* = 1_n$,
\item $\theta(\alpha^* x \alpha) = \alpha^* \theta(x) \alpha$ for every $x \in K_m$ and every isometry $\alpha \in \cM_{m,n}$.
\end{enumerate}
We say that $\theta$ is {\em continuous} if the restriction $\theta|_{K_n}$ is continuous for every $n$, and we say that $\theta$ is {\em bounded} if $\|\theta\|_\infty < \infty$, where $\|\theta\|_\infty$ is the {\em uniform norm} defined by

 \[
\| \theta\|_\infty = \sup_{x \in K} \|\theta(x)\|.
\]
We say that $\theta$ is a {\em homeomorphism} and that $K$ and $L$ are {\em affinely homeomorphic} if $\theta$ has a continuous affine inverse. We let $\rA(K, L)$ denote the space of continuous affine nc maps from $K$ to $L$. We let $\rA(K) = \rA(K,\cM)$, and we refer to $\rA(K)$ as the space of {\em continuous affine nc functions} on $K$.
\end{defn}

\begin{rem}
Arguing as in Proposition~\ref{P:nc combinationss}, we see that continuous affine nc maps between closed nc convex sets respect nc convex combinations. Specifically, if $K$ and $L$ are closed nc convex sets and $\theta : K \to L$ is a continuous affine nc map, then
 \[
\theta \left( \sum \alpha_i^* x_i \alpha_i \right) = \sum \alpha_i^* \theta(x_i) \alpha_i.
\]
for a bounded family $\{x_i \in K_{n_i}\}$ and a family $\{ \alpha_i \in \cM_{n_i,n} \}$ satisfying $\sum \alpha_i^* \alpha_i = 1_n$.
\end{rem}

\begin{prop} \label{P:uniform bound}
Let $K$ and $L$ be compact nc convex sets and let $\theta : K \to L$ be a continuous affine nc map. Then $\theta$ is bounded with
\[
\|\theta\| = \|\theta|_{K_{\aleph_0}}\| = \sup_{n < \infty} \|\theta|_{K_n}\|.
\]
\end{prop}

\begin{proof}
For each $n$, since $K_n$ is compact and $\theta|_{K_n}$ is continuous and affine, $\|\theta|_{K_n}\| < \infty$. Moreover, it is clear that $\|\theta|_{K_n}\|$ is an increasing function of $n$. Hence $\sup_{m < \infty} \|\theta|_{K_m}\| \leq \|\theta|_{K_{\aleph_0}}\| \leq \|\theta\|$.

To obtain the reverse inequalities, we argue as in the proof of Proposition \ref{prop:convergent-nets}. Fix arbitrary $n$ and $x \in K_n$. Let $\{\alpha_i \in \cM_{n,n_i}\}$ be a net of finite rank isometries satisfying $\lim \alpha_i \alpha_i^* = 1_n$. Let $\{\beta_i \in \cM_{n,m_i}\}$ be a net of isometries satisfying $\beta_i \beta_i^* = 1_n - \alpha_i \alpha_i^*$. Then $\lim \beta_i \beta_i^* = 0$.

Let $x_i = \alpha_i^* x \alpha_i \in K_{n_i}$. Fix $y \in K_0$ and let $z_i = \alpha_i x_i \alpha_i^* + \beta_i (y \otimes 1_{m_i}) \beta_i^*$. Then by (1) of Definition \ref{defn:nc-convex-set}, $z_i \in K_n$, and from above, $\lim z_i = x$. Hence by (2) of Definition \ref{defn:nc-affine-map},
\[
\theta(x) = \lim \theta(z_i) = \lim \alpha_i \theta(x_i) \alpha_i^* + \beta_i \theta(y \otimes 1_{m_i}) \beta_i^* = \lim \alpha_i \theta(x_i) \alpha_i^*.
\]
Therefore, $\|\theta|_{K_n}\| \leq \sup_{m < \infty} \|\theta|_{K_m}\| \leq \|\theta|_{K_{\aleph_0}}\|$.
\end{proof}

We will now make the assumption that $K \subset \cM(E)$ where $E$ is a dual operator system, 
and that $K$ is compact with respect to the weak-$*$ topology induced from $\cM(E)$.

\begin{lem}\label{L:affine separate}
Let $K$ be a compact nc convex set. Then $\rA(K)$ contains $E_*$, and thus separates the points of $K$. 
\end{lem}

\begin{proof}
By the remarks in Section~\ref{sec:cardinality-dimension-topology}, the weak-$*$ topology corresponds to the point-weak-$*$
topology induced by $\cM(E_*)$ obtained by using the identification between $\cM_n(E)$ and $\cbmaps(E_*,\cM_n)$.
In particular, we identify $E_*$ with the affine nc function in $\rA(K)$ given by $\phi(x) := x(\phi)$ for $x \in K_n$.
Clearly $E_*$ separates points in $\cM_n(E)$, and hence in $K_n$ for all $n$.
\end{proof}

We will need to consider two natural topologies on $K$ induced by the functions in $\rA(K)$.

\begin{defn} \label{D:point-weak* topology}
The \emph{point-weak* topology} on $K$ is the weakest topology that makes every affine nc function $a \in \rA(K)$ continuous.
Since each $\cM_n$ is equipped with the weak-$*$ topology, this is the weakest topology that makes the maps $K_n \to \bC : x \to \ip{a(x) \xi, \eta}$ continuous on $K_n$ for all $a \in \rA(K)$, $n \le\kappa$ and $\xi, \eta \in H_n$. 

The \emph{point-strong topology} on $K$ is the weakest topology that makes the maps $K_n \to H_n : x \to a(x) \xi$ continuous on $K_n$ for all $a \in \rA(K)$, $n \le\kappa$ and $\xi \in H_n$. 
\end{defn}

\begin{rem} \label{rem:point-strong-point-ultrastrong}
Since $K$ is bounded,  the point-weak* and point-weak operator topologies will coincide. Similarly, the point-strong topology will coincide with the point-ultrastrong topology; and since $\rA(K)$ is self-adjoint, the point-SOT* and point-ultrastrong* topologies cooincide. 
The point-ultrastrong* topology will come up when we discuss the C*-algebra generated by $\rA(K)$.
\end{rem}

\begin{lem}\label{lem:weak*-point-weak*}
Let $K$ be a compact nc convex set over a dual operator space $E$. The topology on $K$ coincides with the point-weak* topology.
\end{lem}

\begin{proof}
Since $\rA(K)$ consists of continuous functions on $(K,\text{weak*})$ and the point-weak* topology is the weakest topology making these functions continuous, the identity map on $K$ is continuous from the weak* topology to the point-weak* topology. The continuous affine nc functions separate points of $K$ by Lemma~\ref{L:affine separate}, and thus the point-weak* topology is Hausdorff. Since $K$ is compact in the weak* topology, it follows that this map is a homeomorphism.
\end{proof}

We now make a useful observation about $\rA(K)$.

\begin{thm} \label{T:affine cnts on K1}
If $K$ is a compact nc convex set, then an affine nc function $a:K \to \M$ is continuous if it is continuous on $K_1$;
and it is positive if it is positive on $K_1$.
Moreover, $\|a\| \le 2 \|a|_{K_1}\|$. 
The restriction map $\rho: a \mapsto a|_{K_1}$ is a unital order isomorphism of $\rA(K)$ onto $A(K_1)$.
\end{thm}

\begin{proof}
By Lemma~\ref{lem:weak*-point-weak*}, $a$ is continuous on $K_n$ if $x \to \ip{a(x)\xi,\eta}$ is continuous for all $\xi, \eta \in H_n$,
and it is positive if $\ip{a(x)\xi,\xi} \ge 0$ for all $\xi \in H_n$.
We may suppose that $\|\xi\|=1$, and consider $\xi$ as an isometry from $\bC$ into $H_n$.
Then
\[
 \ip{a(x)\xi,\xi} = \xi^* a(x) \xi = a(\xi^* x \xi) .
\]
Since $\xi^* x \xi \in K_1$, the continuity of $a|_{K_1}$ ensures the continuity of $x \to \ip{a(x)\xi,\xi}$ on $K_n$,
and the positivity of $a|_{K_1}$ ensures the positivity of $a$.
The polarization identity shows that
\[
 \ip{a(x)\xi,\eta} = \frac14 \sum_{k=0}^3 i^k \ip{a(x)\xi + i^k \eta, \xi + i^k \eta}
\]
is continuous. Moreover, this identity shows that
\begin{align*}
 \|a(x)\| &= \sup_{\|\xi\|=\|\eta\|=1} | \ip{a(x)\xi,\eta}| \\&
 \le \frac14 \|a|_{K_1}\|  \sup_{\|\xi\|=\|\eta\|=1} \sum_{0 \le k \le 3} \| \xi + i^k \eta \|^2 \\&
 = \|a|_{K_1}\|  \sup_{\|\xi\|=\|\eta\|=1} \|\xi\|^2 + \|\eta\|^2
 = 2 \|a|_{K_1}\| . 
\end{align*}
In particular, $a$ is determined by $a|_{K_1}$.

The restriction map $\rho$ is bounded below, and hence the range is a closed subspace of $A(K_1)$.
The range contains $\bC\one + E_*$, where $\one$ is the constant function $\one(x)=1$.
Thus by \cite{Alfsen}*{Corollary I.1.5}, $\bC\one + E_*$ is dense in $A(K_1)$; and hence the range is the entire $A(K_1)$.
The map $\rho$ is evidently unital, positive and bijective.
The first paragraph shows that $\rho^{-1}$ is also positive.
Therefore $\rho$ is an order isomorphism.
\end{proof}

\begin{rem}
This inequality  is sharp. Let 
\[ K_n = \{ T \in M_n : W(T) \subset \ol{\bD} \} , \]
where $W(T) = \ol{ \{ \ip{T\xi,\xi} : \|\xi\|=1\}}$ is the numerical range of $T$.
Then the identity map $\id$ is affine and $\|\id|_{K_1}\|=1$ while $A = \begin{sbmatrix} 0&2\\0&0\end{sbmatrix} \in K_2$ and $\|A\|=2$.
\end{rem}

\section{Categorical duality} \label{sec:categorical-duality}

\subsection{Convex sets and function systems}

If $C$ is a compact convex set, then the space $\rA(C)$ of complex-valued continuous affine functions on $C$ is a function system, also referred to as an Archimedean order unit space in the literature. This means that it is an ordered complex $*$-vector space with a distinguished Archimedean order unit \cite{PT2009}. Specifically, the order on $\rA(C)$ is determined by the positive cone $\rA(C)^+$ consisting of positive continuous affine functions on $C$. The *-operation on $\rA(C)$ is defined by conjugation and the order unit on $\rA(C)$ is the constant function $1_{\rA(C)}$. The state space of $\rA(C)$, consisting of positive unital functionals in the dual of $\rA(C)$, is compact with respect to the weak* topology and affinely homeomorphic to $C$ via the evaluation map.

On the other hand, let $V$ be a (closed) function system with state space $C$ equipped with the weak* topology. For $v \in V$, the function $\hat{v} : C \to \bC$ defined by $\hat{v}(x) = x(v)$ is a continuous affine function on $C$. Kadison's representation theorem \cite{Kad1951} asserts that the unital map $V \to \rA(C) : v \to \hat{v}$ is an order isomorphism. Hence every function system is order isomorphic to a function system of continuous affine functions on a compact convex set. 

The above results can be conveniently expressed in the language of category theory. Let $\mathrm{Conv}$ denote the category of compact convex sets with morphisms consisting of continuous affine maps and let $\mathrm{FuncSys}$ denote the category of function systems with morphisms consisting of unital order homomorphisms. The above results are equivalent to the statement that $\mathrm{Conv}$ and $\mathrm{FuncSys}$ are dually equivalent via the contravariant functor $A : \mathrm{Conv} \to \mathrm{FuncSys}$.

For a compact convex set $C$, $\rA(C)$ is the function system of continuous affine functions on $C$ as above. If $D$ is a compact convex set and $\phi : C \to D$ is a continuous affine map, then the unital order homomorphism $A\phi : \rA(D) \to \rA(C)$ is defined by $A\phi(b)(x) = b(\phi(x))$ for $b \in \rA(D)$ and $x \in C$. 

The inverse functor $A^{-1} : \mathrm{FuncSys} \to \mathrm{Conv}$ is defined similarly. For a function system $V$ with state space $C$, $A^{-1}V = C$. If $W$ is a function system with state space $D$ and $\psi : W \to V$ is a unital order homomorphism, then the continuous affine map $A^{-1}\psi : C \to D$ is defined by $b((A^{-1}(\psi)(x)) = \psi(b)(x)$ for $b \in W$ and $x \in C$.

\subsection{Noncommutative convex sets and operator systems}

In this section we will show that the category of compact nc convex sets is dually equivalent to the category of operator systems.

The arguments in this section are similar to arguments of Webster and Winkler \cite{WebWin1999}*{Proposition 3.5}. They proved that the category of compact matrix convex sets is dually equivalent to the category of operator systems. This is not surprising in light of Proposition \ref{prop:equal-finite-levels}, which implies that a compact nc convex set is determined by its finite levels. Major differences between the theory of noncommutative convexity and the theory of matrix convexity will only begin to appear in the next section.

For a compact nc convex set $K$, the space $\rA(K)$ of continuous affine nc functions on $K$ is an operator system. This means that it is a matrix ordered complex $*$-vector space with a distinguished Archimedean matrix order unit \cite{ChoiEff1977}. To see this, it will be convenient for $n \in \bN$ to identify the space $\cM_n(\rA(K))$ with the space of continuous affine nc maps $\rA(K, \cM_n(\cM))$, so that elements in $\cM_n(\rA(K))$ can be viewed as functions taking values in $\cM_n(\cM)$.

For $a \in \cM_n(\rA(K))$, the {\em adjoint} $a^* \in \cM_n(\rA(K))$ is defined by $a^*(x) = a(x)^*$ for $x \in K$. We say that $a$ is {\em self-adjoint} if $a = a^*$. If $a$ is self-adjoint, then we say that it is {\em positive} and write $a \geq 0$ if $a(x) \geq 0$ for all $x \in K$. Letting $\cM_n(\rA(K))^+$ denote the positive elements in $\cM_n(\rA(K))$, the sequence of positive cones $(\cM_n(\rA(K))^+)_{n \in \bN}$ determines the matrix order on $\rA(K)$. Together, this gives $\rA(K)$ the structure of a matrix ordered $*$-vector space.

Since $K$ is compact, elements in $\cM_n(\rA(K))$ are bounded by Proposition~\ref{P:uniform bound}. This implies that the constant function $1_{\rA(K)} \in \rA(K)$ defined by $1_{\rA(K)}(x) = 1_n$ for $x \in K_n$ is an Archimedean matrix order unit for $\rA(K)$.

The operator system structure on $\rA(K)$ induces a matrix norm on $\rA(K)$, i.e. a norm on $\cM_n(\rA(K))$ for each $n \in \bN$. This norm agrees with the uniform norm on $\cM_n(\rA(K))$.

\begin{defn} \label{defn:os-continuous-affine-fcns}
Let $K$ be a compact nc convex set. The {\em operator system of continuous affine nc functions on $K$} is the space $\rA(K)$ equipped with the operator system structure defined above.
\end{defn}

\begin{thm}
Let $K$ be a compact nc convex set and let $L$ denote the nc state space of $\rA(K)$. Then $K$ and $L$ are affinely homeomorphic via the affine nc map $\theta : K \to L$ defined by
\[
\theta(x)(a) = a(x), \qfor x \in K.
\]
\end{thm}

\begin{proof}
The nc state space $L$ of $\rA(K)$ is a compact nc convex set over $\rA(K)^*$ as in Example \ref{ex:nc-state-space}. It is clear that $\theta$ is a continuous affine nc map by the definition of $\rA(K)$. We must show that $\theta$ is a homeomorphism. Since each $K_n$ is compact, it suffices to show that $\theta$ is a bijection.

The injectivity of $\theta$ follows from the fact that $E_*$ separates points in $K$ by Lemma~\ref{L:affine separate}.

For the surjectivity of $\theta$, first note that $\theta(K) \subseteq L$ is a compact nc convex set. 
Suppose for the sake of contradiction that $\theta(K) \ne L$.
By Proposition~\ref{prop:equal-finite-levels}, there is a finite $n$ and a point $y_0 \in L_n \setminus \theta(K)_n$. 
Then by Corollary \ref{cor:separation}, there is a normal completely bounded linear map $\phi : A(K)^* \to \cM_n$ 
and self-adjoint $\gamma \in \cM_n$ such that
\[
\re \phi_n(y_0) \not \leq \gamma \otimes 1_n \quad \text{but} \quad \re \phi_p(y) \leq \gamma \otimes 1_p
\]
for every $p$ and $y \in \theta(K)_p$. Since $\phi$ is normal, we can identify $\phi$ with a continuous affine nc function $a \in \cM_n(\rA(K))$. By the definition of the operator system structure on $\rA(K)$, the second inequality implies $\re a \leq \gamma \otimes 1_{\rA(K)}$. Since $y_0$ is unital and completely positive, this implies
\[
\re y_0(a) \leq y_0(a) \leq \gamma \otimes 1_n,
\]
giving a contradiction.
\end{proof}

The next result is a noncommutative analogue of Kadison's representation theorem.

\begin{thm} \label{thm:nc-kadison-rep}

Let $S$ be a closed operator system with nc state space $K$. For $s \in S$, the function $\hat{s} : K \to \cM$ defined by
\[
\hat{s}(x) = x(s), \qfor s \in S,\ x \in K
\]
is a continuous affine nc function on $K$. The map $S \to \rA(K) : s \to \hat{s}$ is a complete order isomorphism.
\end{thm}

\begin{proof}
For $s \in S$, it is clear that $\hat{s}$ is a continuous affine function on $K$. Kadison's representation theorem implies that the map $S \to \rA(K) : s \to \hat{s}$ is an order isomorphism, so it remains to show that it is a complete order isomorphism. For this, it suffices to show that it preserves the matrix order, meaning that for $n \in \bN$ and $s \in \cM_n(S)$, if  $s \geq 0$ then $\hat{s} \geq 0$. But this follows immediately from the fact that $K$ consists of completely positive maps on $S$.
\end{proof}

\begin{defn}
We let $\mathrm{NCConv}$ denote the category with objects consisting of compact nc convex sets and morphisms consisting of continuous affine nc maps. We will refer to this as the category of {\em compact nc convex sets}. We let $\mathrm{OpSys}$ denote the category with objects consisting of closed operator systems and morphisms consisting of unital complete positive maps. We will refer to this as the category of {\em closed operator systems}. 
\end{defn}

We now define the functor $\rA : \mathrm{NCConv} \to \mathrm{OpSys}$ implementing the dual equivalence between $\mathrm{NCConv}$ and $\mathrm{OpSys}$. For a compact nc convex set $K$, $\rA(K)$ is the operator system of continuous affine nc functions on $K$ as in Definition \ref{defn:os-continuous-affine-fcns}. For compact nc convex sets $K$ and $L$ and a continuous affine map $\theta : K \to L$, $\rA(\theta): \rA(L) \to \rA(K)$ is a unital completely positive map defined by
\[
\rA(\theta)(b)(x) = b(\theta(x)), \qfor b \in \rA(L),\ x \in K.
\]

The functor $\rA$ has an inverse $\rA^{-1} : \mathrm{OpSys} \to \mathrm{NCConv}$. For an operator system $S$ with nc state space $K$, $\rA^{-1}(S) = K$. For operator systems $S$ and $T$ with nc state spaces $K$ and $L$ respectively and a unital completely positive map $\phi : T \to S$, $\rA^{-1}(S) : K \to L$ is a continuous affine nc map defined by
\[
b(\rA^{-1}(\phi)(x)) = \phi(b)(x), \qfor b \in T,\ x \in K.
\]

\begin{thm} \label{thm:dually-equivalent}
The map $A : \mathrm{NCConv} \to \mathrm{OpSys}$ is a contravariant functor with inverse $A^{-1} : \mathrm{OpSys} \to \mathrm{NCConv}$. In particular, the categories $\mathrm{NCConv}$ and $\mathrm{OpSys}$ are dually equivalent.
\end{thm}

The next result follows immediately from Theorem \ref{thm:dually-equivalent}.

\begin{cor}
Let $K$ and $L$ be compact nc convex sets. Then $\rA(K)$ and $\rA(L)$ are isomorphic if and only if $K$ and $L$ are affinely homeomorphic. Hence two operator systems are unitally completely order isomorphic if and only if their nc state spaces are affinely homeomorphic. 
\end{cor}

\section{Noncommutative functions} \label{sec:nc-functions}

\subsection{Functions on compact convex sets} \label{sec:classical-min-max-cstar-alg}

An essential component of classical Choquet theory is the interplay between the space $\rA(C)$ of continuous affine functions on a compact convex set $C$ and the C*-algebra $\rC(C)$ of continuous functions on $C$. 

The Stone-Weierstrass theorem implies that the C*-algebra $\rC(C)$ of continuous functions on $C$ is generated by the function system $\rA(C)$ of continuous affine functions on $C$. In fact, $\rC(C)$ is uniquely determined by the following universal property: $\rC(C)$ is generated by $\rA(C)$ and for any unital commutative C*-algebra $\rC(X)$ and unital order embedding $\phi : \rA(C) \to \rC(X)$ satisfying $\ca(\phi(\rA(C))) = \rC(X)$, there is a surjective homomorphism $\pi : \rC(C) \to \rC(X)$ such that $\pi|_{\rA(C)} = \phi|_{\rA(C)}$.
\[
\begin{tikzcd}
\rA(C) \arrow[rd, "\phi", hook] \arrow[r, "\id", hook] & \rC(C) \arrow[d, "\pi", two heads] \\
& \rC(X) = \ca(\phi(\rA(C)))
\end{tikzcd}
\]
This says that $\rC(C)$ is the maximal commutative C*-algebra generated by a unital order embedding of $\rA(C)$. 

The Riesz-Markov-Kakutani representation theorem implies that the state space of $\rC(C)$ can be identified with the space $\rP(C)$ of regular Borel probability measures on $C$. For a point $x \in C$, a measure $\mu \in \rP(C)$ is said to represent $x$, and $x$ is said to be the barycenter of $\mu$, if $\mu|_{\rA(C)} = x$. Since the point mass $\delta_x \in \rP(C)$ represents $x$, every point in $C$ has at least one representing probability measure. Moreover, $x$ has a unique representing measure if and only if $x \in \partial C$, where $\partial C$ denotes the extreme boundary of $C$. More generally, the Choquet-Bishop-de Leeuw integral representation theorem, which we will review later, asserts that for any $x \in C$, it is always possible to choose a representing measure that is supported on $\partial C$ in an appropriate sense.

The closure of the extreme boundary $\ol{\partial C}$ is the Shilov boundary of the function system $\rA(C)$. This means that the restriction map $\rho : \rC(C) \to \rC(\ol{\partial C})$ is a unital order embedding, and the C*-algebra $\rC(\ol{\partial C})$ is uniquely determined by the following universal property: for any unital commutative C*-algebra $\rC(X)$ and any unital order embedding $\phi : \rA(C) \to \rC(X)$ satisfying $\ca(\phi(\rA(C))) = \rC(X)$, there is a surjective homomorphism $\pi : \rC(X) \to \rC(\ol{\partial C})$ satisfying $\pi \circ \phi = \rho$. 
\[
\begin{tikzcd}
                                        & \rC(X) = \ca(\phi(\rA(C))) \arrow[d, "\pi", two heads] \\
\rA(C) \arrow[r, "\rho", hook] \arrow[ru, "\phi", hook] & \rC(\ol{\partial C})
\end{tikzcd}
\]
This says that $\rC(\ol{\partial C})$ is the minimal commutative C*-algebra generated by a unital order embedding of $\rA(C)$.

\subsection{Noncommutative functions}

In this section we will introduce a definition of nc function on a compact nc convex set. We will associate a C*-algebra of nc functions to every compact nc convex set that plays a role in the noncommutative setting analogous to the role in the classical setting of the C*-algebra of continuous functions on a compact convex set. In Section \ref{S:C*max}, we will see that the elements in this C*-algebra are, in fact, precisely the continuous nc functions on $K$, when continuity is defined in an appropriate sense.

\begin{defn}
Let $K$ be a compact nc convex set and let $f : K \to \cM$ be a function. We say that $f$ is an {\em nc function} if it is graded, respects direct sums and is unitarily equivariant, meaning that
\begin{enumerate}
\item $f(K_n) \subseteq \cM_n$ for all $n$,
\item $f(\sum \alpha_i x_i \alpha_i^*) = \sum \alpha_i f(x_i) \alpha_i^*$ for every family $\{ x_i \in K_{n_i} \}$ and every family of isometries $\{ \alpha_i \in \cM_{n_i,n} \}$ satisfying $\sum \alpha_i \alpha_i^* = 1_n$,
\item $f(\beta x \beta^*) = \beta f(x) \beta^*$ for every $x \in K_n$ and every unitary $\beta \in \cM_n$.
\end{enumerate}
We say that $f$ is {\em bounded} if $\|f\|_\infty < \infty$, where $\|f\|_\infty$ denotes the {\em uniform norm} defined by
\[
\|f\|_\infty = \sup_{x \in K} \|f(x)\|.
\]
We let $\rB(K)$ denote the space of all bounded nc functions on $K$.
\end{defn}

\begin{rem}
It is clear that affine nc functions on $K$ are in particular nc functions on $K$. Moreover, by Proposition \ref{P:uniform bound}, functions in the space $\rA(K)$ of continuous affine nc functions on $K$ are bounded nc functions. Therefore, $\rB(K)$ contains $\rA(K)$.
\end{rem}

\begin{rem}
The theory of analytic nc functions has gained a large following in recent years.
The book \cite{KKV2014} lays out the fundamentals of this theory, which has its roots in 
Taylor's work \cite{Tay1972} on a functional calculus for multivariable functions in non-commutating variables, as well as in Voiculescu's work \cite{Voi2008} on free probability. In the analytic setting, nc functions are similarity invariant (in an appropriate restricted sense), and not just unitarily invariant as in our setting.

In fact, the notion of an nc function has even older roots in the work of Takesaki \cite{Tak1967} on a Gelfand representation theorem for noncommutative C*-algebras. Takesaki considers admissible operator fields, which are maps $T : \Rep(A,H) \to \B(H)$ that respect direct sums and are unitarily equivariant. A modified definition due to Bichteler \cite{Bic1969} allowed degenerate representations and equivalence by partial isometries, which enabled him to extend Takesaki's representation theorem beyond the separable case.

We will see that a compact nc convex set $K$ can be identified, at least as a set, with the representation space of a C*-algebra $\rC(K)$, meaning that Takesaki's formulation is rather close to our notion of an nc function. While we will not directly require the theorem of Takesaki and Bichteler, we will adapt their proof to our setting.
\end{rem}

\begin{prop} \label {P:B(K)_is_a_vNalg}
Let $K$ be a compact nc convex set. Then $\rB(K)$ is a von Neumann algebra.
\end{prop}

\begin{proof}
The space $\rB(K)$ contains all affine nc functions, and in particular contains the constant function $\one$. For $f \in \rB(K)$,
the adjoint $f^*$ is given by $f^*(x) = f(x)^*$ for $x \in K$. It is clear that $f^*$ is graded and preserves direct sums.
For $x \in K_n$ and a unitary $\beta \in M_n$,
\[
 f^*(\beta x \beta^*) = f(\beta x \beta^*)^* = (\beta f(x) \beta^*)^* = \beta f(x)^* \beta^* = \beta f^*(x) \beta^*,
\]
So $f^*$ is an nc function. 
The product on $\rB(K)$ is defined pointwise: for $f,g \in \rB(K)$, $(fg)(x) = f(x) g(x)$ for $x \in K$.
This is easily seen to be an nc function.
It is also easy to check that the norm limit of nc functions is an nc function, so $\rB(K)$ is a Banach $*$-algebra.

There is a natural $*$-representation of $\rB(K)$, given by
\[
 \pi_u:\rB(K) \to \prod_n \prod_{x\in K_n} M_n ,\quad \pi_u(f)(x) = \bigoplus_{x\in K} f(x) = f(x_u),
\]
where $x_u = \bigoplus_{x\in K} x$.
This representation is isometric by definition, preserves the adjoint, and is linear and multiplicative.
Thus $\pi_u$ is a faithful $*$-representation, and $\rB(K)$ is a C*-algebra.
To check that $\pi_u(\rB(K))$ is WOT-closed, suppose  $(f_\lambda)_\Lambda$ is a net such that $T = \wotlim_\Lambda \pi(f_\lambda)$.
Then $T = \bigoplus_{x\in K} T(x)$, and hence determines a graded function.
Since each $f_\lambda$ preserves direct sums and is unitarily equivariant, the same is true for the WOT-limit.
Therefore $\pi_u(\rB(K))$ is a von Neumann algebra.

By a theorem of Dixmier \cite{Dixmier53}, $\pi_u(\rB(K))$ has a unique predual. 
So $\rB(K)$ also has a unique predual.
By Sakai's Theorem~\cite{Sak1956}, $\rB(K)$ is weak-$*$--weak-$*$ homeomorphic to $\pi_u(\rB(K))$.
Thus $\rB(K)$ is an abstract von Neumann algebra with faithful normal representation $\pi_u$.
\end{proof}

\begin{defn} \label{defn:nc-continuous-functions}
Let $K$ be a compact nc convex set. 
For $x \in K_n$, let $\delta_x : \rB(K) \to \B(H_n)$ denote the normal representation corresponding to evaluation at $x$.
Let $\rC(K)$ denote the C*-algebra subalgebra of $\rB(K)$ generated by $\rA(K)$.
\end{defn}

The analogy with the classical setting suggests that the C*-algebra $\rC(K)$ should consist precisely of the nc functions that are continuous on $K$. In Section \ref{S:C*max}, we will see that this is true when $K$ is equipped with the point-strong topology, for which we will adapt ideas from Takesaki \cite{Tak1967} and Bichteler's \cite{Bic1969} noncommutative Gelfand representation of a C*-algebra. However, it turns out that elements in $\rC(K)$ are not necessarily continuous on $K$ with respect to the point-weak* topology, as we will now discuss.

Observe that for $n < \infty$, an nc function in $\rA(K)$ is weak*-to-norm continuous on $K_n$ since the weak* and norm topologies agree on $\cM_n$. Since $\rC(K)$ is the C*-algebra generated by $\rA(K)$, and since multiplication on $\cM_n$ is jointly continuous in the norm topology, it follows that for $n < \infty$, elements in $\rC(K)$ are continuous on $K_n$ with respect to the point-weak* topology. However, the next example shows that for infinite $n$, elements in $\rC(K)$ are not necessarily continuous on $K_n$ with respect to the point-weak* topology.

\begin{example} \label{ex:continuity-may-fail-infinite-levels}
Consider the function system $S = \spn\{ 1, z, \bar{z}\}$ in $ \rC(\bT)$. For each $n$, every unital completely positive map $\phi : S \to \cM_n$ is determined by $\alpha := \phi(z) \in \cM_n$. Clearly $\|\alpha\| \leq 1$. On the other hand, for $\alpha \in \cM_n$ with $\|\alpha\| \leq 1$, von Neumann's inequality implies the existence of a unital completely positive map $\phi : S \to \cM_n$ satisfying $\phi(z) = \alpha$.

Let $K$ denote the nc state space of $S$, so that $S$ is isomorphic to $\rA(K)$. Then it follows from above that for each $n$, $K_n$ is affinely homeomorphic to the unit ball of $\cM_n$. Let $a \in \rA(K)$ denote the nc function corresponding to $z \in S$. Then $a^*$ corresponds to $\bar{z} \in S$. For an nc state $x \in K_n$ corresponding to a contraction $\alpha \in \cM_n$ as above, $a(x) = \alpha$ and $a^*(x) = \alpha^*$. Let $a^* a$ denote the pointwise product of $a^*$ and $a$. Then $a^* a$ is an nc function on $K$ with $(a^* a)(x) = a^*(x) a(x)$.

For $t \in (0,1]$, identify $\cM_{\aleph_0}$ with $\B(L^2[0,1])$ and let $x_t \in K_{\aleph_0}$ denote the nc state corresponding to the isometry $\alpha_t \in \cM_{\aleph_0}$ defined by
\[ \alpha_t(f)(x) = \begin{cases} t^{-1/2} f(t^{-1}x) & 0 < x < t,\\ 0 & t \le x < 1. \end{cases} \]
Let $x_0 \in K_{\aleph_0}$ denote the point corresponding to $0_{\aleph_0}$. Then 
\[ \lim_{t \to 0} a(x_t) = \lim_{t \to 0} \alpha_t = 0 .\]
Similarly, $\lim_{t \to 0} a^*(x_t) = 0$. So $\lim_{t \to 0} x_t = x_0$. However,
\[
\lim_{t \to 0} (a^* a)(x_t) = \lim_{t \to 0} \alpha_t^* \alpha_t = 1_{\aleph_0} \ne 0 = (a^* a)(x_0).
\]
It follows that $a^* a$ is not continuous as a function on $K_{\aleph_0}$.
\end{example}

The next example shows that bounded nc functions are not necessarily determined by their values on finite levels.

\begin{example} \label{ex:ncfunctions-not-determined-by-finite-levels}
Consider the Cuntz operator system
\[
S = \spn\{1,s_1,s_2,s_1^*,s_2^*\} \subseteq \O_2,
\]
where $s_1,s_2$ are the standard generators of the Cuntz C*-algebra $\O_2$ (see Example~\ref{ex:cuntz} for more details). Let $K$ denote the nc state space of $S$ so that $S$ is completely order isomorphic to $\rA(K)$. For each $n$, a point $x \in K_n$ determines a contractive $1 \times 2$ matrix $[x(s_1)\ x(s_2)]$ with entries in $\cM_n$ (a row contraction). Conversely, a row contraction $[\alpha_1\ \alpha_2]$ with entries in $\cM_n$ determines a point in $x \in K_n$. So $K_n$ is affinely homeomorphic to the compact convex set of row contractions with entries in $\cM_n$ (which we can identify with a subset of $\cM_{n,2n}$). Since $\O_2$ is simple and infinite dimensional, every representation is infinite dimensional. The corresponding representation $\delta_x$ of $\rC(K)$ factors through $\O_2$ if and only if $[\alpha_1\ \alpha_2]$ is a unitary. 

Every row contraction decomposes uniquely as $[\alpha_1\ \alpha_2] \oplus [\beta_1\ \beta_2]$ where $[\alpha_1\ \alpha_2]$ is a row unitary
and $[\beta_1\ \beta_2]$ has no row unitary summand. This property of having no row unitary summand is preserved by direct sums.
Define a function $f : K \to \cM$ by
\[
 f(x) = 1_n \oplus 0_m \qif [x(s_1)\ x(s_2)] = [\alpha_1\ \alpha_2] \oplus [\beta_1\ \beta_2]
\]
where $[\alpha_1\ \alpha_2]$ is a row unitary and $[\beta_1\ \beta_2]$ has no row unitary summand.
Then $f$ is evidently bounded, graded, preserves direct sums and is unitarily equivariant. 
So it is a bounded nc function (in fact a projection) in $\rB(K)$. 
The significance of this example is that $f$ vanishes at all finite levels, but is nonzero. 
Therefore nc functions are not necessarily determined by their values on finite levels. 
Moreover this function is continuous on all finite levels but is not continuous on $K_{\aleph_0}$.
\end{example}

\subsection{$\rC(K)$ as a C*-algebra} \label{S:C(K)_as_C*alg}

Our goal is to identify $\rB(K)$ as the bidual of $\rC(K)$, so that it is the universal enveloping von Neumann algebra.
Our approach is inspired by Takesaki's \cite{Tak1967} Gelfand duality theorem for noncommutative C*-algebras, which was extended to non-separable C*-algebras by Bichteler \cite{Bic1969}.
In an earlier draft of this paper, we directly used the theorem of Takesaki and Bichteler. We will now give a slightly simpler direct proof, but our proof is clearly based on theirs. 

We will assume that the cardinal $\kappa$ is at least as large as the largest dimension of a Hilbert space on which $\rC(K)$ has a cyclic representation.
Let $\Rep(\rC(K)) = \coprod_{n \le\kappa} \Rep(\rC(K), H_n)$, where $\Rep(\rC(K), H_n)$ is the space of all non-degenerate
$*$-representations of $\rC(K)$ into $\B(H_n)$.
Each $f\in\rB(K)$ determines a bounded function on $\Rep(\rC(K))$ by $\hat f(\pi) = \pi(f)$.

Put the topology of point-weak-$*$ convergence on $\Rep(\rC(K))$; i.e., the weakest topology such that for each $f \in \rC(K)$, $n \le \kappa$ and 
a weak-$*$-open set $U \subset \B(H_n)$, the set $\{ \pi \in  \Rep(\rC(K), H_n) : \pi(f) \in U \}$ is open.
If $f \in \rC(K)$, then $\hat f$ is continuous by definition of the topology.
It is a well-known result that the point-weak-$*$, point-WOT, point-SOT, point-ultrastrong and point-ultrastrong-$*$ topologies all coincide on $\Rep(\rC(K))$.
See \cite{Bic1969}*{Section II, Lemma}.

On the other hand, we can consider the point-ultrastrong-$*$ topology $\tau_{us^*}$ on $K$, the weakest topology such that each $f \in \rA(K)$ is point-ultrastrong* continuous on $K$. This does not generally coincide with the weak-$*$ topology, and in particular $K$ will not generally be compact with respect to $\tau_{us^*}$. We will write $(K, \tau_{us^*})$ for this space.

\begin{rem}
For $n \ge \aleph_0$, the space $\Rep(\rC(K), H_n)$ is not generally compact, although it naturally sits as a subset of the nc state space of $\rC(K)$, which is compact.
Indeed, let $\pi \in \Rep(\rC(K), H_n)$ be any non-degenerate non-irreducible representation with image not consisting only of scalars.
Then there is a projection $P \in \B(H_n)$ with infinite rank which is not in $\pi(\rC(K))''$. 
Then $\phi(f) = P \pi(f) |_{PH_n}$ is a unital completely positive map which is not a representation.
Let $\alpha$ be an isometry of $H_n$ onto $PH_n$.
Then $\alpha^* \phi \alpha \in \ucpmaps(\rC(K), H_n) \setminus \Rep(\rC(K), H_n)$.
Let $Q_n \in \B(H_n)$ be an increasing sequence of finite rank projections converging strongly to the identity, and choose unitaries $u_n$ such that $u_nQ_n = \alpha Q_n$. Then $u_n$ converges strongly to $\alpha$ and hence $u_n^* \pi u_n \in \Rep(\rC(K), H_n)$ converges point-weak-$*$ to $\phi$. Hence $\Rep(\rC(K), H_n)$ is not closed in $ \ucpmaps(\rC(K), H_n)$.
\end{rem}

\begin{prop} \label {P:reps of C(K)}
The restriction map $\rho: \Rep(\rC(K)) \to K$ given by $\rho(\pi) = \pi|_{\rA(K)}$ is a continuous, bijective graded function which preserves direct sums and is unitarily equivariant, i.e. $\rho(u\pi u^*) = u \rho(\pi) u^*$ for $\pi \in  \Rep(\rC(K), H_n)$ and $u \in \U(H_n)$.
However the inverse map is generally not continuous.
\end{prop}

\begin{proof}
The topology on $\Rep(\rC(K))$ and $K$ are the point-weak-$*$ topologies with respect to $\rC(K)$ and $\rA(K)$ respectively.
Hence $\rho$ is continuous.
For $\pi \in \Rep(\rC(K),H_n)$, $ \pi|_{\rA(K)}$ is a unital completely positive map into $\B(H_n)$, and thus belongs to $K_n$, so $\rho$ is graded.
Also $\rho$ clearly preserves direct sums and is unitarily equivariant.

Each $x\in K$ gives rise to the representation $\delta_x$, and $\rho(\delta_x) = x$. So $\rho$ is surjective.
Since $\rC(K)$ is the C*-algebra generated by $\rA(K)$, there is at most one way to extend any unital completely positive map on $\rA(K)$ to a $*$-homomorphism.
Thus $\rho$ is a bijection.

The nc convex set $K$ is compact, but $\Rep(\rC(K))$ is not generally compact.
Hence $\rho^{-1}$ cannot be continuous in general.
\end{proof}

\begin{thm} \label {T:B(K) is second dual}
$\rB(K)$ is the universal von Neumann algebra $\rC(K)^{**}$ of $\rC(K)$.
The weak-$*$ closure of $\rA(K)$ in $\rB(K)$ is isometrically isomorphic to $\rA(K)^{**}$,
and it coincides with the space of all bounded affine nc functions on $K$.
\end{thm}

\begin{proof}
We work in $\pi_u(\rB(K))$.
Suppose that $P$ is a projection in $\pi_u(\rC(K))'$. 
Then $PH_u$ reduces $\pi_u$, so that $\pi_u = \pi_1 \oplus \pi_2$ with respect to $H_u = PH_u \oplus P^\perp H_u$.
By the previous proposition, there are points $x_1$ and $x_2$ in $K$ so that $\pi_i = \delta_{x_i}$ and $x_u = x_1 \oplus x_2$.
In particular, for each $i$, $\pi_i$ extends to a normal representation $\delta_{x_i}$ of $\rB(K)$.
Therefore, for $f\in\rB(K)$, 
\[
 \pi_u(f) = f(x_1 \oplus x_2) = f(x_1) \oplus f( x_2 ) = \pi_1(f) \oplus \pi_2(f).
\]
In particular, $\pi_u(f)$ commutes with $P$.
Since the von Neumann algebra $\pi_u(\rC(K))'$ is spanned by its projections, $\pi_u(\rB(K)) \subset \pi_u(\rC(K))''$.
Now $ \pi_u(\rC(K))''$ is the WOT-closure of $\pi_u(\rC(K))$ by the double commutant theorem.
Since $\rC(K) \subset \rB(K)$ and $\pi_u(\rB(K))$ is a von Neumann algebra, $\pi_u(\rB(K)) = \pi_u(\rC(K))''$.

Every representation of $\rC(K)$ is of the form $\delta_x$ for $x\in K$ by Proposition~\ref{P:reps of C(K)},
and $\delta_x$ extends to a normal representation of $\rB(K)$.
This is the property which defines the universal enveloping von Neumann algebra of $\rC(K)$.
By the characterization of the bidual of a C*-algebra \cites{Sherman, Takeda}, $\rB(K)$ is isomorphic to the bidual of $\rC(K)$.

Since $\rA(K)$ is a closed subspace of $\rC(K)$, the weak-$*$ closure of $\rA(K)$ in $\rB(K)$ is isometrically isomorphic to $\rA(K)^{**}$.
By Theorem~\ref{T:affine cnts on K1}, $\rA(K)$ is isomorphic to $A(K_1)$.
By \cite{AS2001}*{Propositions 2.128 and 2.3}, $A(K_1)^{**}$ is the space of all bounded affine functions on $K_1$.
Therefore the restriction map $\rho$ from $\rA(K)^{**}$ to $K_1$ must map onto the space of all bounded affine functions on $K_1$.
By Theorem~\ref{T:affine cnts on K1}, this restriction is bounded below, is positive and unital, and has a positive inverse.
Therefore it is an order isomorphism, and $\rA(K)^{**}$ coincides with the space of all bounded affine nc functions on $K$.
\end{proof}

\begin{rem}
Von Neumann algebras have a unique predual, and conversely by Sakai's theorem, 
a C*-algebra with a predual has a faithful normal representation as a von Neumann algebra.
The predual of $\rB(K)$ is given by the weak-$*$ topology on $\pi_u(\rB(K)) \subset \B(H_u)$.
Indeed, every linear functional has the form $\phi(f) = \ip{ \pi_u(f) \xi, \eta}$ for $\xi,\eta \in H_u$.
\end{rem}

\subsection{Continuity and $\rC(K)$} \label{S:C(K) continuity}

We can now deal with the issue of continuity of functions in $\rC(K)$. This section also relies on ideas from Takesaki and Bichteler,

\begin{lem} \label {L:tau us*}
Every $f\in \rC(K)$ is continuous on $(K, \tau_{us*})$. 
The restriction map $\rho: \Rep(\rC(K), \text{point-ultrastrong}^*) \to (K, \tau_{us*})$ is a homeomorphism.
\end{lem}

\begin{proof}
As in Proposition~\ref{P:reps of C(K)}, $\rho$ is a continuous bijection.

Every $f \in \rA(K)$ is continuous on $(K, \tau_{us*})$ by definition.
However multiplication and adjoints are continuous in the ultrastrong-$*$ topology.
Thus polynomials of elements of $\rA(K)$ are also continuous from $K_n$ to $(\M_n, \text{ultrastrong-}*)$. 
It follows that every element of $\rC(K)$ is point-ultrastrong-$*$ continuous on $(K, \tau_{us*})$.
This means that $\rho^{-1}$ is continuous. Hence the restriction map is a homeomorphism.
\end{proof}

\begin{thm} \label {T:continuity of C(K)}
Let $f \in \rB(K)$. The following are equivalent:
\begin{enumerate}[labelwidth=5mm,align=left]
\item $f \in \rC(K)$.
\item[$(2a)$] $\hat f$ is continuous on $(\Rep(\rC(K)), \text{point-ultrastrong-}*)$.
\item[$(2b)$] $\hat f$ is continuous on $(\Rep(\rC(K)), \text{point-SOT})$.
\item[$(2c)$] $\hat f$ is continuous on $(\Rep(\rC(K)), \text{point-weak-$*$})$.
\stepcounter{enumi}
\item $f$ is continuous on $(K, \tau_{us*})$.
\item $f$ is continuous on $(L_1, \tau_{\rC(K)})$, where $L_1$ be the scalar state space of $\rC(K)$ with its weak-$*$ topology.
\end{enumerate}
\end{thm}

\begin{proof}
Let $L$ be the nc state space of $\rC(K)$.
Then $L_1$ is the (classical) state space of $\rC(K)$.
Define a map $\theta$ on $\rB(K)$ by $\theta(f) (s) = f(s)$ for $s \in L_1$.
Every $s \in \rC(K)^* = \rB(K)_*$ is a weak-$*$ continuous linear functional on $\rB(K) = \rC(K)^{**}$.
$L_1$ is a compact convex set in the $\tau_{\rC(K)}$ topology, and $\theta(f)$ is a bounded affine function on $L_1$.
Every unit vector $\xi \in H_u$ determines a state $s(f) = \ip{\pi_u(f) \xi,\xi}$ in $L_1$.
It follows that $\| \theta(f) \| = \|f\|$ when $f=f^*$, and $\frac12 \|f\| \le \| \theta(f) \| \le \|f\|$ in general.

By categorical duality, $\rC(K)$ is completely order isomorphic to $\rA(L)$.
The map $\theta: \rC(K) \to A(L_1)$ is a linear isomorphism by Theorem~\ref{T:affine cnts on K1}.
Since $\theta$ is injective, $\theta(f)$ is continuous on $L_1$ if and only if $f\in\rC(K)$.
So (1) and (4) are equivalent.

As mentioned above, \cite{Bic1969}*{Section II, Lemma} shows that (2a), (2b) and (2c) are equivalent.
Also (1) implies that $\hat f$ is continuous on $\Rep(\rC(K))$ in all of these topologies. 
Lemma~\ref{L:tau us*} shows that (2a) and (3) are equivalent.

It remains to establish that if $\hat f$ is continuous as a function on $\Rep(\rC(K))$ in the point-SOT topology, 
then $f$ must be continuous on $L_1$. 
Let $(\phi_\lambda)_\Lambda$ be a net in $L_1$ converging to $\phi$.
We will construct a cofinal subnet such that $\big( f(\phi_{\lambda'}) \big)_{\Lambda'}$ converges to $f(\phi)$.
The same argument then shows that every cofinal subnet has a cofinal subnet which converges to $f(\phi)$,
which implies that the original net converges to $f(\phi)$; whence $f$ is continuous on $L_1$.

Fix an irreducible representation $\pi_0$ in $\Rep(\fA, H_0)$.
Let $(\pi_\phi, \xi_\phi)$ denote the GNS representation of $\phi$.
Set  $\pi = \pi_\phi \oplus (\pi_0\otimes 1_\kappa)$ in $\Rep(\fA, H_\kappa)$ and set $\xi = \xi_\phi \oplus 0 \in H_\kappa = H_\phi \oplus K_0$,
where $K_0 = H_0 \otimes H_\kappa$.
Clearly $(\pi,\xi)$ represents $\phi$.

Let $\Lambda'$ be the net consisting of finite subsets 
\[ F = \{ (f_i,\xi_i) : f_i \in b_1(\rC(K)), \ \xi_i \in b_1(H_\kappa) : 1 \le i \le n \} \]
ordered by inclusion. 
To each $F$, we associate the open set $U_F \times V_F$ in $\Rep(\fA, H_\kappa) \times H_\kappa$ given by 
\[
 U_F = \big\{ \sigma \in \Rep(\fA, H_\kappa) : \| \sigma(f_i) \xi_i - \pi(f_i) \xi_i \| < \ep \ \FOR 1 \le i \le n  \big\}  
\]
and $V_F = b_{1/n}(\xi)$.
By Bichteler \cite{Bic1969}*{section II, Proposition 4}, there is a neighbourhood $W_F$ of $\phi$ in $L_1$ 
so that every $\psi\in W_F$ has a representative $(\sigma,\eta)$ in $U_F \times V_F$; i.e., $\psi(f) = \ip{ \sigma(f) \eta, \eta}$.
We construct our subnet $\psi_F$ recursively, defining it for $|F|=n$ at the $n$th stage.
Assuming that $\psi_G = \phi_{\lambda_G}$ has been defined for all $G \subsetneq F$, 
select $\lambda > \lambda_G$ for all $G\subsetneq F$ with $\phi_\lambda \in W_F$; and set $\lambda_F = \lambda$.
Let $(\pi_F, \xi_F) \in U_F \times V_F$ represent $\psi_F$.
By construction, $(\pi_F, \xi_F)$ converges in $(\Rep(\rC(K)), \text{point-SOT})$ to $(\pi,\xi)$.
Since $\hat f$ is continuous, 
\[
 \lim_F \psi_F(f) = \lim_F \ip{\pi_F(f) \xi_F,\xi_F} = \ip{\pi(f)\xi,\xi} = \phi(f) .
\]
Therefore (2b) implies (4), completing the proof.
\end{proof}

The proof of Proposition~\ref{P:uniform bound} can be applied verbatim to prove the next result.

\begin{prop}
Let $K$ be a compact nc convex set and let $f  \in \rC(K)$ be a continuous nc function. Then $f$ is bounded with
\[
\|f\| = \|f|_{K_{\aleph_0}}\| = \sup_{n < \infty} \|f|_{K_n}\|.
\]
\end{prop}

For the remainder of this paper, we refer to elements in $\rC(K)$ as continuous nc functions.

\subsection{Maximal C*-algebra} \label{S:C*max}

In this section we will show that $\rC(K)$ is uniquely determined by an important universal property.

Kirchberg and Wassermann \cite{KirWas1998} introduced the maximal C*-algebra $\cmax(S)$ of an operator system $S$. 
This C*-algebra is uniquely determined up to isomorphism by the following universal property: 
there is a unital complete order embedding $\iota : S \to \cmax(S)$ such that $\ca(\iota(S)) = \cmax(S)$ 
and for any C*-algebra $A$ and unital complete order embedding $\phi : S \to A$ satisfying $\ca(\phi(S)) = A$, 
there is a unique homomorphism $\pi : \cmax(S) \to A$ satisfying $\pi \circ \iota = \phi$.
\[
\begin{tikzcd}
S \arrow[rd, "\phi", hook] \arrow[r, "\iota", hook] & \cmax(S) \arrow[d, "\pi", two heads] \\
& A = \ca(\phi(S))
\end{tikzcd}
\]

An easy consequence of the existence of $\cmax(S)$ is the fact that every unital completely positive map $\phi:S \to \B(H)$ extends to a (unique) $*$-homomorphism of $\cmax(S)$.
This property also characterizes $\cmax(S)$.
We have observed that $\rC(K)$ has this property relative to $\rA(K)$.
Therefore we obtain the following result.

\begin{thm} \label{T:C*max}
Let $K$ be a compact nc convex set. Then $\rC(K) \simeq \cmax(\rA(K))$ via a $*$-isomorphism which is the identity on $\rA(K)$.
\end{thm}

\subsection{Representing maps} \label{SS:representing maps}

For a compact nc convex set $K$, unital completely positive maps $\mu : \rC(K) \to \cM_n$ play the role of probability measures in the classical setting. In this section we will introduce a natural notion of representing maps for points in $K$.

\begin{defn} \label{defn:representing-map}
Let $K$ be a compact nc convex set. For $x \in K_n$, we say that a unital completely positive map $\mu : \rC(K) \to \cM_n$ {\em represents} $x$, and that $x$ is the {\em barycenter} of $\mu$, if $\mu$ restricts to $x$ on the function system $\rA(K)$ of continuous affine functions on $K$, i.e. if $\mu|_{\rA(K)} = x$.
If $\delta_x$ is the unique representing map for $x$, then we will say that $x$ has a {\em unique representing map}.
\end{defn}

It will be important to determine the points in $K$ that have unique representing maps. We will revisit this in Section \ref{sec:representations-maps}.

Because of the identification of $\rB(K)$ with the enveloping von Neumann algebra of $\rC(K)$ in Section \ref{S:C*max}, every unital completely positive map $\mu : \rC(K) \to \cM_n$ has a unique weak*-continuous extension from $\rB(K)$ to $\cM_n$. We will continue to denote this extension by $\mu$.

\subsection{Minimal C*-algebra} \label{sec:cmin}

In this section we will review the notion of the Shilov boundary of an operator system along with the corresponding notion of minimal C*-algebra of an operator system which, as in the classical setting with the C*-algebra of continuous functions on the Shilov boundary, satisfies an important universal property.

The existence of a noncommutative analogue of the Shilov boundary was conjectured by Arveson \cite{Arv1969}, and the existence and uniqueness was proved by Hamana \cite{Ham1979}. For an operator system $S$, the {\em minimal C*-algebra} $\cmin(S)$ is uniquely determined up to isomorphism by the following universal property: there is a unital complete order embedding $\iota : S \to \cmin(S)$ such that $\ca(\iota(S)) = \cmin(S)$ and for any unital C*-algebra $A$ and unital complete order embedding $\phi : S \to A$ satisfying $\ca(\phi(S)) = A$, there is a surjective homomorphism $\pi : A \to \cmin(S)$ satisfying $\pi \circ \phi = \iota$.
\[
\begin{tikzcd}
                                        & A = \ca(\phi(S)) \arrow[d, "\pi", two heads] \\
S \arrow[r, "\iota", hook] \arrow[ru, "\phi", hook] & \cmin(S)                          
\end{tikzcd}
\]
In the literature, $\cmin(S)$ is often referred to as the {\em C*-envelope} of $S$. 

The minimal C*-algebra has been computed for many operator systems in the literature. For now, we give two simple examples.  We will consider more examples in Section \ref{sec:examples}.

\begin{example}
If $A$ is a unital C*-algebra, then it is clear that $\cmin(A) = A$.
\end{example}

\begin{example}
Let $A$ be a simple unital C*-algebra and let $S \subseteq A$ be an operator system such that $\ca(S) = A$. Since $\cmin(S)$ is a quotient of $A$, the simplicity of $A$ implies that $\cmin(S) = A$.
\end{example}

Let $K$ be a compact nc convex set. Then it follows from the universal properties of the maximal C*-algebra $\rC(K)$ and the minimal C*-algebra $\cmin(\rA(K))$ that there is a unique surjective homomorphism $\pi : \rC(K) \to \cmin(\rA(K))$ such that $\pi|_{\rA(K)} = \iota$, where $\iota : \rA(K) \to \cmin(\rA(K))$ denotes the canonical unital complete order embedding. We will say more about the relationship between $K$ and the structure of $\cmin(\rA(K))$ in Section \ref{sec:extreme-pts-minimal-c-star-alg}.

\section{Dilations of points and representations of maps} \label{sec:dilations-representations}

\subsection{Dilations, compressions and maximal points}

For a compact nc convex set $K$, unital completely positive maps on $\rC(K)$ play the role of probability measures in the classical setting. The nc state space of $\rC(K)$ is a compact nc convex set, and relationships between the graded components of this space provide it with a rich structure that has no classical counterpart.

\begin{defn} \label{defn:dilation}
Let $K$ be an nc convex set. We will say that a point $x \in K_m$ is {\em dilated} by a point $y \in K_n$ and refer to $y$ as a {\em dilation} of $x$ if there is an isometry $\alpha \in \cM_{n,m}$ such that $x = \alpha^* y \alpha$. In this case we will say that $x$ is a {\em compression} of $y$. If $y$ decomposes with respect to the range of $\alpha$ as $y = y_1 \oplus y_2$ for some $y_i \in K$, then we will say that the dilation is {\em trivial}. We will say that $x$ is {\em maximal} if it has no non-trivial dilations.
\end{defn}

\begin{rem}
If $y$ is a trivial dilation of $x$, then $\alpha$ implements a unitary equivalence between $x$ and $y_1$; so $y \cong x \oplus y_2$.

Suppose that $x \in K_n$ can be written as a finite nc convex combination $x = \sum \alpha_i^* x_i \alpha_i$ for $\{x_i \in K_{n_i}\}$ and $\{\alpha_i \in \cM_{n_i,n}\}$ satisfying $\sum \alpha_i^* \alpha_i = 1_n$. Let $y = \oplus_{i=1}^k x_i$ and let $\alpha = [\alpha_1 \ \cdots \ \alpha_k]^t$. Then $\alpha$ is an isometry and $x = \alpha^* y \alpha$, so $x$ is a compression of $y$. Hence if $x$ is maximal, then $y \simeq x \oplus z$ for some $z \in K$.
\end{rem}

The next result is a restatement of an important result of Dritschel and McCullough \cite{DriMcC2005}*{Theorem 1.2}.

\begin{restatable}{thm}{MaximalDilationTheorem} \label{thm:dritschel-mccullough}
Let $K$ be a compact nc convex set. Then every point in $K$ has a maximal dilation.
\end{restatable}

We will give a new proof of Theorem \ref{thm:dritschel-mccullough} in Section \ref{sec:max-in-diln-order} using ideas from this paper.

\subsection{Representations of maps} \label{sec:representations-maps}

Stinespring's dilation theorem asserts that completely positive maps on C*-algebras dilate to representations. However, understanding the dilation theory of completely positive maps on more general operator systems is a much more difficult problem. The framework of noncommutative convexity provides a powerful new perspective on this issue.

Let $K$ be a compact nc convex set. In this section we will begin to see how questions about unital completely positive maps on $\rC(K)$ can be reduced to questions about points in $K$.

If $\pi : \rC(K) \to \cM_n$ is a representation, then there is an nc state $x \in K_n$ such that $\pi = \delta_x$. Specifically, $x = \pi|_{\rA(K)}$ is the barycenter of $\pi$. Therefore, if $\mu : \rC(K) \to \cM_m$ is a unital completely positive map, then Stinespring's theorem implies there is a point $x \in K_n$ and an isometry $\alpha \in \cM_{m,n}$ such that $\mu = \alpha^* \delta_x \alpha$. Considered as points in the nc state space of $\rC(K)$, $\mu$ is dilated by $\delta_x$ in the terminology of Definition \ref{defn:dilation}. 

\begin{defn}
Let $K$ be a compact nc convex set and let $\mu : \rC(K) \to \cM_m$ be a unital completely positive map. We will say that a pair $(x,\alpha)$ consisting of a point $x \in K_n$ and an isometry $\alpha \in \cM_{n,m}$ {\em is a representation of $\mu$} if $\mu = \alpha^* \delta_x \alpha$. We will say that the representation $(x,\alpha)$ of $\mu$ is {\em minimal} if $\{ f(x) \alpha H_m : f \in \rC(K) \}$ is dense in $H_n$. 
\end{defn}

\begin{rem}
By Stinespring's theorem, a minimal representation $(x,\alpha) \in K_n \times \cM_{n,m}$ of $\mu$ is unique in the sense that if $(y,\beta) \in K_p \times \cM_{p,m}$ is another minimal representation of $\mu$, then $n = p$ and there is a unitary $\gamma \in \cM_n$ such that $x = \gamma y \gamma^*$ and $\alpha = \gamma \beta$. 
\end{rem}

In Section \ref{SS:representing maps}, we observed that every unital completely positive map $\mu : \rC(K) \to \cM_m$ extends to a unital completely positive map $\mu : \rB(K) \to \cM_m$ using the fact that $\rB(K)$ is the enveloping von Neumann algebra of $\rC(K)$. This extension can be described more concretely in the following way: Let $(x,\alpha) \in K_n \times \cM_{n,m}$ be a minimal representation of $\mu$. Then $\mu$ can be extended by defining
\[
\mu(f) = \alpha^* f(x) \alpha, \qfor f \in \rB(K).
\]
To see that this extension is well defined, let $(y,\beta) \in K_p \times \cM_{p,m}$ be another minimal representation. Then from above, there is a unitary $\gamma \in \cM_{n}$ such that $x = \gamma y \gamma^*$ and $\alpha = \gamma \beta$. Then by the unitary equivariance of $f$,
\[
\beta^* f(y) \beta = \beta^* f(\gamma^* x \gamma) \beta = \beta^* \gamma^* f(x) \gamma \beta = \alpha^* f(x) \alpha.
\]
The map $\delta_x$ is normal on $\rB(K)$, so $\mu$ is the composition of normal maps, and hence is itself normal. The fact that this definition of $\mu$ agrees with the previous definition now follows from the uniqueness of the normal extension of a unital completely positive map to the enveloping von Neumann algebra.

Using the notion of maximal points, we can now characterize points with unique representing maps in the sense of Section \ref{SS:representing maps}.

\begin{prop} \label{prop:criterion-unique-rep-map}
Let $K$ be a compact nc convex set. A point in $K$ has a unique representing map if and only if it is maximal.
\end{prop}

\begin{proof}
Suppose $x \in K_m$ has a unique representing map. Let $y \in K_n$ be a maximal dilation of $x$. Then there is an isometry $\alpha \in \cM_{n,m}$ such that $x = \alpha^* y \alpha$. Define a unital completely positive map $\mu : \rC(K) \to \cM_m$ by $\mu = \alpha^* \delta_y \alpha$. Then $\mu$ has barycenter $x$. Since $x$ has a unique representing map, it follows that $\mu = \delta_x$. Therefore, $\delta_y \cong \delta_x \oplus \delta_z$ for some $z \in K_p$, where the decomposition is taken with respect to the range of $\alpha$. In particular, $y \cong x \oplus z$. Since the summands of a maximal point in $K$ are maximal, it follows that $x$ is maximal.

Conversely, suppose that $x \in K_m$ is maximal. Let $\mu : \rC(K) \to \cM_m$ be a unital completely positive map with barycenter $x$. Let $(y,\alpha) \in K_n \times \cM_{n,m}$ be a representation of $\mu$. Then $x = \alpha^* y \alpha$, so $y$ is a dilation of $x$. The fact that $x$ is maximal implies that $y \cong x \oplus z$ for some $z \in K_p$, where the decomposition is taken with respect to the range of $\alpha$. Hence $\delta_y \cong \delta_x \oplus \delta_z$, and so $\mu = \delta_x$.
\end{proof}

\begin{prop} \label{prop:maximal-reps-factor-through-cmin}
Let $K$ be a compact nc convex set. If $x \in K_n$ is maximal, then the corresponding representation $\delta_x : \rC(K) \to \cM_n$ factors through $\cmin(\rA(K))$. Conversely, if the only representing map for $x$ that factors through $\cmin(\rA(K))$ is $\delta_x$, then $x$ is maximal.
\end{prop}

\begin{proof}
Suppose $x \in K_n$ is maximal. Let $\iota : \rA(K) \to \cmin(\rA(K))$ denote the canonical embedding and define $\phi : \iota(\rA(K)) \to \cM_n$ by $\phi = x \circ \iota^{-1}$. By Arveson's extension theorem we can extend $\phi$ to a unital completely positive map $\psi : \cmin(\rA(K)) \to \cM_n$. Let $q : \rC(K) \to \cmin(\rA(K))$ denote the canonical quotient map. Then $(\phi \circ q)|_{\rA(K)} = x$. Since $x$ has a unique representing map, it follows that $\phi \circ q = \delta_x$. In particular, $\ker \delta_x \supseteq \ker q$. 

Conversely,  suppose that the only representing map for $x$ that factors through $\cmin(\rA(K))$ is $\delta_x$. Let $y \in K_p$ be a maximal dilation of $x$ and let $\alpha  \in \cM_{p,n}$ be an isometry such that $x = \alpha^* y \alpha$. Define a unital completely positive map $\mu : \rC(K) \to \cM_n$ by $\mu = \alpha^* \delta_y \alpha$. Then $\mu$ has barycenter $x$. From above, $\delta_y$ factors through $\cmin(\rA(K))$. Hence $\mu$ also factors through $\cmin(\rA(K))$. Therefore, by assumption $\mu = \delta_x$ and arguing as in the proof of Proposition \ref{prop:criterion-unique-rep-map} implies that $x$ is maximal.
\end{proof}

\section{Extreme points} \label{sec:extreme-points}

\subsection{Extreme points}

In this section we will introduce the definition of extreme point for an nc convex set. The basic idea is that there should be no way of expressing an extreme point as a non-trivial nc convex combination.

\begin{defn} \label{defn:extreme-point}
Let $K$ be an nc convex set. We will say that a point $x \in K_n$ is {\em extreme} if whenever $x$ is written as a finite nc convex combination $x = \sum \alpha_i^* x_i \alpha_i$ for $\{x_i \in K_{n_i}\}$ and nonzero $\{\alpha_i \in \cM_{n_i,n}\}$ satisfying $\sum \alpha_i^* \alpha_i = 1_n$, then each $\alpha_i$ is a positive scalar multiple of an isometry $\beta_i \in \cM_{n_i,n}$ satisfying $\beta_i^* x_i \beta_i = x$ and each $x_i$ decomposes with respect to the range of $\alpha_i$ as a direct sum $x_i = y_i \oplus z_i$ for $y_i,z_i \in K$ with $y_i$ unitarily equivalent to $x$. The set of all extreme points is denoted $\partial K = \coprod_n (\partial K)_n$.
\end{defn}

We will occasionally be interested in the (classical) extreme points of the compact convex set $K_n$ for some $n$, which we will denote by $\partial K_n$.

We also define a notion of pure point, which more closely resembles the classical notion of extreme point. We are grateful to Bojan Magajna for suggesting a definition that is preserved by affine nc homeomorphisms (see Proposition \ref{prop:preservation-extreme-points}) inspired by \cite{Maj2016}.

\begin{defn}\label{defn:pure-point}
Let $K$ be an nc convex set. We will say that a point $x \in K_n$ is {\em pure} if whenever $x$ is written as a finite nc convex combination $x = \sum \alpha_i^* x_i \alpha_i$ for $\{x_i \in K_{n_i}\}$ and nonzero $\{\alpha_i \in \cM_{n_i,n}\}$ satisfying $\sum \alpha_i^* \alpha_i = 1_n$, then each $\alpha_i$ is a positive scalar multiple of an isometry $\beta_i \in \cM_{n_i,n}$ satisfying $\beta_i^* x_i \beta_i = x$.
\end{defn}

\begin{rem} \label{rem:pure}
A pure point $x \in K_n$ is a (classical) extreme point of the compact convex set $K_n$. However, we will see in the next proposition that a (classical) extreme point of $K_n$ is not necessarily pure. If $x$ is pure, then it cannot be decomposed as a (non-trivial) direct sum, so the corresponding representation $\delta_x : \rC(K) \to \cM_n$ is irreducible. Note however that even if $\delta_x$ is irreducible, it is not necessarily true that $x$ is pure. For example, for any $x \in K_1$, $\delta_x$ is a character on $\rC(K)$, and in particular is irreducible. 
\end{rem}

\begin{prop} \label{prop:extreme-iff-pure-maximal}
Let $K$ be an nc convex set. A point $x\in K$ is extreme if and only if it is both pure and maximal.
\end{prop}

\begin{proof}
Suppose $x$ can be written as a finite nc convex combination $x = \sum \alpha_i^* x_i \alpha_i$ for $\{x_i \in K_{n_i}\}$ and nonzero $\{\alpha_i \in \cM_{n_i,n}\}$ satisfying $\sum \alpha_i^* \alpha_i = 1_n$. The condition that each $\alpha_i$ is a positive scalar multiple of an isometry $\beta_i \in \cM_{n_i,n}$ satisfying $\beta_i^* x_i \beta_i = x$ is equivalent to $x$ being pure. The condition that each $x_i$ decomposes with respect to the range of $\alpha_i$ as a direct sum $x_i = y_i \oplus z_i$ for $y_i,z_i \in K$ with $y_i$ unitarily equivalent to $x$, combined with the preceding condition, is equivalent to the maximality of $x$.
\end{proof}

The next result will be (implicitly) invoked when we apply the dual equivalence between compact nc convex sets and operator systems from Section \ref{sec:categorical-duality}.

\begin{prop} \label{prop:preservation-extreme-points}
Let $K$ and $L$ be nc convex sets and let $\theta : K \to L$ be an affine nc homeomorphism. Then $\theta$ maps pure points in $K$ to pure points in $L$ and maximal points in $K$ to maximal points in $L$. Hence $\theta$ maps extreme points in $K$ to extreme points in $L$.
\end{prop}

\begin{proof}
Let $K$ and $L$ be nc convex sets and let $\theta : K \to L$ be an affine nc homeomorphism. Then there is an inverse affine nc homeomorphism $\theta^{-1} : L \to K$.

Let $x \in K_n$ be a pure point and suppose that $\theta(x)$ can be written as a finite nc convex combination $\theta(x) = \sum \alpha_i^* y_i \alpha_i$ for $\{y_i \in L_{n_i}\}$ and nonzero $\{\alpha_i \in \cM_{n_i,n}\}$ satisfying $\sum \alpha_i^* \alpha_i = 1_n$. Then applying $\theta^{-1}$ to both sides implies $x = \sum \alpha_i^* \theta^{-1}(y_i) \alpha_i$. Since $x$ is pure, each $\alpha_i$ is a positive scalar multiple of an isometry $\beta_i \in \cM_{n_i,n}$ satisfying $\beta_i^* \theta^{-1}(y_i) \beta_i = x$, say $\alpha_i = \gamma_i \beta_i$ for $\gamma_i > 0$. Applying $\theta$ to both sides implies $\beta_i^* y_i \beta_i = \theta(x)$. Hence $\theta(x)$ is pure. 

Now let $z \in K_m$ be a maximal point and let $u \in L_n$ be a dilation of $\theta(z)$. Then there is an isometry $\xi \in \cM_{n,m}$ such that $\theta(z) = \xi^* u \xi$. Applying $\theta^{-1}$ to both sides implies $z = \xi^* \theta^{-1}(u) \xi$. Hence $\theta^{-1}(u)$ is a dilation of $z$. Since $z$ is maximal, $\theta^{-1}(u)$ decomposes with respect to the range of $\xi$ as $\theta^{-1}(u) = z \oplus w$ for some $w \in K$. Since $\theta$ respects direct sums, applying $\theta$ to both sides implies $u$ decomposes with respect to the range of $\xi$ as $u = \theta(z) \oplus \theta(v)$. Hence $\theta(z)$ is maximal.

The fact that $\theta$ maps extreme points in $K$ to extreme points in $L$ now follows from Proposition \ref{prop:extreme-iff-pure-maximal}. 
\end{proof}

\begin{rem}
Say that a compact nc convex set $K$ over a dual operator space $E$ is {\em regularly embedded} if there is an {\em nc hyperplane} $H \subseteq \cM(E)$ of the form
\[
H_n = \{x \in E_n : \theta(x) = \gamma 1_n \}
\] 
for a continuous affine nc map $\theta : E \to \cM$ and a constant $\gamma \in \bR$ such that $K \subseteq H$ and $0_n \notin H_n$ for all $n$. This is a noncommutative analogue of the notion of a regular embedding of a compact convex set (see \cite{Alfsen}*{Chapter 2}).

If $K$ is regularly embedded, then a point $x \in K_n$ is pure in the sense of Definition \ref{defn:pure-point} if and only if whenever $x$ is written as a finite nc convex combination $x = \sum \alpha_i^* x_i \alpha_i$ for $\{x_i \in K_{n_i}\}$ and nonzero $\{\alpha_i \in \cM_{n_i,n}\}$ satisfying $\sum \alpha_i^* \alpha_i = 1_n$, then each $\alpha_i^* x_i \alpha_i$ is a positive scalar multiple of $x$. This is analogous to the definition of a pure unital completely positive map (see e.g. \cite{DK2015}).

To see this, suppose that $\alpha_i^* x_i \alpha_i = \delta_i x$ for $\delta_i > 0$ and let $\beta_i = \delta_i^{-1/2} \alpha_i$. Then $\beta_i^* x_i \beta_i = x$. Applying $\theta$ to both sides implies $\gamma \beta_i^* \beta_i = \gamma 1_n$, i.e. $\beta_i^* \beta_i = 1_n$. Hence $\beta_i$ is an isometry satisfying $\beta_i^* x_i \beta_i = x$.

The canonical affine nc homeomorphism from $K$ to the nc state space of the operator system $\rA(K)$ of affine nc functions on $K$ is a regular embedding of $K$ into the dual operator space $\rA(K)^*$ with respect to the nc hyperplane $H \subseteq \cM(\rA(K)^*)$ defined by
\[
H_n = \{x \in \cM_n(\rA(K)^*) : 1_{\rA(K)}(x) = 1_n \},
\]
where $1_{\rA(K)} \in \rA(K)$ denotes the unit. Hence in this case, the points in $K$ that are pure in the sense of Definition \ref{defn:pure-point} are precisely the points in $K$ that are pure unital completely positive maps.
\end{rem}

\begin{example} \label{ex:arveson-pure-c-star-alg}
Let $A$ be a C*-algebra with nc state space $K$. Arveson \cite{Arv1969}*{Corollary 1.4.3} showed that a point $x \in K_n$ is pure if and only if $x$ is a compression of an irreducible representation of $A$. In particular, if $A$ is commutative so that every irreducible representation of $A$ is a character, then for $n \geq 2$ no point of $K_n$ is pure.
\end{example}

\begin{example}\label{ex:c-star-alg-extreme-pts}
Let $A$ be a C*-algebra with nc state space $K$ so that $A$ is completely order isomorphic to $\rA(K)$. If $x \in K$ is a representation of $A$, then it is clear that $x$ is necessarily maximal. On the other hand, if $x$ is maximal, then by Proposition \ref{prop:criterion-unique-rep-map}, the representation $\delta_x$ is the unique representing map for $x$. Moreover, by Proposition \ref{prop:maximal-reps-factor-through-cmin}, $\delta_x$ factors through $\cmin(\rA(K)) = A$. So $x$ is a representation of $A$. Therefore, $x$ is maximal precisely when it is a representation of $A$. If $x \in K$ is a representation, then Example \ref{ex:arveson-pure-c-star-alg} implies that it is pure if and only if it is irreducible. It follows that the extreme points $\partial K$ of $K$ are precisely the irreducible representations of $A$.
\end{example}

\begin{thm}\label{thm:extreme}
Let $K$ be a compact nc convex set.
A point $x \in K_n$ is an extreme point if and only if the representation $\delta_x : \rC(K) \to M_n$ is both irreducible and the unique representing map for $x$.
\end{thm}

\begin{proof}
If $x$ is extreme, then by Proposition \ref{prop:extreme-iff-pure-maximal} it is pure and maximal. In this case, Remark \ref{rem:pure} implies that $\delta_x$ is irreducible and Proposition \ref{prop:criterion-unique-rep-map} implies that $\delta_x$ is the unique representing map for $x$.

For the converse, suppose that $\delta_x$ is both irreducible and the unique representing map for $x \in K_n$. By Proposition \ref{prop:extreme-iff-pure-maximal}, to show that $x$ is extreme it suffices to show that $x$ is pure and maximal. Proposition \ref{prop:criterion-unique-rep-map} implies that $x$ is maximal.

To see that $x$ is pure, suppose that $x$ can be written as a finite nc convex combination $x = \sum \alpha_i^* x_i \alpha_i$ for $\{x_i \in K_{n_i}\}$ and nonzero $\{\alpha_i \in \cM_{n_i,n}\}$ satisfying $\sum \alpha_i^* \alpha_i = 1_n$. Define a unital completely positive map $\mu : \rC(K) \to \cM_n$ by $\mu = \sum \alpha_i^* \delta_{x_i} \alpha_i$. Then $\mu$ has barycenter $x$, and hence represents $x$. Since $x$ has a unique representing map, this implies $\mu = \delta_x$. Since $\delta_x$ is irreducible, it follows from Example \ref{ex:arveson-pure-c-star-alg} that it is a pure point in the nc state space of $\rC(K)$. Hence each $\alpha_i$ is a scalar multiple of an isometry $\beta_i$ satisfying $\beta_i^* \delta_{x_i} \beta_i = \delta_x$, implying $\beta_i^* x_i \beta_i = x$. Hence $x$ is pure. 
\end{proof}

\begin{example}\label{Ex:commutative_extreme}
Let $C$ be a compact convex set and let $\rA(C)$ denote the function system of continuous affine functions on $C$, considered as an operator subsystem of the C*-algebra $\rC(C)$ of continuous functions on $C$. Let $K$ denote the nc state space of $\rA(C)$, so that $\rA(C)$ is completely order isomorphic to $\rA(K)$. Then $K_1 = C$ and $\cmin(\rA(K)) = \rC(\ol{\partial C})$ (see the beginning of Section \ref{sec:cmin}). We will show that $\partial K = \partial C$.

For $x \in \partial K$, Theorem \ref{thm:extreme} implies that the representation $\delta_x$ is both irreducible and maximal. In this case, Proposition \ref{prop:maximal-reps-factor-through-cmin} implies that $\delta_x$ factors through $\rC(\ol{\partial C})$. Since $\rC(\ol{\partial C})$ is commutative, it follows that $x \in K_1$. Hence $x \in (\partial K)_1$ and it is clear that $x \in \partial C$.

On the other hand, suppose $x \in \partial C$. If $y \in K_n$ dilates $x$, then there is an isometry $\alpha \in \cM_{n,1}$ such that $x = \alpha^* y \alpha$. Define a state $\mu : \rC(K) \to \bC$ by $\mu = \alpha^* \delta_y \alpha$. By Proposition \ref{prop:maximal-reps-factor-through-cmin}, $\delta_y$ factors through $\rC(\ol{\partial C})$. Hence $\mu$ factors through $\rC(\ol{\partial C})$. So by the Riesz-Markov-Kakutani representation theorem, $\mu$ can be identified with a probability measure on $\ol{\partial C}$ with barycenter $x$. Since $x$ is an extreme point in $C$, it follows that $\mu = \delta_x$. Hence $y$ is a trivial dilation of $x$, implying that $x$ is maximal. Since $\delta_x$ is irreducible, it follows from Theorem \ref{thm:extreme} that $x \in (\partial K)_1$. Therefore $\partial K = \partial C$.
\end{example}

\subsection{Existence of extreme points}

The fact that every compact nc convex set has extreme points is highly non-trivial. In fact, it is equivalent to a conjecture of Arveson \cite{Arv1969} about the existence of boundary representations for operator systems, which was open for over 45 years. The conjecture was eventually verified by Arveson himself \cite{Arv2008} in the separable case and by the authors \cite{DK2015} in the general case. 

Let $S$ be an operator system. An irreducible representation $\pi : \cmin(S) \to \B(H)$ is said to be a boundary representation for $S$ if whenever $\phi : \cmin(S) \to \B(H)$ is a unital completely positive map satisfying $\phi|_S = \pi|_S$, then $\phi = \pi$. In other words, $\pi$ is a boundary representation for $S$ if the restriction $\pi|_S$ has a unique extension to a unital completely positive map on $\cmin(S)$. The next result is an immediate consequence of Theorem \ref{thm:extreme}.

\begin{cor}\label{cor:extreme}
Let $S$ be an operator system with nc state space $K$. The extreme points $\partial K$ of $K$ are precisely the restrictions of boundary representations of $S$.
\end{cor}

The next result is a restatement of \cite{DK2015}*{Theorem 2.4}. It asserts that the existence of extreme points in a compact nc convex set, 

\begin{restatable}{thm}{PureDilationTheorem} \label{thm:dilation-pure-to-extreme}
Let $K$ be a compact nc convex set. Then every pure point in $K$ has an extreme dilation.
\end{restatable}

We will give a new proof of Theorem \ref{thm:dilation-pure-to-extreme} using ideas from this paper in Section \ref{sec:max-in-diln-order}.

The next result is a restatement of \cite{DK2015}*{Theorem 3.1}. It asserts that the extreme points in a compact nc convex set completely norm the affine nc functions on the set. 

\begin{thm} \label{thm:extreme-pts-norming}
Let $K$ be a compact nc convex set. For every $n \in \bN$ and $a \in \cM_n(\rA(K))$ there is an extreme point $x \in \partial K$ such that $\|a\| = \|a(x)\|$.
\end{thm}

The next result is a restatement of \cite{DK2015}*{Theorem 3.4}. It asserts that the extreme points in a compact nc convex set give rise to a faithful representation of the minimal C*-algebra of the corresponding operator system of continuous affine nc functions.

\begin{thm} \label{thm:extreme-pts-complete-order-monomorphism}
Let $K$ be a compact nc convex set. 
Define a Hilbert space $H = \oplus_n \oplus_{x \in (\partial K)_n} H_n$ and let $\pi : \rC(K) \to \B(H)$ denote the representation defined by $\pi = \oplus_{x \in \partial K} \delta_x$. Then the restriction $\pi|_{\rA(K)}$ is a complete order monomorphism and the image $\pi(\rC(K))$ is isomorphic to $\cmin(\rA(K))$.
\end{thm}

\subsection{Accessibility of extreme points} \label{sec:accessibility-extreme-points}

Let $K$ be a compact nc convex set. It was shown in \cite{DK2015}*{Theorem 3.4} that the map restricting functions in $\rA(K)$ to $\partial K$ is completely isometric. In particular, this implies that the minimal C*-algebra $\cmin(\rA(K))$ is completely determined by this restriction. However, the set of points in $\partial K$ that can be obtained by dilating pure points in finite components of $K$ as in Theorem \ref{thm:dilation-pure-to-extreme} can be a proper subset of $\partial K$. The corresponding boundary representations are called accessible in \cite{Kri2019}. We thank Ben Passer for providing some examples, of which the following is a variant.

\begin{example}\label{Ex:Passer}
Let $\bF_2 = \langle u,v \rangle$ denote the free group on two generators and let $S = \spn\{1,u,u^*,v,v^*\}$ in $\ca(\bF_2 )$. Let $K$ denote the nc state space of $S$, so that $\rA(K)$ is completely order isomorphic to $S$. Let $a,b \in \rA(K)$ denote the continuous affine nc functions corresponding to $u,v \in S$ respectively. 

A point $x \in K_m$ is completely determined by the pair of contractions $(a(x),b(x))$. If $(a(x),b(x))$ is a pair of unitaries, then $x$ is maximal, since it does not have non-trivial dilations. Since unitaries are extreme points in the unit ball of $\cM_m$, it follows that if $(a(x),b(x))$ is an irreducible pair of unitaries, in the sense that they do not have any common non-trivial invariant subspaces, then $x$ is an extreme point. There are many such examples for any $m$.

For $m \in \bN$, the extreme points of the unit ball of $\cM_m$ are precisely the unitaries. Hence if $x \in K_m$ is pure, then $(a(x),b(x))$ is an irreducible pair of unitaries, implying that $x$ is an extreme point. In particular, pure points in $K_m$ do not have non-trivial pure dilations.

It follows that for $m$ infinite, an extreme point in $K_m$ cannot be obtained as the limit of an increasing sequence of finite dimensional pure compressions.
\end{example}

\subsection{Noncommutative Krein-Milman theorem} \label{sec:krein-milman}

In this section we will prove a noncommutative analogue of the Krein-Milman theorem asserting that every compact nc convex set is the closed nc convex hull of its extreme points. We will also prove an analogue of Milman's partial converse to the Krein-Milman theorem.

\begin{defn}
For a dual operator space $E$ and a subset $X \subseteq \cM(E)$, the {\em closed nc convex hull} $\ol{\ncconv}(X)$ of $X$ is the intersection of all closed nc convex sets over $E$ that contain $X$. Equivalently, the closed nc convex hull of $X$ is the closure of the set of all nc convex combinations of elements in $X$. 
\end{defn}

\begin{thm}[Noncommutative Krein-Milman theorem] \label{thm:nc-krein-milman} \strut \\
A compact nc convex set is the closed nc convex hull of its extreme points.
\end{thm}

\begin{proof}
Let $L$ be a compact nc convex set. By the results in Section \ref{sec:categorical-duality}, we can identify $L$ with the nc state space of the operator system $\rA(L)$ of continuous affine nc functions on $L$. In particular, $L$ is a compact nc convex set over the dual operator space $\rA(L)^*$. 

Let $K$ denote the closed nc convex hull of the extreme points of $L$. Clearly $K \subseteq L$. Suppose for the sake of contradiction there is $n$ and $y \in L_n \setminus K_n$. Then by Corollary \ref{cor:separation}, there is a normal completely bounded linear map $\phi : A(L)^* \to \cM_n$ and self-adjoint $\gamma \in \cM_n$ such that
\[
\re \phi_n(y) \not \leq \gamma \otimes 1_n \quad \text{but} \quad \re \phi_p(x) \leq \gamma \otimes 1_p
\]
for every $p$ and $x \in K_p$.

Since $\phi$ is normal, there is $a \in \cM_n(\rA(L))$ such that $\phi_n(x) = a(x)$ for all $x \in L_n$. Let $b = (a + a^*)/2$. Then from above, $b(y) \not \leq \gamma \otimes 1_n$ but $b(x) \leq \gamma \otimes 1_p$ for all $p$ and $x \in K_p$.

Let $H = \oplus_n \oplus_{x \in (\partial L)_n} H_n$ and let $\pi : \rC(L) \to \B(H)$ denote the representation defined by $\pi = \oplus_{x \in \partial L} \delta_x$. Then by Theorem \ref{thm:extreme-pts-complete-order-monomorphism}, the restriction $\pi|_{\rA(L)}$ is a complete order monomorphism. Hence from above, $\pi(b) \leq \gamma \otimes 1_H$.  It follows that $b \leq \gamma \otimes 1_{\rA(L)}$, so in particular $b(y) \leq \gamma \otimes 1_n$, giving a contradiction.
\end{proof}

The next result is a noncommutative analogue of Milman's partial converse to the Krein-Milman theorem.

\begin{thm}\label{thm:milman-converse}
Let $K$ be a compact nc convex set. Let $X \subseteq K$ be closed in the sense of Proposition~\ref{prop:convergent-nets} and closed under compressions, meaning that 
$\alpha^* X_n \alpha \subseteq X_m$ for every isometry $\alpha \in \cM_{n,m}$. If the closed nc convex hull of $X$ is $K$, then $\partial K \subseteq X$.
\end{thm}

\begin{proof}
Let $L$ denote the nc state space of $\rC(K)$ and let 
\[ Z = \ol{\ncconv} \{\delta_x : x \in X\} \subseteq L .\]
Since $K = \ol{\ncconv}(X)$ and the barycenter map from $L$ onto $K$ is continuous and affine, it follows that for every $y \in K$ there is $\mu \in Z$ with barycenter $y$.

Fix $y \in \partial K$ and $\mu \in Z$ with barycenter $y$. Since $y$ is extreme, Theorem \ref{thm:extreme} implies that $\mu = \delta_y$. For $n \in \bN$, there is a standard trick to identify an $n$-dimensional compression of $\delta_y$ with a state supported on the representation $\delta_y \otimes \id_n$ on $\cM_n(\rC(K))$. Since $\delta_y$ is irreducible, so is $\delta_y \otimes \id_n$. In particular, every state supported on $\delta_y \otimes \id_n$ is a pure state. By construction $\ker \delta_y \supseteq \cap_{x \in X} \ker \delta_x$, so $\ker \delta_y \otimes \id_n \supseteq \cap_{x \in X} \ker \delta_x \otimes \id_n$. Hence by \cite{Dixmier}*{Proposition 3.4.2 (ii)}, every state that factors through $\delta_y \otimes \id_n$ is a limit of a net of pure states, each of which is supported on some $\delta_x \otimes \id_n$ for $x \in X$. 

This translates to the statement that every $n$-dimensional compression of $\delta_y$ is a point-weak* limit of compressions of $\{\delta_x : x \in X \}$. Since the barycenter map is continuous and affine, and since $X$ is both closed and closed under compressions, arguing as in the proof of Proposition \ref{prop:convergent-nets} implies that $y \in X$. 
\end{proof}

\begin{rem}
Simple examples demonstrate that the assumption that $X$ is closed under compressions is necessary. For instance, consider Example~\ref{Ex:commutative_extreme}. Fix any point $y \in K_n$ and let $X$ denote the closure of the set $\{ x \oplus y : x \in \partial K  \}$. 
This is contained in $K_{n+1}$, so it is disjoint from $\partial K = \partial C \subset K_1$. Nevertheless, it follows from Theorem \ref{thm:nc-krein-milman} that $K$ is the closed nc convex hull of $X$.

This trick fails for certain infinite dimensional examples like the Cuntz system of Examples~\ref{ex:ncfunctions-not-determined-by-finite-levels} and \ref{ex:cuntz}. It follows from a version of Voiculescu's noncommutative Weyl-von Neumann theorem (see \cite{BrownOzawa}*{Corollary 1.7.7}) that for any $y\in K_n$ with $n < \infty$,  the point-norm closure of $\{ x \oplus y : x \in \partial K  \}$ contains all representations of $\O_n$ (restricted to $S$). 
So in particular, the point-weak-$*$ closure contains $\partial K$.
\end{rem}

\subsection{Extreme points and the minimal C*-algebra} \label{sec:extreme-pts-minimal-c-star-alg}

In this section we relate the extreme points of a compact nc convex set to the nc state space of the corresponding minimal C*-algebra. 

For a compact nc convex set $K$, the universal properties of $\rC(K)$ and $\cmin(\rA(K))$ imply the existence of a unique surjective homomorphism $q : \rC(K) \to \cmin(\rA(K))$ such that $q|_{\rA(K)} = \iota$, where $\iota : \rA(K) \to \cmin(\rA(K))$ denotes the canonical unital complete order embedding. It follows from the results in Section \ref{sec:categorical-duality} that the nc state space of $\cmin(\rA(K))$ is affinely homeomorphic to a closed subset of nc states on $\rC(K)$; namely, the set of nc states on $\rC(K)$ that factor through $\cmin(\rA(K))$.

For $x \in \partial K$, Proposition \ref{prop:criterion-unique-rep-map} implies that the corresponding representation $\delta_x$ factors through $\cmin(\rA(K))$. Hence the extreme boundary $\partial K$ corresponds to a subset of the irreducible representations of $\cmin(\rA(K))$. However, we have seen that, as in the classical setting, this subset will often be proper (see e.g. Example \ref{Ex:free semicircular}).

Motivated by the classical setting, we can think of the set of irreducible representations of $\cmin(\rA(K))$ as the Shilov boundary of $\rA(K)$, and $\partial K$ as the Choquet boundary of $\rA(K)$. The next result describes the precise relationship between these sets. It can be viewed as a noncommutative analogue of the fact that in the classical setting, the Shilov boundary is the closure of the Choquet boundary.

\begin{thm} \label{thm:nc-shilov-boundary-closure-nc-choquet-boundary}
Let $K$ be a compact nc convex set and let $L$ denote the nc state space of $\cmin(\rA(K))$, identified with the set of nc states on $\rC(K)$ that factor through $\cmin(\rA(K))$. Then $L$ is the closed nc convex hull of the set $\{\delta_x \in \partial K \}$. 
\end{thm}

\begin{proof}
Let $X = \{\delta_x : x \in \partial K \}$. For $x \in \partial K$, Proposition \ref{prop:maximal-reps-factor-through-cmin} implies that the corresponding representation $\delta_x$ factors through $\cmin(\rA(K))$. Hence the representation $\sigma :=  \oplus_{x \in \partial K} \delta_x$ factors through $\cmin(\rA(K))$, and we can view it as a representation of $\cmin(\rA(K))$. Theorem \ref{thm:extreme-pts-norming} implies that the restriction $\sigma|_{\rA(K)}$ is a unital complete order embedding. Hence by the universal property of $\cmin(\rA(K))$, $\sigma$ is faithful.

Let $y \in K$ be a point such that the corresponding representation $\delta_y$ factors through $\cmin(\rA(K))$. Then
\[
\ker \delta_y \supseteq \bigcap_{x \in \partial K} \ker \delta_x = \ker \sigma.
\]
Hence an argument similar to the proof of Theorem \ref{thm:milman-converse} implies that $\delta_y$ is contained in the closed nc convex hull of $X$. 
Every irreducible representation of $\cmin(\rA(K))$ is of this form, and by Example \ref{ex:c-star-alg-extreme-pts}, 
these are precisely the extreme points of $L$. 
Therefore, it follows from Theorem \ref{thm:nc-krein-milman} that $L$ is the closed nc convex hull of $X$.
\end{proof}

\subsection{Examples} \label{sec:examples}

In this section we will illustrate the results we have obtained so far with some examples.

\begin{example} \label{ex:compacts}
Let $S \subseteq \M_n$ for $n < \infty$ be an irreducible operator system, so that $\ca(S) = \M_n$. Since $\M_n$ is simple, $\cmin(S) = \M_n$. Let $K$ denote the nc state space of $S$. For an extreme point $x \in (\partial K)_n$, the corresponding representation $\delta_x : \rC(K) \to \cM_n$ is an irreducible representation of $\M_n$. Since every irreducible representation of $\M_n$ is equivalent to the identity representation $\id : \M_n \to \M_n$ and $\partial K$ is closed under unitary equivalence, it follows that
\[
\partial K = (\partial K)_n = \{\alpha \id \alpha^* : \alpha \in \rU(\M_n)\}.
\]
\end{example}

\begin{example}\label{ex:cuntz}
For $d\geq 2$, the Cuntz algebra $\O_d$ is the universal C*-algebra generated by $d$ elements $s_1,\ldots,s_d$ satisfying the Cuntz relations
\[
\sum_{i=1}^d s_i s_i^* = 1, \quad s_i^* s_j = \delta_{ij} 1.
\]
The algebra $\O_d$ is simple and infinite dimensional.

Let $S = \spn \{1,s_1,s_1^*,\ldots,s_d,s_d^* \}$. Then $\ca(S) = \O_d$, so by the simplicity of $\O_d$, $\cmin(S) = \O_d$. Let $K$ denote the nc state space of $S$.

Every point $x \in K_n$, is completely determined by the row contraction $X = [x(s_1),\ldots,x(s_d)] \in \cM_n^d$. We say that $X$ is irreducible if it cannot be decomposed as a (non-trivial) direct sum. Note that $X$ is irreducible if and only if the representation $\delta_x : \rC(K) \to \cM_n$ is irreducible. We say that $X$ is a coisometry if $XX^* = \sum_{i=1}^d x(s_i)x(s_i)^* = 1_n$. If, in addition, $X^*X = \big[ x(s_i)^* x(s_j) \big] = 1_d \otimes 1_n$, then we say that $X$ is a row unitary. Note that $X$ is a row unitary if and only if $x(s_1),\ldots,x(s_d)$ satisfy the Cuntz relations.

Suppose that the representation $\delta_x$ is the unique representing map for $x$. Proposition \ref{prop:criterion-unique-rep-map} and Proposition \ref{prop:maximal-reps-factor-through-cmin} imply that $\delta_x$ factors through $\cmin(S) = \O_d$, so $x(s_1),\ldots,x(s_d)$ satisfy the Cuntz relations. Hence $X$ is a row unitary. In particular, if $x$ is extreme, then $X$ is a row unitary.

Conversely, suppose that $X$ is a row unitary and let $\mu : \rC(K) \to \cM_n$ be a unital completely positive map with barycenter $x$. Then by the Kadison-Schwartz inequality,
\[
1_n = \sum_{i=1}^n x(s_i) x(s_i)^* = \sum_{i=1}^n \mu(s_i) \mu(s_i)^* \leq \sum_{i=1}^n \mu(s_i s_i^*) = 1_n.
\]
This implies that $\mu(s_i s_i^*) = x(s_i) x(s_i)^*$ for each $i$, so $S$ belongs to the multiplicative domain of $\mu$. Hence $\mu = \delta_x$, implying $\delta_x$ is the unique representing map for $x$.

It now follows from Theorem \ref{thm:extreme} that $x$ is extreme if and only if $X$ is an irreducible row unitary. Hence there is a correspondence between points in the extreme boundary $\partial K$ and irreducible representations of $\cmin(S) = \O_d$. In particular, $K$ has no finite dimensional extreme points.
\end{example}

\begin{example}\label{Ex:free semicircular}
Let $a_1,a_2 \in M$ be freely independent semicircular (self-adjoint) operators contained in a von Neumann algebra of type $II_1$. For example, letting $s_1,s_2$ denote the generators of the Cuntz algebra $\O_2$ as in Example \ref{ex:cuntz}, we can take $a_i = s_i + s_i^*$ for each $i$. Consider the operator system $S = \spn\{1,a_1,a_2\}$ with nc state space $K$.

Let $A = \ca(S)$. Then $A$ is simple by \cite{Dyk1999}, so $\cmin(S) = A$. Furthermore, since $A$ is separable, Voiculescu's theorem \cite{Voi1976} (see \cite{Davidson}*{Corollary II.5.6}) implies that all representations of $A$ are approximately unitarily equivalent. In other words, every separable representation of $A$ is a point-norm limit of representations that are all unitarily equivalent to any other separable representation. For our purposes, the interesting thing about the operator system $S$ is that not every irreducible representation of $A$ restricts to an extreme point in $\partial K$. In particular, $\partial K$ is not closed in the point-norm topology, and thus is not closed in the point-weak* topology.

We now consider specific representations of $\O_2$ belonging to the class of atomic representations classified in \cite{DP1999}*{Section 3}. Consider the Fock space $F_2 = l^2(\bF_2^+)$, where $\bF_2^+$ denotes the free semigroup on $\{1,2\}$, with $\epsilon \in \bF_2^+$ denoting the empty word. The canonical orthonormal basis for $F_2$ is $\{\delta_w : w \in \bF_2^+\}$. Let $L_1, L_2 \in \B(F_2)$ denote the isometries defined by $L_i \delta_w = \delta_{iw}$ for $i=1,2$ and $w \in \bF_2^+$. Let $\H = \bC \oplus F_2$. 

For $\lambda \in \bT$, consider the representation $\pi_\lambda : \O_2 \to \H$ defined by
\[
\pi_\lambda(s_1) = \begin{bmatrix}\lambda&0\\0&L_1\end{bmatrix}, \quad
\pi_\lambda(s_2) = \begin{bmatrix}0&0\\ \delta_\epsilon&L_2\end{bmatrix}.
\]
Let $\sigma_\lambda = \pi_\lambda|_A$. 

Identifying $H_{\aleph_0}$ with $\H$, we obtain $x_\lambda \in K_{\aleph_0}$ by setting $x_\lambda = \sigma_\lambda|_S$. The corresponding representation $\delta_{x_\lambda} : \rC(K) \to \cM_{\aleph_0}$ satisfies $\delta_{x_\lambda} = \sigma_\lambda \circ q$, where $q : \rC(K) \to \cmin(S) = A$ denotes the canonical quotient homomorphism. We will show that $\delta_{x_\lambda}$ is irreducible for all $\lambda$, but that $x_\lambda$ is an extreme point if and only if $\lambda = \pm 1$.

Since $\pi_\lambda$ is an atomic representation of $\O_2$ corresponding to the primitive word `$1$', it is irreducible by \cite{DP1999}. Write $\lambda = r + is$. Let $p$ be a projection in $\sigma_\lambda(A)'$. Let $\eta = 1 \oplus 0 \in \H$. Replace $p$ by $p^\perp$ if necessary to ensure that $p \eta \ne 0$. We will show that $p = 1_{\H}$.

First note that $\sigma_\lambda(a_1) \eta = 2r \eta$. Hence
\[ \sigma_\lambda(a_1) p \eta = p \sigma_\lambda(a_1) \eta = 2r p \eta .\]
Since $L_1+L_1^*$ does not have any eigenvectors, this implies $p \eta = \eta$. Thus $\sigma_\lambda(a_2) \eta = 0 \oplus \delta_\epsilon$ lies in $\Ran(p)$.
It can now be shown by induction on word length that $0 \oplus \delta_w \in \Ran(p)$ for all words $w \in \bF_2^+$. Hence $p=1_{\H}$ and $\sigma_\lambda$ is irreducible. It follows that $\delta_{x_\lambda}$ is irreducible.

Now suppose $\lambda = 1$. Since $\delta_{x_1}$ is irreducible, Theorem \ref{thm:extreme} implies that $x_1$ is an extreme point if and only if $x_1$ has a unique representing map. By Proposition \ref{prop:maximal-reps-factor-through-cmin}, this is the case if and only if the only representing map of $x_1$ that factors through $\cmin(S) = A$ is $\delta_{x_1}$. Equivalently, $x_1$ is an extreme point if and only if whenever $\phi : A \to \B(\H)$ is a unital completely positive map satisfying $\phi|_{S} = \sigma_1|_S$, then $\phi = \sigma_1$.

Let $\phi : A \to \B(\H)$ be a unital completely positive map satisfying $\phi|_{S} = \sigma_1|_S$. By Arveson's extension theorem we can extend $\phi$ to a unital completely positive map $\phi_1 : \O_2 \to \B(\H)$. Let $\rho_1 : \O_2 \to \B(\K)$ be a minimal Stinespring representation for $\phi_1$, so that $\phi_1 = P_{\H} \rho_1|_{\H}$. 

Each $\rho_1(s_i)$ is an isometry, and
\begin{align*}
2\eta &= \sigma_1(a_1) \eta = \phi_1(a_1) \eta \\&
= P_{\H} \rho_1(a_1) \eta = P_{\H}(\rho_1(s_1) + \rho_1(s_1)^*) \eta.
\end{align*}
It follows that $\rho_1(s_1) \eta = \eta$ and $\rho_1(s_1)^* \eta = \eta$. In particular, $\eta$ lies in $\Ran(\rho_1(s_1)) = \Ran(\rho_1(s_2))^\perp$, so
\[
\rho_1(a_2) \eta = (\rho_1(s_2) + \rho_1(s_2)^*) \eta = \rho_1(s_2) \eta.
\]
This also shows that $\eta$ is coinvariant for $\rho_1(s_1)$ and $\rho_1(s_2)$. Hence
\[ \delta_\epsilon = \sigma_1(a_2) \eta = \phi_1(a_2) \eta = P_{\H} \rho_1(a_2) \eta = P_{\H} \rho_1(s_2) \eta. \]
Therefore $\rho_1(s_2) \eta = \delta_\epsilon$.

It now follows that $\delta_\epsilon$ is a wandering vector for the tuple $(\rho_1(s_1),\rho_1(s_2))$ in the sense of \cite{DP1999}. To see this, note that for $w \in \bF_2^+$,
\[ \ip{ \rho_1(s_w) \delta_\epsilon, \delta_\epsilon} = \ip{\delta_\epsilon, \rho_1(s_w)^* \delta_\epsilon}. \]
If $w = 1v$ for $v \in \bF_2^+$, then $\rho_1(s_w)^* \delta_\epsilon = \rho_1(s_v)^* \rho_1(s_1)^* \delta_\epsilon =0$. Otherwise, if
$w = 2v$, then $\rho_1(s_w)^* \delta_\epsilon = \rho_1(s_v)^* \rho_1(s_2)^*\delta_\epsilon = \rho_1(s_v)^* \eta \in \bC \oplus 0$.
Either way, the inner product vanishes.

An easy induction argument now shows that $\rho_1(s_i) \delta_w = \sigma_1(s_i) \delta_{w}$ and $\rho_1(s_i)^* \delta_w = \sigma_1(s_i)^* \delta_w$ for all $w \in \bF_2^+$. This implies $\bC \oplus F_2$ is invariant for $\rho_1(\O_2)$, and hence that $\rho_1 = \sigma_1 \oplus \rho_1'$ for some representation $\rho_1' : \O_2 \to \B(\K \ominus \H)$. Therefore, $\phi = \sigma_1$, and $x_1$ is an extreme point. A similar argument works when $\lambda = -1$, so $x_{-1}$ is also an extreme point.

Now suppose $\lambda \ne \pm1$. Let $\L = \bC \oplus (\bC \oplus F_2) \oplus F_2$ and consider the representation $\tau_\lambda : \O_2 \to \B(\L)$ defined by
\[
\tau_\lambda(s_1) 
= \left[ \begin{array}{c|cc|c}0&0&0&0\\
\hline 
s&r&0&0\\
0&0&L_1&0\\
\hline
-r\delta_\epsilon & s\delta_\epsilon & 0 & L_1\end{array}\right]
\qand
\tau_\lambda(s_2) 
= \left[\begin{array}{c|cc|c}
1&0&0&0\\
\hline
0&0&0&0\\
0&\delta_\epsilon&L_2&0\\
\hline
0&0&0&L_2\end{array}\right].
\]
Define $\psi_\lambda : \O_2 \to \B(H)$ by $\psi_\lambda(a) = P_{\H} \tau_\lambda(a)|_{\H}$. Then $\psi_\lambda|_A = \sigma_\lambda|_A$, however
\[ 
\psi_\lambda(a_1^2) - \sigma_\lambda(a_1^2) = s^2 \eta \eta^*.
\]
So $\psi_\lambda \ne \sigma_\lambda$.
Hence in this case $x_\lambda$ is not an extreme point.
\end{example}

\section{Noncommutative convex functions}

\subsection{Convex functions and convex envelopes}

The convex structure of a compact convex set $C$ gives rise to the notion of convexity for a function on $C$. The Stone-Weierstrass theorem for lattices implies that the convex functions span a dense subset of the C*-algebra $\rC(C)$ of continuous functions on $C$. We saw in Section \ref{sec:classical-min-max-cstar-alg} that $\rC(C)$ is the maximal commutative C*-algebra generated by the function system $\rA(C)$ of continuous affine functions on $C$ in a certain precise sense. There is another important idea connecting $\rC(C)$ to $\rA(C)$ for which the convex structure of $C$ is essential.

For $f \in \rC(C)$, the convex envelope $\env{f}$ of $f$ is the best approximation of $f$ from below by a convex lower semicontinuous function. It is defined by
\[
\env{f} = \sup \{a \in \rA(C) : a \leq f \}.
\]
The function $f$ is convex if and only if $\bar{f} = f$.

There is also a geometric definition which is more readily generalized. Let $\operatorname{epi}(f) = \{ (x,t) : x \in C,\ t \ge f(x) \}$. Then
\[ \operatorname{epi}(\env{f}) = \ol{\conv}(\operatorname{epi}(f) ) \] 
is the closed convex hull of the epigraph of $f$.

One explanation for the importance of the convex envelope is that it encodes information about the set of representing measures of a point. Specifically, if $C$ is a compact convex set and $f \in \rC(C)$ is a continuous function with convex envelope $\env{f}$, then for $x \in C$,
\[
\env{f}(x) = \inf_\mu \int_C f\, d\mu,
\]
where the infimum is taken over all probability measures $\mu$ with barycenter $x$. This infimum is attained. Moreover, the measure $\mu$ is supported on the extreme boundary $\partial C$ in an appropriate sense if and only if
\[
\int_C f\, d\mu = \int_C \env{f}\, d\mu
\]
for every $f \in \rC(C)$. We will revisit this characterization in Section \ref{sec:orders}.

\subsection{Noncommutative convex functions} \label{sec:convex-nc-fcns}

In this section we will introduce a notion of convexity for nc functions. We will need to consider matrices of bounded nc functions. For a compact nc convex set $K$ and $f = (f_{ij}) \in \cM_n(\rB(K))$, we view $f$ as a function $f : K \to \cM_n(\cM)$ defined by $f(x) = (f_{ij}(x))$ for $x \in K$. Note that $f$ is graded. 
It respects direct sums in the sense that 
\[
 f(x_1 \!\oplus\! x_2) = \big[ f_{ij}(x_1 \!\oplus\! x_2) \big] = \big[ f_{ij}(x_1) \!\oplus\! f_{ij}(x_2) \big] \simeq \big[ f_{ij}(x_1) \big] \!\oplus\! \big[ f_{ij}(x_2) \big] .
\]
The last identification is known as the canonical shuffle.
Finally $f$ is unitarily equivariant in the sense that if $u \in \B(H_n)$ and $x \in K_n$, then 
\[
 f(uxu^*) = \big[ f_{ij}(uxu^*) \big] = \big[ u f _{ij}(x) u^*) \big] = (u \otimes 1_n) f(x) (u \otimes 1_n)^* .
\]
We will refer to $f$ as an nc function on $K$. Say that $f$ is {\em self-adjoint} if $f(x) \in \cM_n(\cM_k)_{sa}$ for all $k$ and all $x \in K_k$. 

\begin{defn} \label{defn:convex-nc-fcn}
Let $K$ be a compact nc convex set and let $f  \in \cM_n(\rB(K))$ be self-adjoint bounded nc function. The {\em epigraph} of $f$ is the subset $\epi(f) \subseteq \coprod_m K_m \times \cM_n(\cM_m)$ defined by
\[
\epi_m(f) = \{(x, \alpha) \in K_m \times \cM_n(\cM_m) : x \in K_m \text{ and } \alpha \geq f(x) \}.
\]
We will say that $f$ is {\em convex} if $\epi(f)$ is an nc convex set, and that $f$ is {\em lower semicontinuous} if $\epi(f)$ is closed.
\end{defn}

\begin{rem} \label{rem:simpler-criterion-convexity}
For a self-adjoint bounded nc function $f \in \cM_n(\rB(K))$, the fact that $f$ is graded and respects direct sums implies that $\epi(f)$ is an nc convex set if and only if
\[
f(\alpha^* x \alpha) \leq (1_n \otimes \alpha^*) f(x) (1_n \otimes \alpha) 
\]
for every $m$, every $x \in K_m$ and every isometry $\alpha \in \cM_{m,l}$.
\end{rem}

This next proposition shows that scalar convexity of an nc function implies nc convexity.
This is a higher dimensional analogue of the Hansen-Pedersen result \cite{HanPed2003}*{Theorem 2.1} for an interval.
See Example~\ref{ex:hansen-pedersen}.

\begin{prop} \label{prop:convex-each-level}
Let $K$ be a compact nc convex set and let $f \in \cM_n(\rB(K))$ be a self-adjoint bounded nc function. Then $f$ is convex if and only if
\begin{equation}
f(\lambda x + (1-\lambda) y) \leq \lambda f(x) + (1-\lambda) f(y) \label{eq:convex-each-level-1}
\end{equation}
for all $m$, all $x,y \in K_m$ and all $\lambda \in [0,1]$.
\end{prop}

\begin{proof}
Suppose $f$ is convex. For $x,y \in K_m$ and $\lambda \in [0,1]$, the fact that (\ref{eq:convex-each-level-1}) holds follows from Remark \ref{rem:simpler-criterion-convexity} and the factorization
\[
\lambda x + (1-\lambda) y =
\left[ \begin{matrix}
\sqrt{\lambda} 1_m & \sqrt{1-\lambda} 1_m \\
\end{matrix} \right]
\left[ \begin{matrix}
x & 0 \\
0 & y \\
\end{matrix} \right]
\left[ \begin{matrix}
\sqrt{\lambda} 1_m \\
\sqrt{1-\lambda} 1_m
\end{matrix} \right].
\]

Conversely, suppose $f$ satisfies (\ref{eq:convex-each-level-1}). By Remark \ref{rem:simpler-criterion-convexity}, to show that $f$ is convex, it suffices to show that for $y \in K_m$ and an isometry $\alpha \in \cM_{m,l}$, $f(\alpha^* y \alpha) \leq (1_n \otimes \alpha^*) f(y) (1_n \otimes \alpha)$. Define a unitary $\beta \in \cM_m$ by $\beta = \alpha \alpha^* - (1 - \alpha \alpha^*)$. Decompose $H_m = \Ran(\alpha \alpha^*) \oplus \Ran(1_m - \alpha \alpha^*)$ and identify $\Ran(\alpha \alpha^*)$ and $\Ran(1_m - \alpha \alpha^*)$ with $H_l$ and $H_{m-l}$ respectively. Then we can write
\[
y = \begin{bmatrix}
x & * \\
* & z
\end{bmatrix},
\qquad
\beta = 
\begin{bmatrix}
1_l & 0 \\
0 & -1_{m-l}
\end{bmatrix} , \qand
\alpha = 
\begin{bmatrix}
\gamma \\
0
\end{bmatrix} 
\]
for some $x \in K_l$, $z \in K_{m-l}$, and $\gamma$ is unitary. 
Observe that by unitary equivariance and equation  (\ref{eq:convex-each-level-1}),
\begin{align*}
 f(\alpha^* y \alpha) &= f(\gamma^*x \gamma) = (1_n \otimes \gamma)^* f(x) (1_n \otimes \gamma) \\&
 =  (1_n \otimes \alpha)^* \big( f(x) \oplus f(z) \big) (1_n \otimes \alpha) \\&
 = (1_n \otimes \alpha)^* f(x \oplus z)  (1_n \otimes \alpha) \\&
 = (1_n \otimes \alpha)^* f\big( \tfrac{1}{2}(y + \beta^* y \beta) \big)  (1_n \otimes \alpha) \\&
 \le (1_n \otimes \alpha)^*  \tfrac{1}{2}\big( f(y)+ f(\beta^* y \beta) \big)  (1_n \otimes \alpha) \\&
 = (1_n \otimes \alpha)^*  \tfrac{1}{2}\big( f(y)+ \beta^*f(y)\beta \big)  (1_n \otimes \alpha) \\&
 = (1_n \otimes \alpha)^* f(y) (1_n \otimes \alpha) .  
\end{align*}
Therefore $f$ is nc convex.
\end{proof}

\begin{example} \label{ex:affine is convex}
Let $K$ be a compact nc convex set. It is easy to check that every continuous self-adjoint affine nc function $a \in \cM_n(\rA(K))$ is convex and lower semicontinuous.
\end{example}

\begin{example} \label{ex:hansen-pedersen} 
Fix a compact interval $I \subseteq \bR$. Recall that a continuous real-valued function $f \in \rC(I)$ is operator convex if
\[
f(\lambda \alpha + (1-\lambda) \beta) \leq \lambda f(\alpha) + (1-\lambda) f(\beta)
\]
for all $n$, all self-adjoint $\alpha,\beta \in (\cM_n)_{sa}$ with spectrum in $I$ and all $\lambda \in [0,1]$.

Define an nc convex set $K$ by setting
\[
K_n = \{\alpha \in (\cM_n)_{sa} : \sigma(\alpha) \subseteq I \},
\]
where $\sigma(\alpha)$ denotes the spectrum of $\alpha$. Note that $K_1 = I$. A continuous real-valued function $f \in \rC(I)$ determines a continuous self-adjoint nc function in $\rC(K)$ by the continuous functional calculus, while on the other hand, a continuous self-adjoint nc function in $\rC(K)$ restricts to a continuous real-valued function in $\rC(I)$. It follows immediately from Proposition \ref{prop:convex-each-level} that a self-adjoint nc function $f \in \rC(K)$ is convex if and only if it restricts to an operator convex function in $\rC(I)$.

This fact, that scalar operator convexity on an interval implies nc convexity in the sense of Definition \ref{defn:convex-nc-fcn}, is essentially the Hansen-Pedersen-Jensen inequality \cite{HanPed2003}*{Theorem 2.1}. 
\end{example}

\begin{example} \label{ex:single-op-convex-function}

Let $K$ be a compact nc convex set. If $[c,d] \subseteq \bR$ is a compact interval and $f \in \rC([c,d])$ is an operator convex function on $[c,d]$, then for self-adjoint $a \in \rA(K)$ satisfying $c 1_{\rA(K)} \leq a \leq d 1_{\rA(K)}$,  it follows as in Example \ref{ex:affine is convex} and Remark~\ref{rem:simpler-criterion-convexity} that $f(a) \in \rC(K)$ is convex.
\end{example}

\begin{rem}
There is also a natural notion of concave nc function in the noncommutative setting. However, a self-adjoint bounded nc function $f \in \cM_n(\rB(K))$ is convex if and only if $-f$ is concave, so there is no disadvantage to working only with convex functions.
\end{rem}

For a compact nc convex set $K$, we do not know if convex nc functions in $\rC(K)$ always span a dense subset of $\rC(K)$. However, the next result will be sufficient our purposes.

\begin{prop}\label{P:agree-on-convex}
Let $K$ be a compact nc convex set and let $\mu,\nu:\rC(K) \to \cM$ be unital completely positive maps such that $\mu(f) = \nu(f)$ for every $n<\infty$ and every convex nc function $f \in \cM_n(\rC(K))$. Then $\mu = \nu$.
\end{prop}

\begin{proof}
For $t \in [-1,1]$, the function $h_t(x) = x^2(1-tx)^{-1}$ is operator convex on the interval $(-1,1)$ \cite{BenShe1955}. 
The Taylor series expansion of $h_t$ at $x=0$ is $h_t(x) = \sum_{n\ge0} t^n x^{n+2}$. 
Hence for self-adjoint $a \in \cM_n(\rA(K))$ with $\|a\| < 1$, Example \ref{ex:single-op-convex-function} implies that the continuous nc function $h_t(a) \in \rC(K)$ is convex. Hence by assumption,
\[ 0 = (\mu - \nu)(h_t(a)) = \sum_{n\ge0} (\mu-\nu)(a^{n+2}) t^n  \qforal t \in [-1,1]. \]
It follows that the analytic function $k(z) = \sum_{n\ge0} (\mu-\nu)(a^{n+2}) z^n$ is identically zero. 
Therefore $\mu(a^n) = \nu(a^n)$ for all $n\ge2$. This also holds for $n=0$ and $n=1$ by hypothesis.

We now show that if $f = a_1 \cdots a_n$ for self-adjoint $a_1,\ldots,a_n \in \rA(K)$, then $\mu(f) = \nu(f)$. To see this, define self-adjoint $b = (b_{ij}) \in \cM_{n+1}(\rA(K))$ by setting $b_{i,i+1}=b_{i+1,i}=a_i$ for $1 \le i \le n$ and setting all other entries to zero. It is easy to check that the $(1,n+1)$ entry of $b^n$ is $f$. From above, $\mu(b^n) = \nu(b^n)$. Hence $\mu(f) = \nu(f)$. It follows that $\mu$ and $\nu$ agree on the C*-algebra generated by $\rA(K)$, namely $\rC(K)$, and we conclude that $\mu=\nu$.
\end{proof}

\subsection{Multivalued noncommutative functions}

A major difficulty in the noncommutative setting is the fact that the self-adjoint elements of a noncommutative von Neumann algebra do not form a lattice. Inspired by work of Wittstock \cite{Wit1981} and Winkler \cite{Win1999}, we will overcome this difficulty by working with multivalued functions.

\begin{defn}
Let $K$ be an nc convex set and let $F : K \to \cM_n(\cM)_{sa}$ be a multivalued self-adjoint function. 
We say that $F$ is a {\em multivalued nc function} if it is non-degenerate, graded and unitarily equivariant, preserves direct sums and is upwards directed, meaning that
\begin{enumerate}
\item $F(x) \ne \emptyset$ for every $x \in K$,
\item $F(K_m) \subseteq \cM_n(\cM_m)$ for all $m$,
\item $F(x \oplus y) = F(x) \oplus F(y)$ for every $x,y \in K$,
\item $F(\beta x \beta^*) = (1_n \otimes \beta) F(x) (1_n \otimes \beta^*)$ for every $x \in K_m$ and every unitary $\beta \in \cM_m$,
\item $F(x) = F(x) + \cM_n(\cM_m)_+$ for every $m$ and every $x \in K_m$.
\end{enumerate}
We say that $F$ is {\em bounded} if there is a constant $\lambda > 0$ such that for every $\beta \in F(x)$ there is $\alpha \in F(x)$ with $\alpha \leq \beta$ such that $\|\alpha\| \leq \lambda$. If $F$ is bounded, then we let $\|F\|$ denote the infimum of all $\lambda$ as above. Otherwise we write $\|F\| = \infty$. If $G : K \to \cM_n(\cM)$ is another multivalued nc function, then we will write $F \leq G$ if $F(x) \supseteq G(x)$ for every $x \in K$.
\end{defn}

\begin{defn}
Let $K$ be a compact nc convex set and let $F : K \to \cM_n(\cM)$ be a bounded multivalued nc function. The {\em graph} of $F$ is the subset $\graph(F) \subseteq \coprod_m K_m \times \cM_n(\cM_m)$ defined by
\[
\graph_m(F) = \{(x,\alpha) \in K_m \times \cM_n(\cM_m) : x \in K_m \text{ and } \alpha \in F(x) \}.
\]
We say that $F$ is {\em convex} if $\graph(F)$ is an nc convex set, and that $F$ is {\em lower semicontinuous} if $\graph{F}$ is closed.
\end{defn}

\begin{example} \label{ex:single-valued-to-multivalued-nc-function}
Let $K$ be a compact nc convex set and let $f \in \cM_n(\rB(K))$ be self-adjoint. There is a bounded multivalued nc function $F : K \to \cM_n(\cM)$ naturally associated to $f$ defined by $F(x) = [f(x), +\infty)$ for $x \in K$. Note that $F$ is the unique multivalued nc function with $\graph(F) = \epi(f)$.
\end{example}

Recall that if $K$ is a compact nc convex set and $\mu : \rC(K) \to \cM_k$ is a unital completely positive map, then $\mu$ can be extended to a unital completely positive map on the C*-algebra $\rB(K)$ of bounded single-valued nc functions. We will extend $\mu$ further and make sense of the expression $\mu(F)$ when $F : K \to \cM_n(\cM)$ is a bounded multivalued nc function.

\begin{defn}
Let $K$ be a compact nc convex set and let $\mu : \rC(K) \to \cM_k$ be a unital completely positive map. Let $(x,\alpha) \in K_m \times \cM_{m,k}$ be a minimal representation for $\mu$. For a bounded multivalued nc function 
$F : K \to \cM_n(\cM)$, we define 
\[ \mu(F) = (1_n \otimes \alpha^*) F(x) (1_n \otimes \alpha) .\]
\end{defn}

\begin{rem}
The fact that this extension of $\mu$ is well defined follows from unitary equivariance of multivalued nc functions. The argument is similar to the argument for bounded single-valued nc functions from Section \ref{sec:dilations-representations}.
\end{rem}

Let $f \in \cM_n(\rB(K))$ be self-adjoint and let $F : K \to \cM_n(\cM)$ denote the corresponding bounded multivalued nc function defined as in Example \ref{ex:single-valued-to-multivalued-nc-function}. If $G : K \to \cM_n(\cM)$ is a multivalued nc function, then we will write $f = G$, $f \leq G$ and $f \geq G$ if $F = G$, $F \leq G$ or $F \geq G$ respectively.

\subsection{Noncommutative convex envelopes}

In this section we will introduce a notion of convex envelope for continuous nc functions that will play a similarly important role in the noncommutative setting. We will need to work with multivalued nc functions, and this introduces some technical difficulties. However, the results in this section will also apply to single-valued functions via the correspondence in Example \ref{ex:single-valued-to-multivalued-nc-function}.

We will define the convex envelope of a function geometrically in terms of the graph of the function. The non-trivial fact that this definition of the convex envelope is equivalent to an appropriate approximation from below by continuous affine nc functions will be the main result in this section.

Let $K$ be a compact nc convex set. For cardinals $m$ and $n$, we will view $f \in \cM_m(\cM_n(\rB(K)))$ as a function $f : K \to \cM_m(\cM_n(\cM))$ in the obvious way. For another function $g \in \cM_m(\cM_n(\rB(K)))$, we will write $f \leq g$ if $f$ and $g$ are self-adjoint and $f(x) \leq g(x)$ for all $x \in K$. For a multivalued nc function $F : K \to \cM_n(\cM)$, we define a multivalued function $1_m \otimes F : K \to \cM_m(\cM_n(\cM))$ by
\[
(1_m \otimes F)(x) = \{ 1_m \otimes \alpha : \alpha \in F(x) \}, \quad x \in K.
\]
Note that $1_m \otimes F$ is not an nc function since $1_m \otimes \alpha \le \beta$ does not imply that $\beta = 1_m \otimes \beta'$.

\begin{defn} \label{def:convex-envelope}
Let $K$ be a compact nc convex set. The {\em convex envelope} of a bounded multivalued function $F : K \to \cM_n(\cM)$ is the multivalued nc function $\env{F} : K \to \cM_n(\cM)$ determined by the property
\[
 \graph{\env{F}}= \ol{\ncconv}(\graph(F)).
\]
That is, the graph of $\env{F}$ is the closed nc convex hull of the graph of $F$.
\end{defn}

\begin{prop} \label{prop:properties-convex-envelope}
Let $K$ be an nc convex set and let $F : K \to \cM_n(\cM)$ be a bounded multivalued nc function with convex envelope $\env{F}$. Then
\begin{enumerate}[label=\normalfont{(\arabic*)}]
\item $\env{F}$ is a lower semicontinuous convex multivalued nc function,
\item $\env{F} \leq F$,
\item if $F$ is nc convex and lower semicontinuous, then $\env{F}=F$,
\item if $F$ is bounded by $\lambda$, then so is $\env{F}$, and
\item if $G$ is a convex nc function such that $G \le F$, then $G \le \env{F}$.
\end{enumerate}
\end{prop}

\begin{proof}
Since the graph of $\env{F}$ is defined to be nc convex and closed, (1) is immediate. Also, evidently $ \graph{\env{F}} \supset \graph(F)$, so $\env{F} \le F$. If $\graph(F)$ is already closed and nc convex, then clearly $\env{F}=F$.

Suppose that $F$ is bounded by $\lambda$. Then $- \lambda I_n\otimes I_k \le F(x)$ for all $x \in K_k$, and this persists for $\env{F}$.
Suppose that $(x,\beta)$ belongs to the (algebraic) nc convex hull of $\graph(F)$; say $(x,\beta) = \sum \alpha_i^* (x_i,\beta_i) \alpha_i$ where $\sum \alpha_i^* \alpha_i = I_k$. Then since $F$ is bounded by $\lambda$, there exist $\gamma_i \in F(x_i)$ with $\gamma_i \le \beta_i$ and $\|\gamma_i\|\le \lambda$.
It follows that $(x,\gamma) \in \env{F}(x)$ where $\gamma = \sum \alpha_i^* \gamma_i \alpha_i \le \beta$ and $\|\gamma\| \le \lambda$.
In general, if $(x,\beta)$ is a limit of a net of such points $(x_j,\beta_j)$, find $(x_j,\gamma_j)$ with $\gamma_j\le \beta_j$ and $\|\gamma_j\|\le \lambda$.
Extract a convergent cofinal subnet with limit $(x,\gamma)$. Then $\gamma\le\beta$ and $\|\gamma\|\le \lambda$. So $\env{F}$ is bounded by $\lambda$.

Finally it is clear from the definition that $\env{F}$ is the largest convex nc function smaller than $F$.
\end{proof}

The next result is a noncommutative analogue of the classical fact that the convex envelope of a function is obtained as the supremum of the continuous affine functions dominated by the function.

\begin{thm} \label{thm:convex-equals-upper-env}
Let $K$ be a compact nc convex set and let $F : K \to \cM_n(\cM)$ be a  bounded multivalued nc function. Then for $x \in K_p$, 
\[
\env{F}(x) = \bigcap_{m \in \bN\,} \bigcap_{a \leq 1_m \otimes F} \{\alpha \in (\cM_n(\cM_p))_{sa} :  1_m \otimes \alpha \ge a(x) \},
\]
where the intersection is taken over all $m$ and all self-adjoint affine nc functions $a \in \cM_m(\cM_n(\rA(K)))_{sa}$ satisfying $a \leq 1_m \otimes F$. 
The same holds if we intersect over all $m \le \kappa$.
\end{thm}

\begin{proof}
Let $\tilde F(x) := \bigcap_{m \in \bN} \bigcap_{a \leq 1_m \otimes F} \{\alpha \in (\cM_n(\cM_p))_{sa} : 1_m \otimes \alpha \ge a(x) \}$ for $x \in K_p$.
It is easy to see that
\[
\graph(\tilde F) = \bigcap_{m \in \bN\,} \! \bigcap_{a \leq 1_m \otimes F} \!\!\!\!\!\! \big\{(x,\alpha) \!\in K \times (\cM_n(\cM))_{sa} \!:\! (x,1_m \otimes \alpha) \in \epi(a) \big\},
\]
where the intersection is taken over all $m$ and all self-adjoint affine nc functions $a$ in $\cM_m(\cM_n(\rA(K)))_{sa}$ satisfying $a \leq 1_m \otimes F$.
This is an intersection of closed nc convex sets, so $\tilde F$ is a lower semicontinuous convex function with $\tilde F \le F$.
Thus by definition of the convex envelope, we have that $\tilde F \le \env{F}$.
We will prove that $\graph{\tilde F} = \graph{\env{F}}$. It remains to show that $\env{F} \le \tilde F$. 

Since $F$ is bounded, $\env{F}$ is bounded by Proposition~\ref{prop:properties-convex-envelope}. 
By replacing $F$ by $F(x) + (\|F\| + 1)$, we may assume that $\env{F}(x) \subseteq [1_n \otimes 1_k, +\infty)$ for every $k$ and every $x \in K_k$. Then for $(x,\alpha) \in \graph_k(\env{F})$, $\alpha \geq 1_n \otimes 1_k$.

Fix $x_0 \in K_k$ and self-adjoint $\alpha_0 \in \cM_n(\cM_k)$ such that $(x_0,\alpha_0) \not \in \graph_k(\env{F})$. 
To show that $(x_0,\alpha_0) \not \in \graph_k(\tilde{F})$, we must show there is a cardinal $m$ and an affine nc function $a \in (\cM_m(\cM_n(\rA(K))))_{sa}$ such that $a \leq 1_m \otimes F$, in the sense that $[a(x),+\infty) \supseteq 1_m \otimes F(x)$ for all $x \in K$, but $a(x_0) \not \leq 1_m \otimes \alpha_0$.  
Since  $\graph(\tilde F)$ and $\graph(\env{F})$ are both closed and nc convex, Proposition~\ref{prop:convergent-nets} and Section \ref{sec:matrix-convexity} show that we can assume that $m$ is finite.

Let $E$ be a dual operator space containing $K$. Since $\graph(\env{F})$ is closed and nc convex, it follows from Corollary \ref{cor:separation} that there is a normal completely bounded self-adjoint map $\theta : E \oplus \cM_n \to \cM_k$ and self-adjoint $\gamma \in \cM_k$ such that $\theta_l((x,\alpha)) \leq \gamma \otimes 1_l$ for every $l$ and every $(x,\alpha) \in \graph_l(\env{F})$, but $\theta_k((x_0,\alpha_0)) \not \leq \gamma \otimes 1_k$. Here we write $\theta_l$ for the amplification $\theta \otimes \id_l$.

Define normal completely bounded maps $\phi : E \to \cM_k$ and $\psi : \cM_n \to \cM_k$ by $\phi(x) = \theta(x,0_n)$ for $x \in E$ and $\psi(\alpha) = - \theta(0_E,\alpha)$ for $\alpha \in \cM_n$. 
As above, we write $\phi_l = \phi \otimes \id_l$ and $\psi_l = \psi \otimes \id_l$.
Then for every $l$ and every $(x,\alpha) \in \graph_l(\env{F})$,
\[
\phi_l(x) - \psi_l(\alpha) =  \theta_l((x,\alpha)) \leq \gamma \otimes 1_l .
\]
Rearranging gives
\begin{equation} \label{eq:convex-equals-upper-env-1}
\phi_l(x) - \gamma \otimes 1_l \leq \psi_l(\alpha).
\end{equation}

We first claim that $\psi$ is completely positive. To see this, note that for $l$ and $(x,\alpha) \in \graph_l(\env{F})$, the normalization ensures that $\alpha \ge 1_n\otimes 1_l$; and the fact that $\env{F}$ is upwards directed implies $(x,\lambda \alpha) \in \graph_l(\env{F})$ for every $\lambda \geq 1$. Hence by (\ref{eq:convex-equals-upper-env-1}),
\begin{equation} \label{eq:convex-equals-upper-env-2}
\phi_l(x) - \gamma \otimes 1_l \leq \psi_l(\lambda \alpha) = \lambda \psi_l(\alpha).
\end{equation}
Dividing both sides by $\lambda$ and taking $\lambda \to \infty$ yields $\psi_l(\alpha) \geq 0$.

Now for positive $\beta \in \cM_n(\cM_l)$ and $\epsilon > 0$, there is some $\lambda > 0$ such that $\lambda (\beta + \epsilon 1_n \otimes 1_l) \geq \alpha$. Then since $\env{F}$ is upwards directed, 
\[ (x, \lambda (\beta + \epsilon 1_n \otimes 1_l)) \in \graph_l(\env{F}) .\]
Hence by (\ref{eq:convex-equals-upper-env-2}),
\[
\lambda \psi_l(\beta + \epsilon 1_n \otimes 1_l) = \psi_l(\lambda (\beta + \epsilon 1_n \otimes 1_l)) \geq 0.
\]
Dividing by $\lambda$ implies $\psi_l(\beta + \epsilon 1_n \otimes 1_l) \geq 0$, and taking $\epsilon \to 0$ gives $\psi_l(\beta) \geq 0$. Hence $\psi$ is completely positive.

We claim that $\theta$ can be chosen to ensure that $\psi(\alpha)$ is invertible for all $\alpha \in \cM_n$ with $\alpha \ne 0$ and $\alpha \geq 0$. To see this, choose a faithful state $\tau : \cM_n \to \bC$ and $\epsilon > 0$ such that
\[
\theta_k(x_0,\alpha_0) - \epsilon \tau_k(\alpha_0) 1_k \otimes 1_k \not \leq \gamma \otimes 1_n.
\]
Then the completely bounded self-adjoint map $\theta' : E \oplus \cM_n \to \cM_k$ defined by $\theta'(x,\alpha) = \theta(x,\alpha) - \epsilon \tau(\alpha) 1_k$ for $(x,\alpha) \in E \oplus \cM_n$ satisfies $\theta'_l((x,\alpha)) \leq \gamma \otimes 1_l$ for every $l$ and every $(x,\alpha) \in \graph_l(\env{F})$, but $\theta'_k((x_0,\alpha_0)) \not \leq \gamma \otimes 1_k$. Furthermore, the map $\psi' : \cM_n \to \cM_k$ defined by $\psi'(\alpha) = -\theta'(x,0_n)$ for $\alpha \in \cM_n$ satisfies
\[
\psi'(\alpha) = -\theta(0_E,\alpha) + \epsilon \tau(\alpha) 1_k = \psi(\alpha) + \epsilon \tau(\alpha) 1_k.
\]
In particular, the positivity of $\psi$ and the faithfulness of $\tau$ implies that $\psi'(\alpha)$ is invertible for all $\alpha \in \cM_n$ with $\alpha \ne 0$ and $\alpha \geq 0$. By replacing $\theta$ by $\theta'$, we can therefore assume that $\psi$ has this property.

Since $\psi$ is completely positive and $k$ and $n$ are finite, Stinespring's theorem provides finite $m$ and an operator $\beta : H_k \to H_n^m$  such that
\[
\psi(\alpha) = \beta^* (1_m \otimes \alpha) \beta \qfor \alpha \in \cM_n.
\]
Then
\[
\psi_l(\alpha) = (\beta^* \otimes 1_l) (1_m \otimes \alpha) (\beta \otimes 1_l) \qfor \alpha \in \cM_n(\cM_l).
\]
Write $\beta = \nu |\beta|$, where $\nu : H_k \to H_n^m$ is a partial isometry with initial space $(\ker\beta)^\perp$. It follows from above that $|\beta|$ is invertible and hence $\nu^* \nu = 1_k$. Let $q = \nu \nu^*$. 

Since $|\beta|$ is invertible, there is an element $\beta' = |\beta|^{-1} \in \cM_k$. Then $\nu = \beta \beta'$, whence  $\nu \beta' \beta^* = \nu\nu^* = q$. Thus for $\alpha \in \cM_n(\cM_l)$,
\[
(\nu \beta' \otimes 1_l) \psi_l(\alpha) (\beta' \nu^* \otimes 1_l) = (q \otimes 1_l) (1_m \otimes \alpha) (q \otimes 1_l).
\]
Hence decomposing $1_m \otimes \alpha \in \cM_m(\cM_n(\cM_l)))$ as a block matrix with respect to the projection $q \otimes 1_l$ as
\[
1_m \otimes \alpha =
\left[
\begin{matrix}
(1_m \otimes \alpha)_{11} & (1_m \otimes \alpha)_{12} \\
(1_m \otimes \alpha)_{21} & (1_m \otimes \alpha)_{22}
\end{matrix}
\right],
\]
we obtain
\begin{equation} \label{eq:convex-equals-upper-env-3}
(1_m \otimes \alpha)_{11} = (\nu \beta' \otimes 1_l) \psi_l(\alpha) (\beta' \nu^* \otimes 1_l).
\end{equation}

For $\epsilon > 0$, define a self-adjoint affine function $a_\epsilon \in \cM_m(\cM_n(\rA(K)))$ by writing it in block matrix form with respect to the projection $q$ as
\[
a_\epsilon = 
\left[
\begin{matrix}
a_{\epsilon,11} & 0 \\
0 & a_{\epsilon,22}
\end{matrix}
\right],
\]
where
\begin{equation} \label{eq:convex-equals-upper-env-4}
a_{\epsilon,11}(x) = (\nu \beta' \otimes 1_l) ( \phi_l(x) - \gamma \otimes 1_l  - \epsilon 1_k \otimes 1_l) (\beta' \nu^* \otimes 1_l)
\end{equation}
and
\[
a_{\epsilon,22}(x) = -\lambda_\epsilon q^\perp
\]
for $x \in K_l$, where $\lambda_\epsilon > 0$ is chosen to satisfy 
\[
\lambda_\epsilon > \epsilon^{-1} \|\beta\|^2 \|\env{F}\|^2 + \|\alpha\|.
\]

We claim that $a_\epsilon \leq 1_m \otimes \env{F}$ in the sense that $[a_\epsilon(x),+\infty) \supseteq 1_m \otimes \env{F}(x)$ for all $x \in K$. 
The boundedness of $\env{F}$ implies that for every $l$ and every $(x,\alpha') \in \graph_l(\env{F})$, there is $\alpha \in \env{F}(x)$ such that $\alpha \leq \alpha'$ and $\|\alpha\| \leq \|\env{F}\|$.
Therefore, in order to show that $a_\epsilon \leq 1_m \otimes \env{F}$, it suffices to show that $a_\epsilon(x) \leq 1_m \otimes \alpha$ for every $l$ and every $(x,\alpha) \in \graph_l(\env{F})$ with $\|\alpha\| \leq \|\env{F}\|$. Taking the Schur complement of the block matrix of $1_m \otimes \alpha - a_\epsilon(x)$ with respect to the projection $q \otimes 1_l$ implies that this condition is equivalent to the inequalities
\[
(1_m \otimes \alpha)_{22}  - a_{\epsilon,22}(x) \geq 0
\]
and
\begin{align*}
(1_m \otimes \alpha)_{11} &- a_{\epsilon,11}(x) \\
&\geq (1_m \otimes \alpha)_{12} \big((1_m \otimes \alpha)_{22}  - a_{\epsilon,22}(x) \big)^{-1} (1_m \otimes \alpha)_{21}.
\end{align*}

The first inequality follows immediately from the choice of $\lambda_\epsilon$, since
\begin{equation}  \label{eq:convex-equals-upper-env-5}
\begin{split}
(1_m \otimes \alpha)_{22}  - a_{\epsilon,22}(x) &= (1_m \otimes \alpha)_{22} + \lambda_\epsilon q^\perp \otimes 1_l  \\
& \geq \epsilon^{-1} \|\beta\|^2 \|\env{F}\|^2 q^\perp \otimes 1_l.
\end{split}
\end{equation}
For the second inequality, observe that (\ref{eq:convex-equals-upper-env-3}) and (\ref{eq:convex-equals-upper-env-4}) imply
\begin{align*}
(1_m \otimes& \alpha)_{11} - a_{\epsilon,11}(x)  \\
&= (\nu \beta' \otimes 1_l)\big( (\psi_l(\alpha) -  \phi_l(x) + \gamma \otimes 1_l)  + \epsilon 1_k \otimes 1_l \big) (\beta' \nu^* \otimes 1_l) \\
&\geq \epsilon \nu (\beta')^2 \nu^* \otimes 1_l \\
&\geq \frac{\epsilon}{\|\beta\|^2} q \otimes 1_l.
\end{align*}
Then since $\|\alpha\| \leq \|\env{F}\|$, (\ref{eq:convex-equals-upper-env-5}) implies that
\begin{align*}
(1_m \otimes \alpha)_{12} \big((1_m &\otimes \alpha)_{22} - a_{\epsilon,22}(x) \big)^{-1} (1_m \otimes \alpha)_{21} \\
&= (1_m \otimes \alpha)_{12} ((1_m \otimes \alpha)_{22} + \lambda_\epsilon q^\perp \otimes 1_l )^{-1} (1_m \otimes \alpha)_{21}  \\
&\leq \frac{\epsilon}{\|\beta\|^2}  q \otimes 1_l.
\end{align*}
Hence the second inequality is also satisfied. Therefore, $a_\epsilon \leq 1_m \otimes \env{F}$.

Finally, we claim there is an $\epsilon > 0$ such that $a_\epsilon(x_0) \not \leq 1_m \otimes \alpha_0$. To see this, suppose for the sake of contradiction that $a_\epsilon(x_0) \leq 1_m \otimes \alpha_0$ for all $\epsilon > 0$. Then in particular, looking at the top left corner of the block matrix of $1_m \otimes \alpha_0 - a_\epsilon(x)$ with respect to the projection $q \otimes 1_k$ and applying (\ref{eq:convex-equals-upper-env-4}) implies
\begin{align*}
0 &\leq (1_m \otimes \alpha_0)_{11} - a_{\epsilon,11}(x_0) \\
&= (\nu \beta' \otimes 1_k) (\psi_l(\alpha_0) - \phi_k(x_0) + \gamma \otimes  1_k + \epsilon 1_k \otimes 1_k) (\beta'  \nu^* \otimes 1_k) \\
&= (\nu \beta'  \otimes 1_k) (-  \theta((x_0,\alpha_0)) + \gamma \otimes  1_k + \epsilon 1_k \otimes 1_k) (\beta'  \nu^* \otimes 1_k)
\end{align*}
Then multiplying on the left by $\beta^* \otimes 1_k$ and on the right by $\beta \otimes 1_k$ and taking $\epsilon \to 0$ implies $\theta_k((x_0,\alpha_0)) \leq \gamma \otimes 1_k$, contradicting our original separation of $(x_0,\alpha_0)$ from the graph of $\env{F}$. We conclude that for some $\epsilon > 0$, the affine nc function $a_\ep$ achieves the desired separation.

We only needed to use $m<\infty$, but clearly one could also intersect over all $m \le \kappa$ since the right hand side always contains $\env{F}$.
\end{proof}

We will need a useful fact regarding the convex envelope and multiplicity.

\begin{cor} \label{cor:tensor-convex-envelope}
Let $K$ be an nc convex set and let $F : K \to \cM_n(\cM)$ be a self-adjoint bounded multivalued nc function with convex envelope $\env{F}$.  
Then $\env{1_l \otimes F} \le 1_l \otimes \env{F}$.
\end{cor}

\begin{proof}
Note that
\begin{align*}
\graph(\env{1_l \otimes F}) &= \ol{\ncconv} \big( \graph(1_l \otimes F) \big) \\
&= \ol{\ncconv} \left\{ \coprod_{x\in K} (x, 1_l \otimes F(x) ) \right\} .
\end{align*}
The nc convex combinations include all points obtained using points $x_i$ and contractions of the form $1_l\otimes \alpha_i$. Therefore
\begin{align*}
 \graph(\env{1_l \otimes F}) &\supset \coprod_{x\in K} (x, 1_l \otimes \env{F}(x) ) \\
 &= \graph(1_l \otimes \env{F}) . \qedhere
\end{align*}
\end{proof}

The next result is a noncommutative analogue of a result of Mokobodzki (see e.g. \cite{Alfsen}*{Proposition I.5.1}).

\begin{prop} \label{prop:convex-env-convex-fcns}
Let $K$ be a compact nc convex set and let $F : K \to \cM_n(\cM)$ be a self-adjoint bounded multivalued bounded nc function with convex envelope $\env{F}$. Then for $x \in K_p$,
\[
 \env{F}(x) = \bigcap_{g \leq 1_m \otimes F} \{ \alpha \in \cM_n(\cM_p) : 1_m \otimes \alpha \ge g(x) \},
\]
where the intersection is taken over all $m$ and all convex nc functions $g \in \cM_m(\cM_n(\rC(K)))$ satisfying $g \leq 1_m \otimes F$.
\end{prop}

\begin{proof}
By Theorem~\ref{thm:convex-equals-upper-env}, $\env{F}(x)$ is the intersection over such sets with respect to continuous affine nc functions. Since every affine nc function is a convex nc function, the intersection over all $m$ and all convex nc functions $g \in \cM_m(\cM_n(\rC(K)))$ satisfying $g \leq 1_m \otimes F$ is smaller. On the other hand, by Proposition \ref{prop:properties-convex-envelope}, $\env{1_m \otimes F}$ is the largest lower semicontinuous convex multivalued nc function dominated by $1_m \otimes F$, meaning that for all such $g$, $g \leq \env{1_m \otimes F}$. Hence by Corollary \ref{cor:tensor-convex-envelope},  $g \leq 1_m \otimes \env{F}$. Therefore, the intersection is precisely  $\graph(\env{F})$.
\end{proof}

\subsection{Completely positive maps}

The next result  shows that, as in the classical setting, the noncommutative convex envelope encodes information about the set of representing maps of a point. 

\begin{thm} \label{thm:convex-env-equivalent-form}
Let $K$ be a compact nc convex set and let $f : K \to \cM_n(\cM)$ be a self-adjoint lower semicontinuous bounded nc function with convex envelope $\env{f}$. Then for $x \in K_m$,
\[
\env{f}(x) = \bigcup_\mu [\mu(f), +\infty),
\]
where the union is taken over all unital completely positive maps $\mu : \rC(K) \to \cM_m$ with barycenter $x$.
\end{thm}

\begin{proof}
Define $F : K \to \cM$ by $F(x) = \cup_\mu [\mu(f), +\infty)$ for $x \in K_m$, where the union is taken over all unital completely positive maps $\mu : \rC(K) \to \cM_m$ with barycenter $x$. Then $F$ is a self-adjoint bounded multivalued nc function since it is clearly graded, preserves direct sums, is unitarily equivariant and upward directed.

We claim that $F$ is lower semicontinuous. Let $(x_i,\alpha_i)$ be a net in $\graph_m(F)$ converging to $(x,\alpha) \in K_m \times \cM_n(M_m)$. Then there are unital completely positive maps $\mu_i : \rC(K) \to \cM_m$ such that $\mu_i$ has barycenter $x_i$ and $\mu_i(f) \leq \alpha_i$. Let $\mu : \rC(K) \to \cM_m$ be a cluster point of the net $(\mu_i)$. Then $\mu$ has barycenter $x$ and $\mu(f) \leq \alpha$, so $(x,\alpha) \in \graph_m(F)$. 

Next we show that $F$ is convex. Suppose that $(x_i,\alpha_i) \in \graph(F)$, where $x_i \in K_{n_i}$ and $\mu_i$ is a unital completely positive map $\mu_i : \rC(K) \to \cM_m$ with barycenter $x_i$ such that $\mu_i(f) \le \alpha_i$. If $\beta_i \in \cM_{n,n_i}$ so that $\sum \beta_i^* \beta_i = 1_n$, let 
\[ x := \sum_i \beta_i^* x_i \beta_i \in K_n \qand \alpha := \sum_i (1_m \otimes \beta_i^*) \alpha_i (1_m \otimes \beta_i) .\]
We need to verify that $(x,\alpha) \in \graph(F)$. Observe that 
\[ \mu := \sum (1_m \otimes \beta_i^*) \mu_i (1_m \otimes \beta_i) \]
is a unital completely positive map $\mu : \rC(K) \to \cM_m$ with barycenter $x$. In addition,
\begin{align*}
 \mu(f) &= \sum (1_m \otimes \beta_i^*) \mu_i(f)  (1_m \otimes \beta_i)\\
 & \le \sum (1_m \otimes \beta_i^*) \alpha_i  (1_m \otimes \beta_i) = \alpha .
\end{align*}
Therefore $(x,\alpha) \in \graph(F)$.

It now suffices to show that $\env{f} = F$. We will accomplish using Theorem \ref{thm:convex-equals-upper-env} by showing that if $a \in \cM_m(\cM_n(\rA(K)))$, then $a \leq 1_m \otimes f$ if and only if $a \leq 1_m \otimes F$ for every $x \in K_m$.

If $a \leq 1_m \otimes F$, then for $x \in K_m$, then $a(x) \leq  1_m \otimes \mu(f)$ for every unital completely positive map $\mu : \rC(K) \to \cM_m$ with barycenter $x$.
In particular, taking $\mu = \delta_x$ implies $a(x) \leq 1_m \otimes f(x)$. Hence $a \leq 1_m \otimes f$.

On the other hand, if $a \leq 1_m \otimes f$, then for $x \in K_m$ and every unital completely positive map $\mu : \rC(K) \to \cM_m$ with barycenter $x$,
\[
a(x) = \mu(a) \leq \mu(1_m \otimes f) = 1_m \otimes \mu(f).
\]
Hence $a \leq 1_m \otimes F$.
\end{proof}

The next result extends Proposition \ref{prop:convex-env-convex-fcns}.

\begin{rem} \label {R:intersection of two convex functions}
 In the following, we will use multivalued convex nc functions formed from a convex nc function $g$ and a constant $C$:
\[
 H = g \cap -C := \epi(g) \cap \epi(-C) = \{ (x, \alpha) : \alpha \ge g(x), \ \alpha \ge -C \}.
\]
This function should not be confused with $h(x) = g(x) \vee (-C)$. 
This definition makes sense since $-C$ is a scalar function.
However $h$ is usually not a convex nc function, and $\epi (h)$ in general smaller than $H$.
For example, the 2x2 self-adjoint matrix $A = \begin{sbmatrix}2&\sqrt2 \\ \sqrt2 & 1\end{sbmatrix}$ satisfies
 \[
  A \ge \begin{bmatrix}1&\phantom{-}0\\0&-1\end{bmatrix} \qand A \ge \begin{bmatrix}0&0\\0&0\end{bmatrix} 
  \quad\text{but}\quad A \not\ge \begin{bmatrix}1&0\\0&0\end{bmatrix}.
 \]
\end{rem}

\begin{cor} \label{cor:convex-env-convex-fcns-ucp}
Let $K$ be a compact nc convex set and let $F : K \to \cM_n(\cM)$ be a self-adjoint multivalued bounded nc function. Then for every unital completely positive map $\mu : \rC(K) \to \cM_p$,
\[
\mu(\env{F}) = \bigcap_{g \leq 1_m \otimes F} \{ \alpha \in \cM_n(\cM_p) :  1_m \otimes \alpha \ge \mu(g) \},
\]
where the intersection is taken over all $m$ and all convex nc functions $g \in \cM_m(\cM_n(\rC(K)))$ satisfying $g \leq 1_m \otimes F$.
\end{cor}

\begin{proof}
By Proposition \ref{prop:convex-env-convex-fcns},
\begin{align*}
\env{F} &= \bigcap_{g \leq 1_m \otimes F}\!\!\!\! \big\{  \alpha \in \cM_n(\cM_q) : 1_m \otimes \alpha \ge g(x)  \big\} ,
\end{align*}
where the intersection is taken over all $m<\infty$ and all convex nc functions $g \in \cM_m(\cM_n(\rC(K)))$ satisfying $g \leq 1_m \otimes F$.
Now $F$ is bounded, say by $C$. By Proposition~\ref{prop:properties-convex-envelope}(4), $\env{F}$ is also bounded by $C$.
For each convex nc function $g \in \cM_m(\cM_n(\rC(K)))$ satisfying $g \leq 1_m \otimes F$, we also have $H = g \cap (-C) \le 1_m \otimes F$.
In particular, $\|H\| \le C$ as well. Hence 
\begin{align*}
\env{F} &= \bigcap_{g \leq 1_m \otimes F}\!\!\!\! \big\{  \alpha \in \cM_n(\cM_q) : 1_m \otimes \alpha \ge H(x)  \big\} ,
\end{align*}
where $H = g \cap (-C)$ and the intersection is taken over all $m$ and all convex nc functions $g \in \cM_m(\cM_n(\rC(K)))$ satisfying $g \leq 1_m \otimes F$.

Fix a minimal representation $(x,\nu) \in K_q \times \cM_{q,p}$ for $\mu$ so that $\mu(f) =  \nu^* f(x)  \nu$ for $f \in \rC(K)$ and $\mu(\env{F}) = (1_n\otimes \nu^*) \env{F}(x) (1_n\otimes \nu)$.  Note that since $\env{F}(x)$ is upward directed and $\mu(\env{F})$ is a corner of $\env{F}(x)$, it is also upward directed.
If $1_m \otimes \alpha \ge g(x)$, then setting $\beta = (1_n \otimes \nu^*) \alpha (1_n \otimes \nu)$, it follows that 
\[
 1_m \otimes \beta \ge (1_m \otimes 1_n \otimes \nu^*) g(x) (1_m \otimes 1_n \otimes \nu) = \mu(g) .
\]
Therefore $\mu(\env{F})$ is contained in the intersection on the RHS of the desired formula.

Conversely,  suppose $\beta \in \cM_n(\cM_p)$ satisfies $1_m \otimes \beta \ge \mu(g)$ for all $m$ and all convex nc functions $g \in \cM_m(\cM_n(\rC(K)))$ with $g \leq 1_m \otimes F$. Fix such $g$ and let $H = g \cap -C$ as above. Note that $\beta \geq -C$. 

For $\ep>0$, define $\alpha_\ep \in \cM_n(\cM_q)_{sa}$ by
\begin{align*}
 \alpha_\ep &= (1_n\otimes \nu) (\beta + \ep 1_p)(1_n \otimes \nu^*) + (1_n \otimes (1_q-\nu\nu^*))( \ep^{-1}\|g\|^2 + \|g\|) \\&
 \simeq \begin{bmatrix} \beta+ \ep & 0 \\ 0 & \ep^{-1}\|g\|^2 + \|g\|  \end{bmatrix} ,
\end{align*}
where the decomposition is taken with respect to the range of $1_n\otimes \nu$.
If we also decompose $g(x)$ with respect to the range of $1_m \otimes 1_n \otimes \nu$, it has the form
\[
 g(x) \simeq \begin{bmatrix} g_{11} & g_{12} \\ g_{21} & g_{22} \end{bmatrix} 
\]
and by hypothesis $g_{11} \simeq \mu(g) \le 1_m \otimes \beta$. Thus
\begin{align*}
 1_m \otimes \alpha_\ep - g(x) &\simeq  
 \begin{bmatrix} 1_m \otimes \beta - g_{11} + \ep(1_m \otimes 1_p) & -g_{12} \\ -g_{21} & \ep^{-1} \|g\|^2 + (\|g\| - g_{22}) \end{bmatrix} \\
 &\ge \begin{bmatrix}  \ep & -g_{12} \\ -g_{21} & \ep^{-1} \|g\|^2  \end{bmatrix} \\
 &\ge \begin{bmatrix} 0&0\\0&0 \end{bmatrix} .
\end{align*}
Hence $-C \leq 1_m \otimes \alpha_\ep \geq g(x)$. Therefore, $1_m \otimes \alpha_\ep \in H(x)$. From above, $\|H\| \leq C$, so there is $\gamma_\ep \in H(x)$ with $\|\gamma_\ep\| \leq C$ such that $\gamma_\ep \leq 1_m \otimes \alpha_\ep$.

Define $\alpha'_\ep \in \cM_n(\cM_q)_{sa}$ by
\begin{align*}
 \alpha'_\ep &=(1_n\otimes \nu) (\beta + \ep 1_p)(1_n \otimes \nu^*) + (1_n \otimes (1_q-\nu\nu^*))( \ep^{-1}C^2 + C) \\
 &\simeq \begin{bmatrix} \beta+ \ep & 0 \\ 0 & \ep^{-1}C^2 + C  \end{bmatrix} . 
\end{align*}
Then arguing as above, $1_m \otimes \alpha'_\ep \geq \gamma_\ep$. Since $\gamma_\ep \in H(x)$ and $H(x)$ is upward directed, $1_m \otimes \alpha'_\ep \in H(x)$.

Since the construction of $\alpha'_\ep$ did not depend on $H$, it follows from the above expression for $\ol{F}(x)$ that $\alpha'_\ep \in \ol{F}(x)$. Since $\|\ol{F}\| = C$, there is $\gamma'_\ep \in \ol{F}(x)$ with $\|\gamma'_\ep\| \leq C$ such that $\gamma_\ep \leq 1_m \otimes \alpha'_\ep$. Furthermore, since the net $(\gamma'_\ep)_{\ep > 0}$ is bounded and $\ol{F}$ is lower semicontinuous, there is a cofinal subnet with limit $\gamma' \in \ol{F}(x)$. By construction
\[
(1_n \otimes \nu^*) \gamma' (1_n \otimes \nu) \leq \lim_{\ep \to 0} (1_n \otimes \nu^*) \alpha'_\ep (1_n \otimes \nu) = \beta.
\]
Therefore, since $(1_n \otimes \nu^*) \gamma' (1_n \otimes \nu) \in (1_n \otimes \nu^*) \ol{F}(x) (1_n \otimes \nu) = \mu(\ol{F})$ and $\mu(\ol{F})$ is upward directed, it follows that $\beta \in \mu(\ol{F})$.
\end{proof}

\subsection{Noncommutative Jensen inequality}

The next result is a natural noncommutative analogue of the classical Jensen inequality.

\begin{thm}[Noncommutative Jensen inequality]
Let $K$ be a compact nc convex set and let $f \in \rB(K)$ be a self-adjoint lower semicontinuous convex nc function. Then for any completely positive map $\mu : \rC(K) \to \cM_n$ with barycenter $x \in K_n$, $f(x) \leq \mu(f)$.
\end{thm}

\begin{proof}
For $x \in K_n$, Theorem~\ref{thm:convex-equals-upper-env} and Theorem \ref{thm:convex-env-equivalent-form} imply that
\[
[f(x),+\infty) = \env{f}(x) = \bigcup_\mu [\mu(f), +\infty),
\]
where the union is taken over all unital completely positive maps $\mu : \rC(K) \to \cM_n$ with barycenter $x$. In particular, $f(x) \leq \mu(f)$ for all such $\mu$.
\end{proof}

\section{Orders on Completely Positive Maps} \label{sec:orders}

\subsection{Classical Choquet order}

The classical Choquet order is a generalization of the even more classical notion of majorization. Let $C$ be a compact convex set. For probability measures $\mu$ and $\nu$ on $\C$, $\nu$ is said to dominate $\mu$ in the Choquet order, written $\mu \prec_c \nu$, if $\mu(f) \leq \nu(f)$ for every convex function $f \in \rC(C)$. The Choquet order is a partial order on the space of probability measures on $C$.

Heuristically, the Choquet order measures the how far the support of a probability measure is from the extreme boundary $\partial C$, in the sense that if $\mu \prec_c \nu$, then the support of $\nu$ is closer to $\partial C$ than the support of $\mu$. In fact, if $\nu$ is maximal in the Choquet order and $C$ is metrizable, then $\nu$ is actually supported on $\partial C$. If $C$ is non-metrizable, then $\partial C$ is not necessarily Borel, but the maximality of $\mu$ still implies that it is supported on $\partial C$ in an appropriate sense.

\subsection{Noncommutative Choquet order}

In this section we will introduce a noncommutative analogue of the Choquet order for unital completely positive maps. The comparison will be with respect to convex continuous nc functions in the sense of Section \ref{sec:convex-nc-fcns}. Eventually, we will see that this order measures how far the support of a unital completely positive map is from the extreme boundary in an appropriate sense.

\begin{defn} \label{defn:nc-Choquet-order}
Let $K$ be a compact nc convex set and let $\mu,\nu : \rC(K) \to \cM_p$ be unital completely positive maps. We say that $\mu$ is dominated by $\nu$ in the {\em nc Choquet order} and write $\mu \prec_c \nu$ if $\mu(f) \leq \nu(f)$ for every $n$ and every convex nc function $f \in \cM_n(\rC(K))$.
\end{defn}

\begin{lem}
Let $K$ be a compact nc convex set and let $\mu,\nu : \rC(K) \to \cM_p$ be unital completely positive maps with $\mu \prec_c \nu$. 
Then $\mu$ and $\nu$ have the same barycenter.
\end{lem}

\begin{proof}
Suppose $\mu \prec_c \nu$. For $a \in \rA(K)$, both $a$ and $-a$ are convex, so $\mu(a) \leq \nu(a)$ and $\mu(-a) \leq \nu(-a)$, implying $\mu(a) = \nu(a)$. Hence $\mu|_{\rA(K)} = \nu|_{\rA(K)}$.
\end{proof}

Since $\env{f}$ is a multivalued function, the following proposition is not immediate,

\begin{prop} \label{prop:choquet-order-convex-env}
Let $K$ be a compact nc convex set and let $\mu,\nu : \rC(K) \to \cM_p$ be unital completely positive maps. 
Then $\mu \prec_c \nu$ if and only if $\mu(\env{f}) \leq \nu(\env{f})$ for every $n$ and every self-adjoint nc function $f \in \cM_n(\rC(K))$.
\end{prop}

\begin{proof}
If $\mu \prec_c \nu$ and $f \in \rC(K)$, Corollary~\ref{cor:convex-env-convex-fcns-ucp} shows that 
\begin{align*}
  \mu(\env{f}) &= \bigcap_{g \leq 1_m \otimes f} \{ \alpha \in M_n(M_p) :  1_m \otimes \alpha \ge \mu(g) \} \\&
  \supseteq  \bigcap_{g \leq 1_m \otimes f} \{ \alpha \in M_n(M_p) :  1_m \otimes \alpha \ge \nu(g) \} = \nu(\env{f}) .
\end{align*}
where the intersection is over all $m<\infty$ and all convex nc functions $g$ in $M_m(M_n(\rC(K)))$ satisfying $g \leq 1_m \otimes f$.
Thus $\mu(\env{f}) \leq \nu(\env{f})$.

The converse follows since if $f \in \rC(K)$ is convex, then $\env{f} = f$. 
\end{proof}

\begin{prop} \label{P:partial-order}
Let $K$ be a compact nc convex set and let $L$ denote the nc state space of $\rC(K)$. 
Then for each $n$, the nc Choquet order is a partial order on $L_n$.
\end{prop}

\begin{proof}
It is easy to see that the nc Choquet order is reflexive and transitive. 
Antisymmetry follows from Proposition~\ref{P:agree-on-convex}.
\end{proof}

\subsection{Dilation order}

There is another natural order for unital completely positive maps relating to the dilation theory of completely positive maps.

\begin{defn} \label{defn:dilation-order}
Let $K$ be a compact nc convex set and let $\mu,\nu : \rC(K) \to \cM_m$ be unital completely positive maps. We say that $\mu$ is dominated by $\nu$ in the {\em dilation order} and write $\mu \prec_d \nu$ if there are representations $(x,\alpha) \in K_n \times \cM_{n,m}$ for $\mu$ and $(y,\beta) \in K_p \times \cM_{p,m}$ for $\nu$ along with an isometry $\gamma \in \cM_{p,n}$ such that $\beta = \gamma \alpha$ and $x = \gamma^* y \gamma$.
\end{defn}

\begin{rem}
Note that $\alpha$ and $\beta$ are isometries satisfying $\mu = \alpha^* \delta_x \alpha$ and $\nu = \beta^* \delta_y \beta$. 
The condition $x = \gamma^* y \gamma$ implies that $y$ dilates $x$.
\end{rem}

\begin{prop} \label{prop:dilation-order}
Let $K$ be a compact nc convex set and let $\mu,\nu : \rC(K) \to \cM_m$ be unital completely positive maps such that $\mu \prec_d \nu$.
Then the representation $(x,\alpha)$ used to demonstrate this relation may be taken to be minimal.
Hence the representation $(x,\alpha)$ may be chosen arbitrarily.
\end{prop}

\begin{proof}
Select representations $(x,\alpha) \in K_n \times \cM_{n,m}$ for $\mu$ and $(y,\beta) \in K_p \times \cM_{p,m}$ for $\nu$ along with an isometry $\gamma \in \cM_{p,n}$ such that $\beta = \gamma \alpha$ and $x = \gamma^* y \gamma$.
If $(x_0,\alpha_0)$ is a minimal dilation of $\mu$, then the uniqueness of this dilation shows that $x \cong x_0 \oplus x_1$ and $\alpha \cong \begin{sbmatrix} \alpha_0 \\ 0 \end{sbmatrix}$. There is a unique isometry $\theta$ so that $\alpha = \theta \alpha_0$.
Therefore $x_0 = \theta^* x \theta = \theta^* \gamma^* y \gamma \theta$ and $\beta = \gamma \alpha = (\gamma \theta) \alpha_0$. 
In particular, taking $\gamma_0 = \gamma \theta$, we can always assume that the representation $(x,\alpha)$ of $\mu$ is minimal.
Note that it is not necessarily true that $(y,\beta)$ will be a minimal representation of $\nu$.

If $(x,\alpha) \in K_n \times \cM_{m,n}$ is any representation of $\mu$, then as in the previous paragraph, $x \cong x_0 \oplus x_1$ and $\alpha = \gamma_0 \alpha_0$. So we can define $\gamma = \gamma_0 \theta^*$.
\end{proof}

\begin{lem} \label{L:same-barycenter}
Let $K$ be a compact nc convex set and let $\mu,\nu : \rC(K) \to \cM_m$ be unital completely positive maps with $\mu \prec_d \nu$. 
Then $\mu$ and $\nu$ have the same barycenter.
\end{lem}

\begin{proof}
Let $(x,\alpha) \in K_n \times \cM_{n,m}$, $(y,\beta) \in K_p \times \cM_{p,m}$ and $\gamma \in \cM_{p,n}$ be as in Definition \ref{defn:dilation-order}. Then for $a \in \rA(K)$,
\[
\mu(a) = \alpha^* a(x) \alpha = \alpha^* a(\gamma^* y \gamma) \alpha = \alpha^* \gamma^* a(y) \gamma \alpha = \beta^* a(y) \beta = \nu(a). \qedhere
\]
\end{proof}

The next result provides a useful reformulation of the dilation order that we will use frequently. 

\begin{prop} \label{prop:criterion-dilation-order}
Let $K$ be a compact nc convex set and let $\mu,\nu : \rC(K) \to \cM_m$ be unital completely positive maps.
Then $\mu \prec_d \nu$ if and only if for any representation $(x,\alpha) \in K_n \times \cM_{n,m}$ for $\mu$, there is a unital completely positive map $\tau : \rC(K) \to \cM_n$ with barycenter $x$ satisfying $\nu = \alpha^* \tau \alpha$.
\end{prop}

\begin{proof}
Suppose $\mu \prec_d \nu$. Given $(x,\alpha) \in K_n \times \cM_{n,m}$, use Proposition~\ref{prop:dilation-order} to find $(y,\beta) \in K_p \times \cM_{p,m}$ and $\gamma \in \cM_{p,n}$ be as in Definition \ref{defn:dilation-order}. Let $\tau = \gamma^* \delta_y \gamma$.  Then $\tau|_{\rA(K)} = \gamma^* y \gamma = x$ and $\nu = \alpha^* \tau \alpha$.

Conversely, suppose there is a representation $(x,\alpha) \in K_n \times \cM_{n,m}$ for $\mu$ and a unital completely positive map $\tau : \rC(K) \to \cM_n$ with barycenter $x$ satisfying $\nu = \alpha^* \tau \alpha$. Choose a representation $(y,\gamma) \in K_p \times \cM_{p,m}$ for $\tau$. Then letting $\beta = \gamma \alpha$, $(y,\beta)$ is a representation for $\nu$ and $\gamma^* y \gamma = \tau|_{\rA(K)} = x$.
\end{proof}

\begin{prop} \label{P:minimal_in_order}
Let $K$ be a compact nc convex set. For $x \in K$, the representation $\delta_x$ is the unique minimal element among the family of representing maps of $x$ with respect to both the nc Choquet order and the dilation order. 
\end{prop}

\begin{proof}
For $x \in K_m$, note that $(x, 1_m)$ is a minimal representation of $\delta_x$. Let $\mu : K_m \to \cM_m$ be a unital completely positive map with barycenter $x$ and let $(y,\alpha) \in K_n \times \cM_{n,m}$ be a minimal representation of $\mu$. Then $x = \alpha^* y \alpha$ and $\mu = \alpha^* \delta_y \alpha$. So $\delta_x \prec_d \mu$.

Let $f \in \cM_n(\rC(K))$ be a convex nc function. Then by Remark~\ref{rem:simpler-criterion-convexity},
\[ f(x) = f(\alpha^* y \alpha) \le \alpha^* f(y) \alpha = \mu(f) .\]
Hence $\delta_x \prec_c \mu$.
\end{proof}

\subsection{Dilation order and convex envelopes}

In this section we will make a connection between convex envelopes and the dilation order. This will be the key fact used in the next section to show that the two orders coincide. In view of Proposition~\ref{P:minimal_in_order}, this generalizes Theorem \ref{thm:convex-env-equivalent-form}.

\begin{thm} \label{thm:convex-env-ucp-dilation-order}
Let $K$ be a compact nc convex set. Let $f \in \cM_n(\rB(K))$ be a self-adjoint bounded nc function with convex envelope $\env{f}$. Then for a unital completely positive map $\mu : \rC(K) \to \cM_k$,
\[
\mu(\env{f}) = \bigcup_{\mu \prec_d \nu} [\nu(f), +\infty),
\]
where the union is taken over all unital completely positive maps $\nu : \rC(K) \to \cM_k$ with $\mu \prec_d \nu$. 
\end{thm}

\begin{proof}
Let $(x,\alpha) \in K_m \times \cM_{m,k}$ be a minimal representation of $\mu$. Then by Theorem \ref{thm:convex-env-equivalent-form},
\[
\mu(\env{f}) = \alpha^* \env{f}(x) \alpha = \bigcup_\tau \alpha^* [\tau(f), +\infty) \alpha = \bigcup_\tau [\alpha^* \tau(f) \alpha, +\infty),
\]
where the union is taken over all unital completely positive maps $\tau : \rC(K) \to \cM_m$ with barycenter $x$.

The result now follows from Proposition~\ref{prop:criterion-dilation-order} which says that a unital completely positive map $\nu : \rC(K) \to \cM_k$ satisfies $\mu \prec_d \nu$ if and only if there is a unital completely positive map $\tau : \rC(K) \to \cM_m$ with barycenter $x$ such that $\nu = \alpha^* \tau \alpha$.
\end{proof}

\begin{cor} \label{cor:criterion-maximality-dilation-order}
Let $K$ be a compact nc convex set. A unital completely positive map $\mu : \rC(K) \to \cM_k$ is maximal in the dilation order if and only if $ [\mu(f),+\infty) = \mu(\env{f})$ for all $f \in \rC(K)$.
\end{cor}

\begin{proof}
If $\mu$ is maximal in the dilation order, then Theorem \ref{thm:convex-env-ucp-dilation-order} immediately implies $[\mu(f),+\infty) = \mu(\env{f})$ for all $f \in \rC(K)$.

Conversely, suppose that $[\mu(f),+\infty) = \mu(\env{f})$ for all $f \in \rC(K)$.
Let $\nu : \rC(K) \to \cM_k$ be a unital completely positive map with $\mu \prec_d \nu$. 
For $f \in \rC(K)$,  Theorem \ref{thm:convex-env-ucp-dilation-order} implies that
\begin{gather*}
[\mu(f),+\infty) = \mu(\env{f}) \supseteq [\mu(f),+\infty) \cup [\nu(f),+\infty) \\
\intertext{and}
[\mu(-f),+\infty) = \mu(\env{-f}) \supseteq [\mu(-f),+\infty) \cup [\nu(-f),+\infty).
\end{gather*}
Therefore $\mu(f) \leq \nu(f)$ and $\mu(-f) \leq \nu(-f)$, implying $\mu(f) = \nu(f)$. Hence $\mu=\nu$ and therefore $\mu$ is maximal.
\end{proof}

\subsection{Equivalence of orders} \label{sec:equivalence-orders}

In this section we will show that the noncommutative Choquet order and the dilation order coincide. 

\begin{thm} \label{thm:equivalence-orders}
Let $K$ be a compact nc convex set and let $\mu,\nu : \rC(K) \to \cM_n$ be unital completely positive maps. 
Then $\mu \prec_c \nu$ if and only if $\mu \prec_d \nu$.
\end{thm}

\begin{proof}
If $\mu \prec_d \nu$, then Theorem~ \ref{thm:convex-env-ucp-dilation-order} implies that for every self-adjoint continuous convex nc function $f \in \cM_n(\rC(K))$, 
\[
 \mu(f) = \mu(\env{f}) = \bigcup_{\mu \prec_d \lambda} [\lambda(f), +\infty) \supseteq \bigcup_{\nu \prec_d \lambda} [\lambda(f), +\infty) = \nu(\env{f}) = \nu(f).
\]
Therefore $\mu(f) \leq \nu(f)$. That is, $\mu \prec_c \nu$.

Conversely, if $\mu \prec_c \nu$, then Proposition \ref{prop:choquet-order-convex-env} implies $\mu(\env{f}) \leq \nu(\env{f})$ for every $n$ and every self-adjoint nc function $f \in \cM_n(\rC(K))$.
Let $\mathcal{F}$ be a dense family of functions in the unit ball of $\rC(K)_{sa}$ such that $\mathcal{F} = -\mathcal{F}$ and define $h \in \cM_m(\rC(K))$ by $h = \oplus_{f \in \mathcal{F}} f$. Then from above, $\mu(\env{h}) \leq \nu(\env{h}) \leq \nu(h)$. Hence by Theorem~ \ref{thm:convex-env-ucp-dilation-order}, there is a unital completely positive map $\lambda : \rC(K) \to \cM_p$ such that $\mu \prec_d \lambda$ and $\lambda(h) \leq \nu(h)$. But then $\lambda(f) \leq \nu(f)$ and $\lambda(-f) \leq \nu(-f)$ for all $f \in \mathcal{F}$, implying $\lambda(f) = \nu(f)$ for all $f \in \mathcal{F}$. Since $\mathcal{F}$ is dense in $\rC(K)$, it follows that $\lambda = \nu$. Hence $\mu \prec_d \nu$.
\end{proof}

\begin{rem}
If one prefers, one could instead consider the directed set of finite $F = -F \Subset \mathcal{F}$.
For each $F$, arguing as above, one obtains $\mu \prec \lambda_F$ such that $\lambda_F|_F = \nu|_F$.
Take a limit of a convergent cofinal subnet to obtain the $\lambda$ constructed in the proof.
\end{rem}

\subsection{Maximality in the dilation order} \label{sec:max-in-diln-order}

In this section, we discuss maximality in the dilation order and its relationship to maximal dilations (which are different).
We will obtain new, shorter, more conceptual proofs of the Dritschel--McCullough theorem~\ref{thm:dritschel-mccullough} \cite{DriMcC2005}*{Theorem 1.2}.
and of Theorem \ref{thm:dilation-pure-to-extreme} \cite{DK2015}*{Theorem 2.4} asserting the existence of extreme points.

\begin{prop} \label{prop:exist-maximal-in-dilation-order}
Let $K$ be a compact nc convex set and let $\mu : \rC(K) \to \cM_n$ be a unital completely positive map. 
Then there is a unital completely positive map $\nu : \rC(K) \to \cM_n$ such that $\mu \prec_d \nu$ and $\nu$ is maximal in the dilation order.
\end{prop}

\begin{proof}
By Theorem~\ref{thm:equivalence-orders}, it suffices to obtain $\mu \prec_c \nu$ which is maximal in the Choquet order. Let $N = \{ \nu : \rC(K) \to \cM_n : \mu \prec_d \nu \}$ with the $\prec_c$ order. We will show that $N$ is inductive and apply Zorn's lemma. Suppose that $C = \{ \nu_i : i \in \I \}$ is a chain in $N$.  The nc state space $L$ of $\rC(K)$ is compact. Hence the net $\{\nu_i \}_\I$ has a cofinal convergent subnet, say $\Lambda$, with $\nu = \lim_\Lambda \nu_{i_\lambda}$ in $L_n$. If $f$ is any nc convex function, $\nu_i(f) \le \nu_j(f)$ for $i \le j \in \I$. Since $\Lambda$ is cofinal, $\nu_i(f) \le \nu(f)$ for all $i\in\I$; whence $\nu$ is an upper bound for $C$ in the nc Choquet order. By Zorn's lemma, there is a maximal element of $N$ in the Choquet order. 
\end{proof}

\begin{lem} \label{lem:dilation-maximal-implies-zeros}
Let $\mu : \rC(K) \to \cM_m$ be a unital completely positive map that is maximal in the dilation order and let $(x,\alpha) \in K_n \times \cM_{n,m}$ be a representation of $\mu$. Let $y \in K_p$ be a dilation of $x$ and let $\beta \in \cM_{p,n}$ be an isometry such that $x = \beta^* y \beta$. Let $\gamma \in \cM_p$ denote the orthogonal projection onto $\Ran(\beta)^\perp$. Then $\alpha^* y \gamma = 0$ and $\gamma y \alpha = 0$. Equivalently, with respect to the decomposition $H_p = \Ran(\alpha) \oplus (\Ran(\beta) \ominus \Ran(\alpha)) \oplus \Ran(\beta)^\perp$, $y$ has block matrix form
\[
y = \left[ \begin{matrix}
* & * & 0 \\
* & * & * \\
0 & * & *
\end{matrix} \right]  
\]
\end{lem}

\begin{proof}
Define a unital completely positive map $\nu : \rC(K) \to \cM_m$ by $\nu = \alpha^* \beta^* \delta_y \beta \alpha$. Then by the definition of the dilation order, $\mu \prec_d \nu$. Hence by the maximality of $\mu$, $\nu = \mu$. In particular, for $a \in \rA(K)$,
\[
\alpha^* a^*(x) a(x) \alpha = \mu(a^*a) = \alpha^* \beta^* a^*(y) a(y) \beta \alpha.
\]
Since $a$ was arbitrary, comparing the left and right hand side of this equation gives the desired result.
\end{proof}

The next result is the Dritschel--McCullough Theorem.

\MaximalDilationTheorem*

\begin{proof}
Fix a point $x \in K_m$. We will construct a maximal dilation of $x$. Let $y_1 = x$ and $n_1 = m$. For $i > 1$, suppose that $y_{i-1} \in K_{n_{i-1}}$ is a dilation of $x$. If $y_{i-1}$ is maximal, then we are done. Hence we can suppose otherwise that $y_{i-1}$ is not maximal. In this case, Proposition~\ref{prop:exist-maximal-in-dilation-order} implies there is a unital completely positive map $\mu_{i-1} : \rC(K) \to \cM_{n_{i-1}}$ with barycenter $y_{i-1}$ that is maximal in the dilation order. Let $(y_i,\alpha_i) \in K_{n_i} \times \cM_{{n_i,n_{i-1}}}$ be a minimal representation of $\mu_{i-1}$. Then $y_i$ dilates $y_{i-1}$, and hence it also dilates $x$. 

Proceeding inductively, we either eventually obtain a maximal dilation of $x$, or we obtain a sequence $(y_i, \alpha_i) \in K_{n_i} \times \cM_{n_i,n_{i-1}}$ such that for $i > 1$, the map $\mu_{i-1} = \alpha_i^* \delta_{y_i} \alpha_i$ is maximal in the dilation order and has barycenter $y_{i-1}$. Assuming the second case, we will now construct a maximal dilation of $x$ from the sequence $(y_i,\alpha_i)$. 

First, we can identify each $H_{n_{i-1}}$ with $\Ran(\alpha_i)$ as appropriate, so that with respect to the decomposition $H_{n_i} = \Ran(\alpha_i) \oplus \Ran(\alpha_i)^\perp$, $y_i$ can be written in block matrix form as
\[
y_i = \left[ \begin{matrix} y_{i-1} & * \\ * & * \end{matrix} \right] \quad \text{for} \quad i > 1.
\]
For $i > 2$, further expanding $y_{i-1}$ in terms of $y_{i-2}$ and applying Lemma \ref{lem:dilation-maximal-implies-zeros} implies that $y_i$ can be written in block matrix form as
\[
y_i = \left[ \begin{matrix} y_{i-2} & * & 0 \\ * & * & * \\ 0 & * & * \end{matrix} \right] \quad \text{for} \quad i > 2.
\]
Hence we can write
\[
y_2 = \left[ \begin{matrix} x & r_1 \\ r_1^* & x_2 \end{matrix} \right],\ 
y_3 = \left[ \begin{matrix} x & r_1 & 0 \\ r_1^* & x_2 & r_2 \\ 0 & r_2^* & x_3 \end{matrix} \right],\
y_4 = \left[ \begin{matrix} x & r_1 & 0 & 0 \\ r_1^* & x_2 & r_2 & 0 \\ 0 & r_2^* & x_3 & r_3 \\ 0 & 0 & r_3^* & x_4 \end{matrix} \right].
\]
Proceeding inductively, we see that for $i > 1$, we can write $y_i$ in block tridiagonal matrix form as
\[
y_i = 
\left[ \begin{matrix}
x & r_1 \\
r_1^* & x_2 & r_2 \\
& r_2^* & x_3 & r_3 \\
& & \ddots & \ddots & \ddots & \\
& & & r_{i-2}^* & x_{i-1} & r_{i-1} \\
& & & & r_{i-1}^* & x_i
\end{matrix} \right]
\quad \text{for} \quad i > 1.
\]
Applying Proposition \ref{prop:convergent-nets} to the sequence $(y_i)$, we obtain $y \in K_n$ with block tridiagonal decomposition
\[
y = \left[ \begin{matrix}
x & r_1 \\
r_1^* & x_2 & r_2 \\
& r_2^* & x_3 & r_3 \\
& & \ddots & \ddots & \ddots & \\
& & & r_{i-2}^* & x_{i-1} & r_{i-1} \\
& & & & r_{i-1}^* & x_i & r_i \\
& & & & & \ddots & \ddots & \ddots 
\end{matrix} \right]
\]
Note that $y$ dilates $x$. We will now show that $y$ is maximal. 

For $i \geq 1$, let $\beta_i \in \cM_{n,n_i}$ be an isometry with range corresponding to the upper left $i \times i$ block submatrix of $y$, so that $\beta_i^* y \beta_i = y_i$. Then by the definition of the dilation order, the corresponding unital completely positive map $\nu_i : \rC(K) \to \cM_{n_i}$ defined by $\nu_i = \beta_i^* \delta_{y_i} \beta_i$ satisfies $\mu_i \prec_d \nu_i$. Hence $\nu_i = \mu_i$ by the maximality of $\mu_i$. In particular, $(y,\beta_i)$ is a representation of $\mu_i$. Now applying Lemma \ref{lem:dilation-maximal-implies-zeros} to each $\mu_i$ with respect to the representation $(y,\beta_i)$ implies that any dilation of $y$ is necessarily trivial. We conclude that $y$ is a maximal dilation of $x$.
\end{proof}

This result has the following striking consequence.
It shows that in the proof above, only the first dilation (to $y_2$ which we call $x$ below) was actually required to obtain a maximal point.

\begin{thm} \label{thm:equivalence-dilation-maximal-unique-representing-map}
Let $K$ be a compact nc convex set and let $\mu : \rC(K) \to \cM_m$ be a unital completely positive map with representation $(x,\alpha) \in K_n \times \cM_{n,m}$. If $\mu$ is maximal in the dilation order and the representation $(x,\alpha)$ is minimal, then $x$ is a maximal point. Conversely, if $x$ is a maximal point then $\mu$ is maximal in the dilation order.
\end{thm}

\begin{proof}
Suppose that $\mu$ is maximal in the dilation order, and let $(x,\alpha)$ be a minimal representation of $\mu$. By Theorem \ref{thm:dritschel-mccullough}, there is maximal point $y \in K_p$ that dilates $x$.  The maximality of $y$ implies that it has a unique representing map, namely $\delta_y$.  Let $\beta \in \cM_{p,n}$ be an isometry such that $\beta^* y \beta = x$ and define a unital completely positive map $\nu : \rC(K) \to \cM_m$ by $\nu = \alpha^* \beta^* \delta_y \beta \alpha$. Then by Proposition \ref{prop:criterion-dilation-order}, $\mu \prec_d \nu$. Hence by the maximality of $\mu$, $\mu = \nu$. Thus the pair $(y,\beta \alpha)$ is a representation of $\mu$. By the minimality of the representation $(x,\alpha)$ and the uniqueness of minimal representations, $y \cong x \oplus z$ for some $z \in K$. Since $y$ is a maximal point, it follows that $x$ is a maximal point.

Conversely, suppose $x$ is a maximal point. Let $\nu : \rC(K) \to \cM_m$ be a unital completely positive map such that $\mu \prec_d \nu$. Then by Proposition \ref{prop:criterion-dilation-order} there is a completely positive map $\tau : \rC(K) \to \cM_n$ with barycenter $x$ such that $\nu = \alpha^* \tau \alpha$. The map $\tau$ has barycenter $x$. Since $x$ is maximal, it has a unique representing map, and hence $\tau = \delta_x$. Hence $\nu = \alpha^* \delta_x \alpha = \mu$ and we conclude that $\mu$ is maximal in the dilation order.
\end{proof}

Now we give a new proof of of Theorem \ref{thm:dilation-pure-to-extreme} \cite{DK2015}*{Theorem 2.4}.

\PureDilationTheorem*

\begin{proof}
Fix a pure point $x \in K_n$. Our primary goal is to find a pure dilation-maximal representing map $\mu$ for $x$ on $\rC(K)$.
Let $L$ denote the nc state space of $\rC(K)$. 
Say that a face $F$ of $L$ is \emph{hereditary} if $\mu\in F'$ and $\mu\prec_d \nu$ implies that $\nu\in F'$.
Let $F = \{ \mu \in L_n : \mu \text{ has barycenter } x \}$. 
Then $F$ is a closed face since if $(1/2)(\nu_1 + \nu_2) \in F$, then by the pureness of $x$, both $\nu_1$ and $\nu_2$ have barycenter $x$.
Furthermore, $F$ is hereditary: if $\mu \in F$ and $\mu \prec_d \nu$, then $\nu \in F$ because the barycenters of $\mu$ and $\nu$ agree by Lemma~\ref{L:same-barycenter}.

Apply Zorn's lemma to the family of all closed hereditary faces contained in $F$ to get a minimal closed hereditary face $F_0$. 
We claim that $F_0$ is a single point.
Suppose otherwise that there are $\mu,\nu \in F_0$ with $\mu \ne \nu$. 
By Proposition~\ref{P:agree-on-convex}, there is a convex nc function $f \in \rC(K)$ such that $\mu(f) \ne \nu(f)$.
The set $\{ \mu(f) : \mu\in F_0 \}$ is a weak-$*$ compact convex subset of $(\cM_n)_{sa}$.
Therefore there is a maximal element $A$ of this set in the usual order on self-adjoint matrices or operators.
Let $F_1 = \{ \mu \in F_0 : \mu(f) = A \}$. The maximality of $A$ shows that this is a proper closed face of $F_0$.
Moreover since $f$ is convex, this set is hereditary. This contradicts the minimality of $F_0$ as a closed hereditary face.
Thus $F_0 = \{\mu_0\}$ is a singleton. Thus $\mu_0$ is an extreme point of $F$. Since $F_0$ is  a face, $\mu_0$ is pure.
Since $F_0$ is hereditary, $\mu_0$ must be maximal in the dilation order.

Let $(y,\alpha)$ be a minimal representation of $\mu_0$.
By \cite{Arv1969}*{Corollary~I.4.3}, $\delta_y$ is irreducible.
By Theorem~\ref{thm:equivalence-dilation-maximal-unique-representing-map}, $y$ has a unique representing map. 
Thus $y$ is an nc extreme point of $K$ by Theorem~\ref{thm:extreme}. 
By Corollary~\ref{cor:extreme}, $\delta_y$ is a boundary representation.
\end{proof}

The proof of Theorem~\ref{thm:extreme} implies the next result.

\begin{cor}\label{cor:pure-dilation-maximal-extreme-rep}
If $\mu : \rC(K) \to \cM_n$ is a pure dilation maximal unital completely positive map with minimal representation $(y,\alpha)$ for $y \in K_p$ and an isometry $\alpha \in \cM_{p,n}$, then $y$ is an nc extreme point of $K$.
\end{cor}

\subsection{Scalar convex envelope} \label{sec:scalar-convex-envelope}

In this section we will show that the maximality of a unital completely positive map on the C*-algebra of continuous nc functions can be detected by looking at scalar-valued functions. As usual we will let $L$ denote the nc state space of $\rC(K)$. Thus $L_1$ denotes the space of (scalar) states of $\rC(K)$. Recall that $\partial L_1$ denotes the scalar extreme points of the convex set $L_1$.

The next two results follow immediately from the definition of the nc Choquet order.

\begin{lem} \label{lem:convexity-dilation-order}
Let $K$ be a compact nc convex set and let $\mu_1,\ldots,\mu_k$ and $\nu_1,\ldots,\nu_k$ be states 
on $\rC(K)$ such that $\mu_i \prec_c \nu_i$ for each $i$. Then for scalars $\lambda_1,\ldots,\lambda_k \geq 0$ with $\lambda_1 + \cdots + \lambda_k = 1$, $\sum \lambda_i \mu_i \prec_c \sum \lambda_i \nu_i$.
\end{lem}

\begin{lem} \label{lem:continuity-dilation-order}
Let $K$ be a compact nc convex set and let $\{\mu_i\}$ and $\{\nu_i\}$ be nets of states on $\rC(K)$ converging in the point-weak* topology to states $\mu$ and $\nu$ on $\rC(K)$ respectively. If $\mu_i \prec_c \nu_i$ for each $i$, then $\mu \prec_c \nu$.
\end{lem}

\begin{defn} \label{defn:scalar-valued-functions}
Let $K$ be a compact nc convex set and let $L_1$ denote the state space of $\rC(K)$. For $f \in \rC(K)$ with convex envelope $\env{f}$, let $\hat{f}$ and $\breve{f}$ denote the scalar-valued functions on $L_1$ defined by
\[
\hat{f}(\mu) = \mu(f) \qand \breve{f}(\mu) = \inf \mu(\env{f}) \qfor \mu \in L_1.
\]
\end{defn}

\begin{rem}
Kadison's representation theorem \cite{Kad1951} implies that as a function system, $\rC(K)$ is order isomorphic to the function system $\rA(L_1)$ of continuous affine functions on $L_1$. For $f \in \rC(K)$, the function $\hat{f}$ is precisely the image of $f$ under the corresponding order isomorphism.
\end{rem}

\begin{prop} \label{prop:scalar-convex-envelope}
Let $K$ be a compact nc convex set and let $L_1$ denote the state space of $\rC(K)$. Let $f \in \rC(K)$ be a self-adjoint continuous nc function. Then the corresponding functions $\hat{f},\breve{f} : L_1 \to \bC$ satisfy
\begin{enumerate} [label=\normalfont{(\arabic*)}]
\item $\breve f  \leq \hat{f}$,
\item $\breve f $ is lower semicontinuous,
\item $\breve f $ is convex.
\end{enumerate}
\end{prop}

\begin{proof}
(1)
This follows immediately from Proposition \ref{prop:properties-convex-envelope}.

(2)
Let $\{\mu_i\}$ be a net in $L_1$ converging to $\mu \in L_1$. We must show that $\lim \inf \breve f (\mu_i) \geq \breve f(\mu)$. For $\epsilon > 0$, Theorem \ref{thm:convex-env-ucp-dilation-order} implies that for each $i$ there is $\nu_i \in L_1$ such that $\mu_i \prec_d \nu_i$ and $\nu_i(f) < \inf \mu_i(\env{f}) + \epsilon$. If $\{\nu_j\}$ is a subnet converging in the weak* topology to $\nu \in L_1$, then Lemma \ref{lem:continuity-dilation-order} implies $\mu \prec_c \nu$. 
So by Theorem~\ref{thm:equivalence-orders}, $\mu \prec_d \nu$.  Thus applying Theorem \ref{thm:convex-env-ucp-dilation-order} again implies $[\nu(f),+\infty) \subseteq \mu(\env{f})$. Hence
\begin{align*}
\breve f (\mu) &= \inf \mu(\env{f}) \leq \nu(f) = \lim \nu_j(f) \\&
\leq \liminf (\inf \mu_j(\env{f})) + \epsilon = \liminf \breve f (\mu_j) + \epsilon.
\end{align*}
Taking $\epsilon \to 0$ gives the desired result.

(3)
For $\mu_1,\mu_2 \in L_1$ and $t \in (0,1)$, Theorem \ref{thm:convex-env-ucp-dilation-order} implies 
\begin{align*}
t \mu_1(\env{f}) \!+\! (1 \!-\! t) \mu_2(\env{f}) &= t \! \bigcup_{\mu_1 \prec_d \nu_1} [\nu_1(f),+\infty) + (1 \!-\! t) \! \bigcup_{\mu_2 \prec_d \nu_2} [\nu_2(f),+\infty) \\
&= \bigcup_{\substack{\mu_1 \prec_d \nu_1 \\ \mu_2 \prec_d \nu_2}} [t\nu_1(f) + (1-t)\nu_2(f), +\infty),
\end{align*}
where the unions are taken over all states $\nu_1,\nu_2 \in L_1$ with $\mu_1 \prec_d \nu_1$ and $\mu_2 \prec_d \nu_2$. 
For such $\nu_1,\nu_2$, Lemma \ref{lem:convexity-dilation-order} and Theorem~\ref{thm:equivalence-orders} imply that 
\[ t\mu_1 + (1-t)\mu_2 \prec_d t \nu_1 + (1-t) \nu_2 .\]
Therefore
\[
t  \mu_1(\env{f}) + (1-t)  \mu_2(\env{f}) \subseteq \bigcup_\lambda [\lambda(f), +\infty) = (t\mu_1 + (1-t)\mu_2)(\env{f}),
\]
where the union is taken over all $\lambda \in L_1$ with $t \mu_1 + (1-t) \mu_2 \prec_d \lambda$.  Hence
\begin{align*}
t \breve f (\mu_1) &+ (1-t) \breve f (\mu_2) = t \inf \mu_1(\env{f}) + (1-t) \inf \mu_2(\env{f}) \\&
\geq \inf (t \mu_1 + (1-t)\mu_2)(\env{f}) = \breve f (t \mu_1 + (1-t) \mu_2),
\end{align*}
and we conclude that $\breve f$ is convex.
\end{proof}

We obtain the following characterization of maximal elements in $K$ in terms of the scalar-valued functions in Definition \ref{defn:scalar-valued-functions}.

\begin{prop} \label{prop:scalar-maximal}
Let $K$ be a compact nc convex set and let $L_1$ denote the state space of $\rC(K)$. A state $\mu \in L_1$ is maximal in the dilation order if and only if $\hat{f}(\mu) = \breve{f}(\mu)$ for all self-adjoint $f \in \rC(K)$. 
\end{prop}

\begin{proof}
For self-adjoint $f \in \rC(K)$, Proposition \ref{prop:properties-convex-envelope} implies that $\mu(\env{f}) \supseteq [\mu(f),+\infty)$. Hence if $\inf \mu(\env{f}) = \breve{f}(\mu) = \env{f}(\mu) = \mu(f)$, then $\mu(\env{f}) = [\mu(f),+\infty)$. The result now follows from Corollary \ref{cor:criterion-maximality-dilation-order}.
\end{proof}

\section{Noncommutative Choquet-Bishop-de Leeuw theorem} \label{sec:nc-bishop-de-leeuw}

\subsection{Classical Choquet-Bishop-de Leeuw theorem}

The classical Choquet-Bishop-de Leeuw theorem asserts that for a compact convex set $C$, every point $x \in C$ can be represented by a probability measure $\mu$ supported on the extreme boundary $\partial C$ of $C$. The result was proved for metrizable $C$ by Choquet \cite{Ch1956}, and for non-metrizable $C$ by Bishop and de Leeuw \cite{BdL1959} (see \cite{Alfsen}*{Section~I.4}).

The set $C$ is metrizable if and only if the corresponding function system $\rA(C)$ is separable. In this case, $\partial C$ is $G_\delta$, and as usual, $\mu$ is said to be supported on $\partial C$ if $\mu(C \setminus \partial C) = 0$. Otherwise, $\partial C$ is not necessarily even Borel, and in this case $\mu$ is said to be supported on $\partial C$ if $\mu(X) = 0$ for every {\em Baire} set $X \subseteq C$ that is disjoint from $\partial C$. Equivalently, $\int_C f\, d\mu = 0$ for every bounded Baire function $f$ on $C$ with support in $C \setminus \partial C$.

\subsection{Noncommutative Choquet-Bishop-de Leeuw theorem}

In this section, we will establish a noncommutative generalization of the Choquet-Bishop-de Leeuw theorem. This result will not require any assumptions about separability. However, as in the classical theory, technical difficulties arise in the non-separable setting. In order to handle these difficulties, and in order to define an appropriate notion of support for a representing map, we will require an appropriate notion of bounded Baire nc function. Before stating the definition, we first recall the definition of the Baire-Pedersen envelope of a C*-algebra, introduced by Pedersen under a different name (see \cite{Pedersen}*{Section~4.5}).

Let $A$ be a C*-algebra.  The {\em Baire-Pedersen envelope} $\fB(A)$ of $A$ is a C*-subalgebra of the bidual $A^{**}$ that contains $A$. It is constructed as the monotone sequential closure of $A$ in its universal representation. If $A$ is commutative, say $A = \rC(X)$ for a compact Hausdorff space $X$, then $\fB(\rC(X))$ is isomorphic to the C*-algebra of bounded Baire functions on $X$.

\begin{defn}
For a compact nc convex set $K$, we let $\rBI(K)$ denote the Baire-Pedersen envelope $\fB(\rC(K))$ of $\rC(K)$ and refer to the elements in $\rBI(K)$ as the {\em bounded Baire nc functions} on $K$. We say that a unital completely positive map $\mu : \rC(K) \to \cM_n$ is {\em supported} on the extreme boundary $\partial K$ if $\mu(f) = 0$ for every bounded Baire nc function $f \in \rBI(K)$ satisfying $f(x) = 0$ for all $x \in \partial K$.
\end{defn}

\begin{rem}
Note that since $\rBI(K) \subseteq \rC(K)^{**} = \rB(K)$, the elements in $\rBI(K)$ are bounded nc functions. For a unital completely positive map $\mu : \rC(K) \to \cM_n$, the restriction to $\rBI(K)$ of the unique normal extension of $\mu$ to $\rB(K)$ coincides with the unique sequentially normal extension of $\mu$ to $\rBI(K)$ (see \cite{Pedersen}*{Theorem 4.5.9}).
\end{rem}

\begin{prop} \label {prop:support-of maximal-states}
Suppose that every scalar state on $\rC(K)$ which is maximal in the dilation order is supported on $\partial K$.
Then every nc state on $\rC(K)$ which is maximal in the dilation order is supported on $\partial K$.
\end{prop}

\begin{proof}
Assume that $\mu \in L_p$ is an nc state on $\rC(K)$ which is maximal in the dilation order.
For every isometry $\alpha \in M_{p,1}$, the scalar state $\alpha^* \mu \alpha$ is maximal in the dilation order by 
Theorem~\ref{thm:equivalence-dilation-maximal-unique-representing-map}.
By hypothesis, $0 = (\alpha^* \mu \alpha)(f) = \alpha^* \mu(f) \alpha$.
Now $\alpha(1) = \xi$ is an arbitrary unit vector in $H_p$, so this can be restated as saying that $\ip{\mu(f) \xi,\xi} = 0$ for all $\xi \in H_p$ with $\|\xi\|=1$.
This says that the numerical radius of $\mu(f)$ is $0$; whence $\mu(f) = 0$.
\end{proof}

The next result is a noncommutative analogue of the Choquet-Bishop-de Leeuw theorem. Note that the result does not place any restrictions on $K$. In particular, $\rA(K)$ is not required to be separable.

\begin{thm}[Noncommutative Choquet-Bishop-de Leeuw theorem] \label{thm:nc-bishop-de-leeuw} 
Let $K$ be a compact nc convex set. For $x \in K_n$ there is a unital completely positive map $\mu : \rC(K) \to \cM_n$ that represents $x$ and is supported on the extreme boundary $\partial K$.
\end{thm}

\begin{rem}
Since the restriction to $\rA(K)$ of the unique surjective homomorphism from $\rC(K)$ onto $\cmin(\rA(K))$ is a unital complete order embedding, Arveson's extension theorem implies that for $x \in K_n$ there is a unital completely positive map $\mu : \rC(K) \to \cM_n$ that represents $x$ and factors through $\cmin(\rA(K))$. Hence we can always choose a representing map for $x$ that is supported on the irreducible representations of $\cmin(\rA(K))$, i.e. on the Shilov boundary of $\rA(K)$.

The discussion in Section \ref{sec:extreme-pts-minimal-c-star-alg} shows that $\partial K$ corresponds to an (often proper) subset of the irreducible representations of $\cmin(\rA(K))$. Therefore, the assertion in Theorem \ref{thm:nc-bishop-de-leeuw} is much stronger. It says that we can always choose a representing map for $x$ that is supported on $\partial K$, i.e. on the Choquet boundary of $\rA(K)$.
\end{rem}

In order to prove Theorem \ref{thm:nc-bishop-de-leeuw}, we will require some preliminary results about the separable case.

\begin{prop} \label{prop:nc-choquet-bishop-de-leeuw-sep}
Let $K$ be a compact nc convex set such that $\rA(K)$ is separable and let $L_1$ denote the state space of $\rC(K)$. Then the set
\[
Z = \{ \nu \in L_1 : \nu \text{ is pure and dilation maximal} \}
\]
is $G_\delta$. If $\mu \in L_1$ is dilation maximal, then there is a regular Borel probability measure $\rho$ on $L_1$ supported on $Z$ with barycenter $\mu$, meaning that $\rho(Z) = 1$ and
\[
\mu(f) = \int_Z \nu(f)\, d\rho(\nu) \qfor f \in \rC(K).
\]
Moreover any regular Borel probability measure $\rho$ on $L_1$ with barycenter $\mu$ is supported on $Z$ in the above sense.
\end{prop}

\begin{proof}
Since $\rA(K)$ is separable, $\rC(K)$ is also separable. 
Let $\{f_k\}$ be a dense sequence in $\rC(K)_{sa}$. For $k,m \in \bN$, let
\[
X_{km} = \{ \nu \in L_1 : (\hat{f_k} - \breve{f_k})(\nu) \geq \tfrac1m \},
\]
where $\hat{f}$ and $\breve{f}$ are defined as in Section \ref{sec:scalar-convex-envelope}. By Proposition~\ref{prop:scalar-convex-envelope}, 
$\breve{f}$ is convex and lower semicontinuous. Since $\hat f$ is continuous and affine, $\hat{f} - \breve{f}$ is concave and upper semicontinuous. Therefore $X_{km}$ is closed. 

By Proposition~\ref{prop:scalar-maximal}, the $G_\delta$ set
\[
Y = L_1 \setminus \big( \bigcup_{k,m} X_{km} \big)
\]
is precisely the set of dilation maximal states on $\rC(K)$. Since $\rC(K)$ is separable, $L_1$ is metrizable. 
So $\partial L_1$ is $G_\delta$  (see e.g. \cite{Alfsen}*{Corollary I.4.4}). It follows that $Z = \partial L_1 \cap Y$ is $G_\delta$. 

By Choquet's integral representation theorem, there is a Borel measure $\rho$ on $L_1$ supported on $\partial L_1$ that represents $\mu$, 
i.e. such that
\[
\mu(f) = \int_{\partial C} \nu(f)\, d\rho(\nu) \qfor f \in \rC(K).
\]
It remains to show that $\rho$ is supported on $Z$, or equivalently that $\rho(X_{km}) = 0$ for $k,m \in \bN$.

Suppose for the sake of contradiction that $\rho(X_{km}) > 0$ for some $k,m \in \bN$. Define probability measures $\tau$ and $\eta$ on $L_1$ by 
\[
\sigma = \rho(X_{km})^{-1} \rho|_{X_{km}} \qand \tau = \rho(L_1 \setminus X_{km})^{-1} \rho|_{L_1 \setminus X_{km}}.
\]
Let $\xi \in L_1$ denote the barycenter of $\sigma$ and let $\eta \in L_1$ denote the barycenter of $\tau$. Note that $\mu = \rho(X_{km}) \xi + \rho(L_1 \setminus X_{km}) \eta$. 

Since $\sigma$ is supported on $X_{km}$, there is a sequence $\{\sigma_i\}$ of finitely supported probability measures on $X_{km}$ such that 
$\lim \sigma_i = \sigma$ in the weak* topology. Each $\sigma_i$ can be written as a finite convex combination $\sigma_i = \sum c_{ij} \delta_{\nu_{ij}}$ of states $\nu_{ij} \in X_{km}$. Let $\xi_i \in L_1$ denote the barycenter of $\sigma_i$. Then by the continuity of the barycenter map, $\lim \xi_i = \xi$ in the weak* topology. Hence
\begin{align*}
\xi(f_k) &= \lim_i \xi_i(f_k)
= \lim_i \sum_j c_{ij} \nu_{ij}(f_k) \\&
\geq \frac{1}{m} + \liminf_i \sum_j c_{ij} \breve{f_k}(\nu_{ij}) \\&
\geq \frac{1}{m} + \liminf_i \breve{f_k}(\xi_i) \\&
\geq \frac{1}{m} + \breve{f_k}(\xi) ,
\end{align*}
where we have used the convexity and lower semicontinuity of $\breve{f_k}$ from Proposition~\ref{prop:scalar-convex-envelope}. Another application of the convexity of $\breve{f_k}$ yields
\begin{align*}
\mu(f_k) &= \rho(X_{km}) \xi(f_k) +\rho(L_1 \setminus X_{km}) \eta(f_k) \\
&\geq  \frac{\rho(X_{km})}{m} + \rho(X_{km}) \breve{f_k}(\xi) + \rho(L_1 \setminus X_{km}) \breve{f_k}(\eta) \\
&\geq \frac{\rho(X_{km})}{m} + \breve{f_k}(\mu) .
\end{align*}
In particular, $\hat{f_k}(\mu) \ne \breve{f_k}(\mu)$. Therefore, by Proposition \ref{prop:scalar-maximal}, $\mu$ is not maximal in the dilation order, providing a contradiction.
\end{proof}

The next result will provide the connection between the separable and the non-separable case.
Let $S \subseteq \rA(K)$ be a separable operator system, and let $K_0$ be the state space of $S$. 
By the Hahn-Banach theorem, the restriction map $q:K \to K_0$ is surjective.
We identify $S$ with $\rA(K_0)$ and observe that $g \in \rB(K_0)$ can be identified with $g\circ q \in \rB(K)$,
and conversely, every function of the form $g\circ q \in \rB(K)$ arises in this way.
So we can identify $\rB(K_0)$ with this subalgebra of $\rB(K)$, and identify $\rC(K_0)$ with the subalgebra of $\rC(K)$ generated by $S$.

\begin{prop} \label{prop:extension-dilation-maximal-subalg}
Let $K$ be a compact nc convex set and let $S \subseteq \rA(K)$ be an operator system with nc state space $K_0$. 
Identify $S$ with $\rA(K_0)$ and identify $\rC(K_0)$ with the C*-subalgebra of $\rC(K)$ generated by $S$. 
Then every dilation maximal pure state on $\rC(K_0)$ extends to a dilation maximal pure state on $\rC(K)$.
\end{prop}

\begin{proof}
Let $\mu_0$ be a dilation maximal pure state on $C(K_0)$ and let $(x_0,\alpha_0)$ be a minimal representation of $\mu_0$ for $x_0 \in (K_0)_n$ and an isometry $\alpha_0 \in \cM_{n,1}$. Note that by Theorem~\ref{thm:equivalence-dilation-maximal-unique-representing-map} and Theorem \ref{thm:extreme}, $x_0$ is an extreme point of $K_0$.

Let $F = \{x \in K_n : x|_{\rA(K_0)} = x_0 \}$. Then $F$ is a closed face of $K_n$. By the (classical) Krein-Milman theorem, the set of extreme points of $F$ is non-empty. Let $x \in F$ be an extreme point. Then $x$ is pure in $K$. By Theorem~\ref{thm:dilation-pure-to-extreme}, we can dilate $x$ to an extreme point $y \in K_p$. In particular, the representation $\delta_y$ is irreducible and dilation maximal. Let $\beta \in \cM_{p,n}$ be an isometry such that $x = \beta^* y \beta$. 

Define a state $\mu$ on $\rC(K)$ by $\mu = \alpha_0^* \beta^* \delta_y \beta \alpha_0$. Then $(y,\beta \alpha_0)$ is a representation of $\mu$. Since $\delta_y$ is irreducible, $(y,\beta \alpha_0)$ is minimal. As $y$ is an extreme point, it follows that $\mu$ is pure and maximal in the dilation order. Furthermore, since $x_0$ is an extreme point in $K_0$, the restriction $y|_{\rA(K_0)}$ must be a trivial dilation of $x_0$. Hence $\mu|_{\rC(K_0)} = \mu_0$. 
\end{proof}

\begin{prop} \label{prop:annihilate-boundary}
Let $K$ be a compact nc convex set. Every dilation maximal state on $\rC(K)$ is supported on the extreme boundary $\partial K$.
\end{prop}

\begin{proof}
Let $\mu$ be a dilation maximal state on $\rC(K)$ and fix $f \in \rBI(K)$ such that $f(x) = 0$ for $x \in \partial K$. We must show that $\mu(f) = 0$. 
By Proposition~\ref{prop:support-of maximal-states}, it suffices to establish this for dilation maximal scalar states $\mu \in L_1$.

By \cite{Pedersen}*{Lemma 4.5.3}, there is a separable operator system $S \subseteq \rA(K)$ such that $f$ lies in the sequential monotone completion of $S$.
We can identify $\rBI(K_0)$ with a subalgebra of $\rBI(K)$.
Then $f \in \rBI(K_0) \subseteq \rBI(K)$. 
The key point is that $f$ is a Baire nc function on the nc state space of a separable operator subsystem.

Next we observe that $\mu_0 = \mu|_{\rC(K_0)}$ is a dilation maximal state.
Let $g \in \rC(K_0)_{sa}$, identified with $f = g \circ q$. 
It is easy to see that $\bar f = \bar g \circ q$. Therefore, since $\mu$ is maximal,
\[
 \mu_0( \bar g) = \mu (\bar f) = \mu(f) = \mu_0(g) .
\]
By Proposition~\ref{prop:scalar-maximal}, $\mu_0$ is dilation maximal.

Let $(L_0)_1$ denote the (scalar) state space of $\rC(K_0)$. Then by Proposition \ref{prop:nc-choquet-bishop-de-leeuw-sep}, the set
\[
 Z = \{ \nu \in (L_0)_1 : \nu \text{ is pure and dilation maximal} \}
\]
is $G_\delta$, and there is a regular Borel probability measure $\rho$ on $(L_0)_1$ supported on $Z$ such that
\begin{align}
 \mu(g) = \mu_0(g) = \int_Z \nu(g) \, d\rho(\nu), \qfor g \in \rC(K_0) \label{eqn:prop:annihilate-boundary-1}
\end{align}

Since $f(x) = 0$ for $x \in \partial K$, Proposition \ref{prop:extension-dilation-maximal-subalg} implies that $f(x) = 0$ for $x \in \partial K_0$. 
Hence by Corollary \ref{cor:pure-dilation-maximal-extreme-rep}, $\nu(f) = 0$ for every $\nu \in Z$.

By \cite{Pedersen}*{Corollary 4.5.13}, $f$ is universally measurable, so the barycenter formula (\ref{eqn:prop:annihilate-boundary-1}) also holds for $f$ (see e.g. \cite{Bro2008}*{Section 5}).  Hence $\mu(f) = 0$. 
\end{proof}

Putting all of these ingredients together yields a proof of our noncommutative Bishop-de Leeuw Theorem.

\begin{proof}[Proof of Theorem $\ref{thm:nc-bishop-de-leeuw}$]
Let $y \in K_p$ be a maximal dilation of $x$ and let $\alpha \in \cM_{p,n}$ be an isometry such that $x = \alpha^* y \alpha$. 
Then the corresponding representation $\delta_y$ is dilation maximal, and hence by Proposition \ref{prop:annihilate-boundary}, 
$\delta_y$ is supported on $\partial K$. 
Define a unital completely positive map $\mu : \rC(K) \to \cM_n$ by $\mu = \alpha^* \delta_y \alpha$. 
Then $\mu$ represents $x$ and is supported on $\partial K$. 
\end{proof}

\section{Noncommutative integral representation theorem} \label{sec:integral-representations}

\subsection{Motivation}

In this section we will restrict our attention to the separable setting and prove a noncommutative analogue of Choquet's integral representation theorem using the results from Section \ref{sec:nc-bishop-de-leeuw}. We suspect that these ideas may work in greater generality, however we will utilize results about direct integral decompositions of representations of separable C*-algebras.

Let $K$ be a compact nc convex set. For $x \in K_n$, the corresponding representation $\delta_x : \rC(K) \to \cM_n$ should be viewed as a noncommutative Dirac measure on $K$ supported at the point $x$. More generally, for a finite set of points $\{x_i \in K_{n_i}\}$ and operators $\alpha_i \in \cM_{n_i,n}$ satisfying $\sum \alpha_i^* \alpha_i = 1_n$, the unital completely positive map $\mu : \rC(K) \to \cM_n$ defined by 
\[
\mu = \sum \alpha_i^* \delta_{x_i} \alpha_i
\]
should be viewed as a finitely supported nc probability measure on $K$.

For a continuous nc function $f \in \rC(K)$, the expression
\[
\mu(f) = \sum_{i=1}^k \alpha_i^* f(x_i) \alpha_i,
\]
should be viewed as the integral of $f$ against the nc measure $\mu$. Note that if $x \in K_n$ denotes the barycenter of $\mu$, then in particular, for a continuous affine nc function $a \in \rA(K)$,
\[
a(x) = \mu(a) = \sum_{i=1}^k \alpha_i^* a(x_i) \alpha_i.
\]

In the next section we will consider a basic theory of noncommutative integration.

\subsection{Noncommutative integration} \label{sec:nc-integration}
In this section we will outline a basic theory of integration against measures taking values in spaces of completely positive maps following an approach originally due to Fujimoto \cite{Fuj1994}.

Let $M$ and $N$ be separable von Neumann algebras and let $\ncpmaps(M,N)$ denote the space of all normal completely positive maps from $M$ to $N$. Let $Z$ be a topological space and let $\bor(Z)$ denote the $\sigma$-algebra of Borel subsets of $Z$.

A $\ncpmaps(M,N)$-valued Borel measure on $Z$ is a countably additive map $\lambda: \bor(Z) \to \ncpmaps(M,N)$, meaning that if $(E_k)$ is a disjoint sequence in $\bor(Z)$ and $E = \bigcup_{k=1}^\infty E_k$, then $\lambda(E) = \sum_{k=1}^\infty \lambda(E_k)$, where the right hand side converges with respect to the point-weak* topology.

Following Fujimoto \cite{Fuj1994}*{Definition 3.6}, we will require that $\lambda$ satisfies an absolute continuity-type condition with respect to a scalar-valued Borel measure.  

Let $\nu$ be a scalar-valued Borel measure on $Z$. For each $a \in M$ and $\rho \in N_*$, we obtain a scalar-valued Borel measure $\lambda_{a,\rho}$ on $Z$ defined by
\[
\lambda_{a,\rho}(E) = \langle \lambda(E)(a), \rho \rangle, \qfor E \in \bor(Z).
\]
We say that $\lambda$ is absolutely continuous with respect to $\nu$ if each $\lambda_{a,\rho}$ is absolutely continuous with respect to $\nu$ for each $\rho \in N_*$. In this case, the Radon-Nikodym theorem implies that for each $\lambda_{a,\rho}$, there is unique $r_{a,\rho} \in L^1(Z,\nu)$ satisfying
\[
\lambda_{a,\rho}(E) = \langle \lambda(E)(a), \rho \rangle = \int_E r_{a,\rho}(z) \, d\nu(z), \qfor E \in \bor(Z).
\]

We say that a function $f : Z \to M$ is {\em bounded} if 
\[ \sup \{\|f(z)\| : z \in Z\} < \infty .\]
We say that $f$ is {\em measurable} if its range $f(Z)$ is separable and $f^{-1}(E)$ is Borel for every weak*-Borel set $E \subseteq M$. We will let $\rBM(Z,M)$ denote the space of all bounded measurable functions from $Z$ to $M$. Since $M$ is separable, the relative weak* topology on bounded subsets of $M$ is metrizable. We equip bounded subsets of $\rBM(Z,M)$ with the corresponding topology of uniform convergence with respect to this metric, which we refer to as the {\em weak*-uniform convergence topology}.

We say that $f \in \rBM(Z,M)$ is {\em simple} if there are sequences $(E_k)_{k=1}^\infty$ in $\bor(Z)$ and $(a_k)_{k=1}^\infty$ in $M$ such that
\[
f = \sum_{k=1}^\infty \chi_{E_k} a_k.
\]
Fujimoto showed \cite{Fuj1994}*{Lemma 3.3} that the space of bounded measurable simple functions is dense in $\rBM(Z,M)$ with respect to the weak*-uniform convergence topology.

For a $\ncpmaps(M,N)$-valued Borel measure $\lambda$ on $Z$ and a simple function $f \in \rBM(Z,M)$ expressed as above, the integral of $f$ with respect to $\lambda$ is defined by
\[
\int_X f \, d\lambda = \sum_{i=1}^\infty \lambda(E_i)(a_i),
\]
where the right hand side converges in the weak* topology on $N$ (see \cite{Fuj1994}*{Lemma 3.1}). As usual, this definition does not depend on any particular expression of $f$.

Viewing integration against $\lambda$ as a linear map on the space of bounded measurable simple functions, Fujimoto showed \cite{Fuj1994}*{Definition 3.8} that if $\lambda$ is absolutely continuous with respect to a scalar measure $\nu$ on $Z$, then there is a unique linear extension to $\rBM(Z,M)$ that is continuous with respect to the weak*-uniform convergence topology. For $f$ in $\rB(Z,M)$, we will let
\[
\int_Z f \, d\lambda.
\]
denote the value of this extension at $f$, and refer to it as the {\em integral of $f$ against $\lambda$}. We will say that $f$ is {\em $\lambda$-integrable}.  For $\rho \in N_*$, it follows from above that
\[
\Big \langle \int_Z f \, d\lambda, \rho \Big \rangle = \int_Z r_{f(z),\rho}(z) \, d\nu(z).
\]

\subsection{Noncommutative integral representation theorem}

In this section we will introduce a definition of nc measure along with a corresponding notion of integration for nc functions. We will then apply these ideas to establish our noncommutative integral representation theorem.

\begin{lem}\label{L:Borel-extreme}
Let $K$ be a compact nc convex set such that $\rA(K)$ is metrizable. Then for each $n$, $(\partial K)_n := \partial K \cap K_n$ is a Borel set. 
\end{lem}

\begin{proof}
In \cite{Arv2008}*{Theorem 2.5}, Arveson shows that $x \in K_n$ is maximal if and only if for every $a\in \rA(K)$ and unit vector $\xi \in H_n$,
for any dilation $y$ of $x$ on a larger space, one has $\|y(a)\xi\| = \|x(a) \xi\|$. And as noted in \cite{DK2015}, it suffices to consider dilations $y$ on a Hilbert space $H_n'=H_n \oplus\bC$. Let $L_n$ denote the space of unital completely positive maps from $\rA(K)$ into $H_n'$. The compression map from $H_n'$ onto $H_n$ determines a surjective continuous map $\rho:L_n \to K_n$. Note that $y\in L_n$ is a dilation of $x$ precisely when $\rho(y)=x$.

Fix a countable dense subset $\{a_i\} \subset \rA(K)$ and a countable dense subset $\{\xi_j\}$ of the unit sphere of $H_n$.
Observe that
\[ F_{ij}(x) = \|x(a_i) \xi_j\| = \sup_k |\ip{ x(a_i) \xi_j, \xi_k}| \qfor x\in K_n \]
is the supremum of continuous functions and thus is lower semicontinuous, and in particular is Borel.
Consider the function
\[ G_{ij}(x) = \sup_{y \in \rho^{-1}(x)} \|y(a_i) \xi_j \| = \sup_k \sup_{y \in \rho^{-1}(x)} |\ip{y(a_i) \xi_j, \eta_k}| , \]
where $\{\eta_k\}$ is a dense subset of the unit sphere of $H_n'$. This is the supremum of the functions
\[ g_{ijk}(x) = \sup_{y \in \rho^{-1}(x)} |\ip{y(a_i) \xi_j, \eta_k}| .\]
This function is upper semicontinuous because if $x_m$ converges to $x$, pick $y_m \in \rho^{-1}(x_m)$ attaining the value $g_{ijk}(x_m)$.
Dropping to a subsequence, we may suppose that $|\ip{y_m(a_i) \xi_j, \eta_k}|$ approaches $\limsup_m g_{ijk}(x_m)$ and the $y_m$ converge to $y\in \rho^{-1}(x)$. 
Thus 
\[ g_{ijk}(x) \ge |\ip{y(a_i) \xi_j, \eta_k}| = \limsup_m g_{ijk}(x_m) .\]
Hence $g_{ijk}$ is Borel and so $G_{ij}$ is also Borel.

It follows that $H_{ij}(x) = G_{ij}(x) - F_{ij}(x)$ is Borel.
By the discussion in the first paragraph, $x$ is maximal if and only if $x \in \bigcap_{i,j} H_{ij}^{-1}(\{0\})$.
Therefore the set of  maximal points in $K_n$ is Borel.
Since $K_n$ is metrizable, the set of pure points $\partial K_n$ is $G_\delta$, and hence Borel.
Since $(\partial K)_n$ is the intersection of $\partial K_n$ and the set of maximal elements by Proposition~\ref{prop:extreme-iff-pure-maximal}, it is Borel.
\end{proof}

\begin{defn} \label{defn:nc-probability-measure}
Let $K$ be a compact nc convex set such that $\rA(K)$ is separable. For  $n \le \aleph_0$, a {\em $\cM_n$-valued finite nc measure} on $K$ is a sequence 
$\lambda = (\lambda_m)_{m \le \aleph_0}$ such that each $\lambda_m$ is a $\ncpmaps(\cM_m,\cM_n)$-valued Borel measure and the sum
\[
\sum_{m \leq \aleph_0} \lambda_m(K_m)(1_m) \in \cM_n
\]
is weak* convergent. For $E \in \bor(K)$, we define $\lambda(E)$ by
\[
\lambda(E) = \sum_{m \leq \aleph_0} \lambda_m(E_m). 
\]
We will say that $\lambda$ is {\em supported on the extreme boundary $\partial K$} if 
\[ \lambda_m(K_m \setminus \partial K) = 0_{\cM_m,\cM_n} \qforal m\le\aleph_0 .\] 
We will say that $\lambda$ is a {\em $\cM_n$-valued nc probability measure} on $K$ if the above sum is equal to $1_n$. Finally, we will say that $\lambda$ is {\em admissible} if each $\lambda_m$ is absolutely continuous with respect to a scalar-valued measure on $K_m$.
\end{defn}

\begin{rem}
Let $\lambda$ be an admissible $\cM_n$-valued finite nc measure on $K$ as above. For a bounded Baire nc function $f \in \rBI(K)$ and $m \le \aleph_0$, the restriction $f|_{K_m} : K_m \to \cM_m$ is a bounded and measurable $\cM_m$-valued function on $K_m$, i.e. $f|_{K_m} \in \rBM(K_m,\cM_m)$. Hence by the discussion in Section \ref{sec:nc-integration}, $f|_{K_m}$ is $\lambda_m$-integrable.
\end{rem}

\begin{example} \label{ex:nc-point-mass}
Let $K$ be a compact nc convex set such that $\rA(K)$ is separable. For $x \in K_n$, define a $\cM_n$-valued finite nc measure $\lambda_x = (\lambda_{x,m})_{m \leq \aleph_0} $ on $K$ by letting $\lambda_{x,m} = 0_{\cM_m,\cM_n}$ for $m \ne n$ and 
\[
\lambda_{x,n}(E) = \begin{cases} \id_n & x \in E,\\ 0_n & x \notin E \end{cases}
\]
for $E \subseteq \bor(K_n)$. Then $\lambda_x$ is the noncommutative analogue of a point mass. Note that $\lambda_x$ is absolutely continuous with respect to the scalar-valued point mass $\delta_x$ on $K_n$. Hence $\lambda_x$ is admissible.
\end{example}

\begin{example}
Let $K$ be a compact nc convex set such that $\rA(K)$ is separable and let $\lambda = (\lambda_m)_{m \leq \aleph_0}$ be a $\cM_n$-valued finite nc measure on $K$. For $\phi \in \ncpmaps(\cM_n,\cM_p)$, the composition $\phi \circ \lambda := (\phi \circ \lambda_m)_{m \leq \aleph_0}$ is a $\cM_p$-valued finite nc measure on $K$. If $\lambda$ is a nc probability measure, and $\phi$ is unital, then $\phi \circ \lambda$ is a nc probability measure. In this setting, scalars are replaced by normal completely positive maps, so $\phi \circ \lambda$ is the noncommutative analogue of a scaling of $\lambda$.

If $\lambda_m$ is absolutely continuous with respect to a scalar-valued probability measure $\nu_m$ on $K_m$, then $\phi \circ \lambda_m$ is also absolutely continuous with respect to $\nu_m$. Hence if $\lambda$ is admissible, then so is $\phi \circ \lambda$.

In particular, for $\alpha \in \cM_{n,p}$, $\alpha^* \lambda \alpha := (\alpha^* \lambda_m \alpha)_{m \leq \aleph_0}$ is a $\cM_p$-valued finite nc measure on $K$. If $\alpha^* \alpha = 1_p$, then $\alpha^* \lambda \alpha$ is an nc probability measure. More generally, this shows that the set of (admissible) finite nc measures on $K$ and the set of (admissible) finite nc probability measures on $K$ each form an nc convex set.
\end{example}

\begin{defn}
Let $K$ be a compact nc convex set such that $\rA(K)$ is separable and let $\lambda$ be an admissible $\cM_n$-valued finite nc measure on $K$. 
For $f \in \rBI(K)$, we define the {\em integral of $f$ with respect to $\lambda$} by
\[
\int_K f \, d\lambda := \sum_{m \le \aleph_0} \int_{K_m} f \, d\lambda_m.
\]
We say that $\lambda$ {\em represents} $x \in K_n$ if $\int_K a\, d\lambda = a(x)$ for all $a \in \rA(K)$. 
\end{defn}

\begin{rem}
For an admissible nc probability measure $\lambda$ as above, it follows from the discussion in Section \ref{sec:nc-integration} that map $\mu: \rC(K) \to \cM_n$ defined by $\mu(f) = \int_K f\, d\lambda$ is unital and completely positive. 
\end{rem}

\begin{example}
Let $K$ be a compact nc convex set. Fix a finite set of points $\{x_i \in K_{n_i}\}$ and a corresponding finite family $\{\alpha_i \in \cM_{n_i,n}\}$ satisfying $\sum \alpha_i^* \alpha_i = 1_n$. For each $i$, let $\lambda_{x_i} = (\lambda_{x_i,m})_{m \leq \aleph_0}$ denote the nc probability measure corresponding to $x_i$ as in Example \ref{ex:nc-point-mass}. Define $\lambda = (\lambda_m)_{m \leq \aleph_0}$ by $\lambda = \sum \alpha_i^* \lambda_{x_i} \alpha_i$. Then $\lambda$  is an admissible finite nc probability measure on $K$. For a function $f \in \rBI(K)$, the integral of $f$ with respect to $\lambda$ is
\[
\int_K f\, d\lambda = \sum_{m \leq \aleph_0} \int_{K_m} f\, d\lambda_m = \sum_i \alpha_i^* f(x_i) \alpha_i.
\]
\end{example}

For the proof of the next result, we will utilize the theory of direct integral decompositions of representations of separable C*-algebras as presented in Takesaki's book \cite{Tak2003}*{Sections IV.6 and IV.8}.

Let $A$ be a separable C*-algebra with state space $L_1$. 
Then $L_1$ is a compact convex subset of the dual of $\rC(L_1)$ with respect to the weak* topology. 
The extreme boundary $\partial L_1$ is precisely the set of pure states of $A$. 
For a state $\nu \in L_1$, let $\pi_\nu : A \to \B(H_\nu)$ denote the GNS representation of $\nu$.

For $\mu \in L_1$, Choquet's theorem implies there is a probability measure $\rho$ on $L_1$ supported on $\partial L_1$ with barycenter $\mu$. 
By \cite{Tak2003}*{Proposition IV.6.23}, we can in addition choose $\rho$ to be orthogonal, which is equivalent to the commutative von Neumann algebra $L^\infty(L_1,\rho)$ being isomorphic to a subalgebra of the commutant $\pi_\mu(A)'$. 
By \cite{Tak2003}*{Corollary IV.8.31}, the orthogonality of $\rho$ implies that $\pi_\mu$ is unitarily equivalent to the direct integral
\[
\pi_\mu \simeq \int_{L_1}^\oplus \pi_\nu \, d\rho(\nu) .
\]
If there is a cardinal number $n$ such that every state in the support of $\mu$ is pure and the corresponding GNS representation acts on a Hilbert space of dimension $n$, then \cite{Tak2003}*{Corollary IV.8.30} implies $\pi_\mu(A)$ is isomorphic to a subalgebra of $L^\infty(L_1,\rho,\B(H_n)) := L^\infty(L_1,\rho) \ol{\otimes} \B(H_n)$. In this case, the map $L_1 \to \Rep(A,H_n) : \nu \to \pi_\nu$ is $\rho$-measurable with respect to the point-weak* topology on $\Rep(A,H_n)$. 

\begin{prop} \label{prop:nc-measure}
Let $K$ be a compact nc convex set such that $\rA(K)$ is separable. For $x \in K$ there is a nc probability measure $\lambda$ on $K$ such that
\[
\int_\lambda f\, d\lambda = \delta_x(f), \qfor f \in \rBI(K).
\]
In particular, $\lambda$ represents $x$. Moreover, if $x$ is maximal then $\lambda$ can be chosen so that it is supported on the extreme boundary $\partial K$.
\end{prop}

\begin{proof}
We may assume that $\delta_x$ is cyclic. Let $L_1$ denote the (scalar) state space of $\rC(K)$ and let $\mu \in L_1$ be a state with GNS representation $\pi_\mu = \delta_x$. Choose a maximal orthogonal measure $\rho$ on $L_1$ with barycenter $\mu$, so that in particular $\rho$ is supported on $\partial L_1$. Then by the discussion preceding the proof, $\delta_x$ is unitarily equivalent to the direct integral
\[
\delta_x \cong \int_{L_1}^\oplus \delta_{y_\nu}\, d\rho(\nu),
\]
where for each $\nu \in L_1$, $y_\nu$ is a minimal representation for $\nu$. 

If $x$ is maximal, then by  Proposition \ref{prop:nc-choquet-bishop-de-leeuw-sep}, $\nu$ is pure and dilation maximal for $\rho$-almost every $\nu \in L_1$. In this case, Corollary \ref{cor:pure-dilation-maximal-extreme-rep} implies that $\delta_{y_\nu}$ is an extreme point for $\rho$-almost every $\nu \in L_1$.

For $m \leq \aleph_0$, let $C_m = \{\nu \in L_1 : y_\nu \in K_m \}$. Then by \cite{Dixmier}*{Lemma 1, page 139}, $C_m$ is Borel. Let $\rho_m = \rho|_{C_m}$. Then from above, $\rho_m$ is supported on $\partial L_1 \cap C_m$. Let
\[
\pi_m = \int_{C_m}^\oplus \delta_{y_\nu}\, d\rho_m(\nu).
\]
We can identify the range of $\pi_m$ with a subalgebra of $L^\infty(C_m, \rho_m, \cM_m)$ and by \cite{Sak2012}*{Theorem 1.22.13}, $L^\infty(C_m, \rho_m, \cM_m)_* = L^1(C_m, \rho_m, (\cM_m)_*)^*$.

Define $\iota_m : C_m \to K_m$ by $\iota_m(\nu) = y_\nu$. Then $\iota_m$ is the composition of the $\rho_m$-measurable map $C_m \to \Rep(\rC(K),H_m) : \nu \to \delta_{y_\nu}$ with restriction to $\rA(K)$. Since the latter map is continuous, $\iota_m$ is measurable.

Define a $\ncpmaps(\cM_m, L^\infty(C_m, \rho_m, \cM_m))$-valued Borel measure $\lambda_m$ on $K_m$ by
\[
\lambda_m(E) = \int_{\iota_m^{-1}(E)}^\oplus \id_{\cM_m} \, d\rho_m(\zeta), \quad E \in \bor(K_m).
\]
Then for $\alpha \in \cM_m$ and $\tau \in L^1(C_m, \rho_m, (\cM_m)_*)$,
\[
(\lambda_m)_{\alpha,\tau}(E) = \langle \lambda_m(E)(\alpha), \tau \rangle = \int_{\iota_m^{-1}(E)}  \langle \alpha, \tau(\nu) \rangle\, d\rho_m(\nu)
\]
for $E \in \bor(K_m)$. Hence $| \langle \lambda_m(E)(\alpha),\tau \rangle | \leq \|\alpha\|\|\tau\|$, so we can define $r_{\alpha,\tau} \in L^1(C_m,\rho_m)$ by
\[
r_{\alpha,\tau}(\nu) = \langle \alpha, \tau(\nu) \rangle, \quad \nu \in C_m.
\]
Then
\[
(\lambda_m)_{\alpha,\tau}(E) = \int_{\iota_m^{-1}(E)} r_{\alpha,\tau}(\nu)\, d\rho_m(\nu)
\]
for $E \in \bor(K_m)$. In particular, $(\lambda_m)_{\alpha,\tau}$ is absolutely continuous with respect to the scalar pushforward measure $\rho_m \circ \iota^{-1}$. Hence $\lambda_m$ is absolutely continuous with respect to $\rho_m \circ \iota_m^{-1}$. Thus by Section \ref{sec:nc-integration}, the integration map against $\lambda_m$ has a unique extension to $\rBM(K_m,\cM_m)$ that is continuous with respect to the weak*-uniform convergence topology.

For $f \in \rBI(K)$ and $\tau \in L^1(C_m, \rho_m, (\cM_m)_*)$,
\begin{align*}
\left\langle \int_{K_m} f \, d\lambda_m, \tau \right\rangle &= \int_{C_m} r_{f(y_\nu),\tau}(\nu) \, d\rho_m(\nu) \\
&= \int_{C_m} \langle f(y_\nu), \tau(\nu) \rangle \, d\rho_m(\nu) \\
&= \langle \pi_m(f), \tau \rangle.
\end{align*}
Hence
\[
\int_{K_m} f \, d\lambda_m = \pi_m(f), \quad f \in \rBI(K).
\]

Let $\lambda = (\lambda_m)_{m \leq \aleph_0}$. Then $\lambda$ is an admissible nc probability measure on $K$. Since $\delta_x \simeq \oplus \pi_m$, it follows from above that for $f \in \rBI(K)$,
\[
\int_K f\, d\lambda = \bigoplus_{m \leq \aleph_0} \int_{K_m} f\, d\lambda_m = \bigoplus_{m \leq \aleph_0} \pi_m(f) \cong \delta_x(f).
\]

If $x$ is maximal, then Proposition \ref{prop:nc-choquet-bishop-de-leeuw-sep} implies that $\nu$ is pure and dilation maximal for $\rho$-almost every $\nu \in L_1$. In this case, $\delta_{y_\nu}$ is an  extreme point for $\rho$-almost every $\nu \in L_1$.
\end{proof}

The next result can be viewed as a kind of Riesz-Markov-Kakutani representation theorem for unital completely positive maps on the C*-algebra of nc continuous functions.

\begin{thm} \label{thm:nc-riesz-markov-kakutani}
Let $K$ be a compact nc convex  set such that $\rA(K)$ is separable and let $\mu : \rC(K) \to \cM_n$ be a unital completely positive map. Then there is an admissible nc probability measure $\lambda$ on $K$ such that
\[
\mu(f) = \int_K f d\lambda, \qfor f \in \rBI(K).
\]
Moreover, if $\mu$ is dilation maximal, then $\lambda$ can be chosen so that it is supported on the extreme boundary $\partial K$.
\end{thm}

\begin{proof}
Let $(x,\alpha) \in K_p \times \cM_{p,n}$ be a minimal representation for $\mu$. Applying Proposition \ref{prop:nc-measure}, we obtain an nc probability measure $\sigma = (\sigma_m)_{m \leq \aleph_0}$ on $K$ such that $\delta_x(f) = \int_K f\, d\sigma$ for all $f \in \rBI(K)$. If $\mu$ is dilation maximal, then $x$ is maximal, in which case $\sigma$ is supported on $\partial K$.

Define a nc probability measure $\lambda = (\lambda_m)_{m \leq \aleph_0}$ by $\lambda_m = \alpha^* \sigma_m \alpha$. Then
\[
\int_{K} f \, d\lambda = \alpha^* \left( \int_{K} f\, d\sigma \right) \alpha = \alpha^* f(x) \alpha = \mu(f)
\]
for $f \in \rBI(K)$. If $\mu$ is dilation maximal, then from above $\sigma$ is supported on $\partial K$ in which case $\lambda$ is supported on $\partial K$.
\end{proof}

\begin{thm}[Noncommutative integral representation theorem] \label{thm:nc-choquet-theorem}
Let $K$ be a compact nc convex set such that $\rA(K)$ is separable. Then for $x \in K$ there is an admissible nc probability measure $\lambda$ on $K$ that represents $x$ and is supported on $\partial K$, i.e. such that
\[
a(x) = \int_K a\, d\lambda, \qfor a \in \rA(K).
\]
\end{thm}

\begin{proof}
Suppose $x \in K_n$. Let $y \in K_p$ be a maximal dilation of $x$ and let $\alpha \in \cM_{p,n}$ be an isometry such that $x = \alpha^* y \alpha$. By Proposition \ref{prop:nc-measure}, there is an admissible nc probability measure $\sigma$ on $K$ that represents $y$ and is supported on $\partial K$. Define a nc probability measure $\lambda = (\lambda_m)_{m \leq \aleph_0}$ on $K$ by $\lambda_m = \alpha^* \sigma_m \alpha$. Then $\lambda$ represents $x$ and is also supported on $\partial K$.
\end{proof}


\end{document}